\documentclass[a4paper, reqno, 11pt, notitlepage]{amsart}
\usepackage{fullpage}
\usepackage[utf8]{inputenc}

    \makeatletter
    \def\paragraph{\@startsection{paragraph}{4}%
    \z@\z@{-\fontdimen2\font}%
    {\normalfont\bfseries}}
    \makeatother

\usepackage{latexsym}
\usepackage{amsmath}
\usepackage{amssymb}
\usepackage{amsthm}
\usepackage{amscd}
\usepackage{mathrsfs}
\usepackage[all]{xy}
\usepackage{graphicx}
\usepackage{comment}
\usepackage{arydshln}
\usepackage{mathtools}

\usepackage{tikz-cd}
\usepackage{nicefrac}
\usepackage{authblk}
\usepackage{textcomp}

\usepackage{framed} \usepackage{color}  

\usepackage{hyperref} 
    \definecolor{darkblue}{rgb}{0,0,.85} 
    \definecolor{darkred}{rgb}{0.84,0,0}
    \hypersetup{colorlinks = true,
            linkcolor = darkblue,
            urlcolor  = darkblue,
            citecolor = darkred,
            anchorcolor = darkblue}
    
\usepackage{accents}

\usepackage[shortlabels]{enumitem}

\usepackage[foot]{amsaddr}

\usepackage{bm}

\usepackage{stmaryrd}

\usepackage{graphics}

\usepackage{mathtools}

\newtheorem{thm}{Theorem}[section]
\newtheorem{prop}[thm]{Proposition}
\newtheorem{thmi}{Theorem}
\newtheorem{propi}{Proposition}
\newtheorem{lem}[thm]{Lemma}
\newtheorem{cor}[thm]{Corollary}

\newtheorem*{remn}{Remark}

\makeatletter
\newtheoremstyle{examplestyle}
  {1em}
  {1em}
  {\addtolength{\@totalleftmargin}{1.0em}
   \addtolength{\linewidth}{-1.0em}
   \parshape 1 1.0em \linewidth}
  {}
  {\bfseries}
  {.}
  {.5em}
  {}

\theoremstyle{examplestyle} 

\newtheorem{eg}[thm]{Example}

\newtheorem{rem}[thm]{Remark}

\newtheorem{defn}[thm]{Definition}

\DeclareMathOperator{\Spec}{Spec}

\DeclareMathOperator{\id}{id}

\DeclareMathOperator{\End}{End}
\DeclareMathOperator{\Hom}{Hom}

\DeclareMathOperator{\Ker}{Ker}

\DeclareMathOperator{\Aut}{Aut}

\DeclareMathOperator{\Lie}{Lie}

\DeclareMathOperator{\GL}{GL}
\DeclareMathOperator{\SL}{SL}

\DeclareMathOperator{\Ad}{Ad}
\DeclareMathOperator{\ad}{ad}
\DeclareMathOperator{\Int}{Int}

\newcommand{\bC}{\mathbb{C}}

\newcommand{\bG}{\mathbb{G}}

\newcommand{\bQ}{\mathbb{Q}}

\newcommand{\bZ}{\mathbb{Z}}

\newcommand{\bb}[1]{\mathbb{#1}}

\newcommand{\cI}{\mathcal{I}}

\newcommand{\cN}{\mathcal{N}}

\newcommand{\cW}{\mathcal{W}}

\newcommand{\mc}[1]{\mathcal{#1}}

\newcommand{\fu}{\mathfrak{u}}

\newcommand{\mf}[1]{\mathfrak{#1}}

\newcommand{\ol}{\overline}

\newcommand{\wh}{\widehat}
\newcommand{\wt}{\widetilde}
\newcommand{\lra}{\longrightarrow}

\newcommand{\mb}[1]{\mathbf{#1}}
\newcommand{\Z}{\mathbb{Z}}
\newcommand{\disc}{\mathrm{disc}}
\newcommand{\Q}{\mathbb{Q}}

\newcommand{\tw}{\mathrm{tw}}
\newcommand{\C}{\mathbb{C}}
\newcommand{\univ}{\mathrm{univ}}
\newcommand{\JM}{\mathsf{JM}}
\newcommand{\Transp}{\underline{\mathrm{Transp}}}

\makeatletter
\renewcommand{\email}[2][]{%
  \ifx\emails\@empty\relax\else{\g@addto@macro\emails{,\space}}\fi%
  \@ifnotempty{#1}{\g@addto@macro\emails{\textrm{(#1)}\space}}%
  \g@addto@macro\emails{#2}%
}
\makeatother

\newcommand{\der}{\mathrm{der}}

\newcommand{\stacks}[1]{\cite[\href{https://stacks.math.columbia.edu/tag/#1}{Tag~#1}]{StacksProject}}

\newcommand{\CG}{{^C}\!G}

\newcommand{\red}{\mathrm{red}}

\renewcommand{\ss}{\mathrm{ss}}

\newcommand{\LG}{{^L}\!G}

\newcommand{\WD}{\mathrm{WD}}

\newcommand{\cat}[1]{\mathbf{#1}}

\newcommand{\h}{\mathcal{O}}

\newcommand{\WDP}{\mathsf{WDP}}
\newcommand{\LP}{\mathsf{LP}}

\DeclareMathAccent{\wtilde}{\mathord}{largesymbols}{"65}

\newcommand*\isomto{%
        \xrightarrow{\raisebox{-0.2 em}{\smash{\ensuremath{\sim}}}}%
    }

    \newcommand{\ov}[1]{\overline{#1}}
    \newcommand{\et}{\mathrm{\acute{e}t}}

   \newcommand{\defeq}{\vcentcolon=}

\title{The Jacobson--Morozov morphism for Langlands parameters in the relative setting}
\author{Alexander Bertoloni Meli$^{(1)}$}
\address[1]{\scriptsize University of Michigan, 530 Church St, Ann Arbor, MI 48109, United States}
\email[1]{\scriptsize abertolo@umich.edu}
\author{Naoki Imai$^{(2)}$}
\author{Alex Youcis$^{(3)}$}
\address[2,3]{\scriptsize Graduate School of Mathematical Sciences, The University of Tokyo,
    3-8-1 Komaba, Meguro-ku, Tokyo, 153-8914, Japan}
\email[2]{\scriptsize naoki@ms.u-tokyo.ac.jp}
\email[3]{\scriptsize ayoucis@ms.u-tokyo.ac.jp}

\date{}

\begin{document}
\begin{abstract}
    We construct a moduli space $\LP_G$ of $\SL_2$-parameters over $\Q$, and show that it has good geometric properties (e.g.\@ explicitly parametrized geometric connected components and smoothness). We construct a \emph{Jacobson--Morozov morphism} ~$\JM\colon \LP_G\to\WDP_G$ (where $\WDP_G$ is the moduli space of Weil--Deligne parameters considered by several other authors). We show that $\JM$ is an isomorphism over a dense open of $\WDP_G$, that it induces an isomorphism between the discrete loci $\LP^\disc_G\to\WDP_G^\disc$, and that for any $\Q$-algebra $A$ it induces a bijection between Frobenius semi-simple equivalence classes in $\LP_G(A)$ and Frobenius semi-simple equivalence classes in $\WDP_G(A)$ with constant (up to conjugacy) monodromy operator.
\end{abstract}

\maketitle

\tableofcontents

\section{Introduction}

\paragraph{Motivation} A problem of fundamental importance in the study of harmonic analysis is the classification of irreducible complex admissible representations of $G(F)$ where $F$ is a non-archimedean local field, and $G$ is a reductive group over $F$. The local Langlands correspondence, a guiding principle for many areas of number theory in the last 40 years, posits a parameterization of such admissible representations in terms of equivalence classes of parameters related to the Galois theory of $F$. These parameters come in several forms. Chief amongst these are the \emph{complex $L$-parameters} which are homomorphisms $\psi\colon W_F\times \SL_2(\C)\to \LG(\C)$ satisfying certain properties (cf.\@ \cite[\S3]{SilbergerZink}), and \emph{complex Weil--Deligne parameters} which are pairs $(\varphi,N)$ where $\varphi\colon W_F\to \LG(\C)$ is a homomorphism and $N$ is a nilpotent element of the Lie algebra of $\wh{G}(\C)$, satisfying certain properties (cf.\@ \cite[\S2.1]{GRAinv}). The notion of equivalence in both cases is that of $\wh{G}(\C)$-conjugacy.

The classical theorem of Jacobson--Morozov (cf.\@ \cite[\S III.11, Theorem 17]{Jacobson}) asserts that the \emph{Jacobson--Morozov map} $\theta \mapsto d\theta\left(\left(\begin{smallmatrix} 0 & 1\\ 0 & 0\end{smallmatrix}\right)\right)$ gives a surjection

\begin{equation*}
    \JM\colon \left\{\begin{matrix}\text{Algebraic homomorphisms}\\ \theta\colon \SL_2(\C)\to \wh{G}(\C)\end{matrix}\right\}\to  \left\{\begin{matrix}\text{Nilpotent elements}\\ N\in\mathrm{Lie}(\widehat{G}(\mathbb{C}))\end{matrix}\right\},
\end{equation*}
which becomes a bijection on the level of $\wh{G}(\C)$-quotients. One may extend this to a \emph{Jacobson--Morozov map} 
\begin{equation*}
     \JM\colon \left\{\begin{matrix}\text{Complex }L
     \text{-parameters}\\ \psi\colon W_F\times \SL_2(\C)\to \LG(\C)\end{matrix}\right\}\to \left\{\begin{matrix}\text{Complex Weil--Deligne parameters}\\ (\varphi,N)\end{matrix}\right\}.
\end{equation*}
This map is not a bijection, even up to equivalence and, in fact, is not even surjective (see Example \ref{eg:JM-not-surj}). However, the Jacobson--Morozov map \emph{does} give a bijection between equivalence classes of Frobenius semi-simple parameters (see \cite[Proposition 2.2]{GRAinv} or \cite[Proposition 1.13]{ImaLLCell}), those which feature most prominently in the local Langlands correspondence. Therefore, in practice the Jacobson--Morozov map allows one to pass fairly freely between these two notions of parameter and to treat them as essentially equivalent. This is useful as each of these perspectives has its own advantages (e.g.\@ as illustrated quite well in \cite{GRAinv}).

The goal of this article is to put the above results on a \emph{moduli-theoretic footing}. Namely we define and study a moduli space of $L$-parameters, and construct a \emph{Jacobson--Morozov morphism} 
\begin{equation*}
    \JM\colon \LP_G\to\WDP_G
\end{equation*}
between the moduli space of $L$-parameters and the moduli space of Weil--Deligne parameters. We then show that there is a natural stratification of the moduli space of Weil--Deligne parameters with the property that over each stratum the Jacobson--Morozov morphism takes a particularly simple form. Using this, we show that the Jacobson--Morozov morphism satisfies some birational-like properties, is an isomorphism over the discrete locus, and that a version of the above bijection between equivalence classes of complex Frobenius semi-simple parameters has an analogue over an arbitrary $\Q$-algebra.

\begin{remn}
The reason we do not restrict our attention to semi-simple parameters is that they do not form a representable presheaf. Thus, to do geometry we are required to work with arbitrary parameters.
\end{remn}

\medskip

\paragraph{Statement of main results} Let $F$ be a non-archimedean local field and $G$ a reductive group over $F$. In \S\ref{ss:L-param-def} we define the \emph{moduli space of $L$-parameters} for $G$ which we denote $\LP_G$. 

\begin{propi}[{see Corollary \ref{cor:LP-pi0}}]\label{propi:L-nice} The moduli space $\LP_G$ is smooth over $\Q$ and has explicitly parameterized affine connected components.
\end{propi}

On the other hand, let $\WDP_G$ denote the moduli space of Weil--Deligne parameters (e.g.\@ as in \cite[\S3.1]{ZhuCohLp}). In \S\ref{ss:JM-mor} we define the \emph{Jacobson--Morozov} morphism 
\begin{equation*}
\JM\colon \LP_G\to \WDP_G.
\end{equation*}
Our major result may then be stated as follows.

\begin{thmi}[{see Theorem \ref{thm:JM-omnibus} and Theorem \ref{thm:JM-isom-disc-locus}}]\label{thmi:JM-weakly-bir} The Jacobson--Morozov morphism is weakly birational and induces an isomorphism $\LP^\disc_G\isomto\WDP_G^\disc$ over the discrete loci.
\end{thmi}

Here we say a morphism of schemes $f\colon Y\to X$ is \emph{weakly birational} if there exists a dense open subset $U\subseteq X$ such that $f\colon f^{-1}(U)\to U$ is an isomorphism. A weakly birational map $f$ is birational if and only if $f$ induces a bijection at the level of irreducible components. Also, the discrete loci inside of $\LP_G$ and $\WDP_G$ are defined, at least when $G$ is semi-simple, as the locus of points where the centralizer of the universal parameter is quasi-finite over the base (see Definition \ref{defn:locus-of-red} and Definition \ref{defn:disc-locus-L} for general definitions).

To prove Theorem \ref{thmi:JM-weakly-bir} we stratify $\WDP_G$ by its nilpotent orbits. Denote by $\wh{\mc{N}}$ the nilpotent variety for $\wh{G}$ and form the stratification $\wh{\mc{N}}^\sqcup\defeq \bigsqcup_N \mc{O}_N$ by its $\wh{G}$-orbits which we treat as a disconnected scheme over $\Q$. We then obtain a stratification $\WDP^\sqcup_G$ by pulling back $\wh{\mc{N}}^\sqcup$ along the natural map $\WDP_G\to \wh{\mc{N}}$. We give an explicit description of the structure of this variety. 

\begin{propi}[{see Corollary \ref{cor:WDP-pi0}}]\label{propi:sqcup-nice} The moduli space $\WDP^\sqcup_G$ is smooth over $\Q$ and has explicitly parameterized connected components.
\end{propi}

The Jacobson--Morozov morphism factorizes through $\WDP_G^\sqcup$ and interacts well with the explicit decompositions indicated in Proposition \ref{propi:L-nice} and Proposition \ref{propi:sqcup-nice}. Utilizing this we show the following, which implies the weakly birational portion of Theorem \ref{thmi:JM-weakly-bir}. 

\begin{propi}[{see Theorem \ref{thm:JM-omnibus}}]\label{propi:JM-bir} The morphism $\JM\colon \LP_G\to\WDP_G^\sqcup$ is birational.
\end{propi}

A key component of our proof of Proposition \ref{propi:JM-bir} is a relative version of the bijection between equivalence classes of complex Frobenius semi-simple parameters. Here, Frobenius semi-simplicity is somewhat delicate and defined in Definition \ref{defn:Frob-ss-WD-param} and Definition \ref{defn:Frob-ss-L-param}.

\begin{thmi}[{see Theorem \ref{thm:rel-JM-param}}]\label{thmi:disc-isom}For any $\Q$-algebra $A$ the map 
\begin{equation*}
    \JM\colon \LP_G(A)/\wh{G}(A)\, \to \,\WDP_G^{\sqcup}(A)/\wh{G}(A)
\end{equation*}
is a bijection on Frobenius semi-simple elements.
\end{thmi}

We finally mention that another important ingredient in our proof of Proposition \ref{propi:JM-bir} is a result which may be interpreted as a stronger version of the isomorphy of the Jacobson--Morozov morphism over the discrete loci, as stated in Theorem \ref{thmi:JM-weakly-bir}. Namely, in Proposition \ref{prop:JM-isom-over-red-locus} we show that the Jacobson--Morozov morphism is an isomorphism over the locus of points of $\WDP_G$ whose centralizer has reductive identity component. The relationship to birationality comes from Proposition \ref{prop:dense-tor-cent} which shows that the locus of such points is dense in $\WDP_G$ and thus, a fortiori, dense in $\WDP^\sqcup_G$ (the same holds true for $\LP_G$).

As the moduli space of Weil--Deligne parameters has featured quite prominently in recent developments in the Langlands program and adjacent fields (e.g.\@ see \cite{BeGeGdef}, \cite{DHKMModLp}, \cite{ZhuCohLp} and \cite{FaScGeomLLC}) we feel that these results will be valuable in the study of the fine structure of the space $\WDP_G$. In particular, one may in theory reduce many questions involving `generic' geometric structure of $\WDP_G$ to the study of $\LP_G$. More specifically, we have stratified the geometric space $\WDP_G$ into pieces such that each stratum is smooth and (essentially) like a homogenous space for a group, and thus simple geometrically (cf.\@ Theorem \ref{thm:WD-const-decomp}). Moreover, each of these strata is birational to similarly defined strata in the representation-theoretically simpler space $\LP_G$. In fact, such ideas have already implicitly appeared in several important geometric results concerning $\WDP_G$ (e.g.\@ see \cite[\S2.3]{BeGeGdef}). 

In addition to its potential uses to study the geometry of $\WDP_G$, we believe that these moduli-theoretic results are clarifying in several other ways. Namely, the weak birationality of the Jacobson--Morozov morphism helps qualify in the classical setting that almost every complex Weil--Deligne parameter is in the image of the Jacobson--Morozov map. Moreover, the isomorphy over the discrete locus may also be used to deduce results of interest even in this classical case (e.g.\@ see Proposition \ref{prop:disc-ss-prop}). Finally, we feel that our explicit description of the moduli space of $L$-parameters (e.g.\@ its set of connected components) helps explain some phenomena differentiating $\LP_G$ from $\WDP_G$ as previously observed by others (c.f.\@ the introduction to \cite{DHKMModLp}).

\medskip

\paragraph{Future directions} While our results are written over $\Q$, it is clear that they extend over $\Z[\frac{1}{N}]$ for sufficiently divisible $N$.
Evidently one cannot hope to extend our results over all of $\Z[\frac{1}{p}]$ as currently written. But, as in op.\@ cit.\@ (and \cite{Helm}), the correct analogue of $\WDP_G$ over $\Z[\frac{1}{p}]$ does not directly involve Weil--Deligne paramters but, instead, involves a scheme of $1$-cocycles for the discretization $W^0_F$ of the tame inertia group. One may then ask whether there is an analogous description of $\LP_G$ which allows our results to work over $\Z[\frac{1}{p}]$.

Also, as the morphism $\JM\colon \LP_G\to\WDP_G$ is weakly birational there exists a dense open subset $U$ of $\WDP_G$ such that $\JM\colon \JM^{-1}(U)\to U$ is an isomorphism. In Proposition \ref{prop:temp-cent-equal} below, we essentially show that the analytication $\JM^{-1}(U)_\C^\mathrm{an}$ contains all (essentially) tempered $L$-parameters. From a geometric perspective (e.g.\@ from the perspective of \cite{FaScGeomLLC}) it is more natural to consider $\ell$-adic $L$-parameters instead of complex ones. One is then naturally led to the ask whether $\JM^{-1}(U)_{\Q_\ell}^\mathrm{an}$ contains the analogue of (essentially) tempered representations, which are the (essentially) $\nu$-tempered representations of Dat (see \cite{DatNu}).

\medskip

\paragraph{Notation and conventions}

\begin{itemize}
    \item $F$ is a non-archimedean local field with residue field of characteristic $p$ and size $q$,
    \item $W_F$ is the Weil group of $F$,
    \item for a Galois extension of fields $k'/k$, we write the Galois group as $\Gamma_{k'/k}$ and we write $\Gamma_k$ for the absolute Galois group of $k$,
    \item for a ring $R$ we shall denote by $\cat{Alg}_R$ the category of $R$-algebras,
    \item we shall frequently abuse terminology and call a covariant functor $\cat{Alg}_R\to\mc{C}$ a $\mc{C}$-valued presheaf,
    \item a reductive group $S$-scheme $H$ will always have connected fibers,
    \item we use the notation $\Int(g)$ for the inner automorphisms associated to an element $g$ of a group,
    \item for a set $X$ we shall denote by $\underline{X}$ the associated constant scheme over $\Q$.
\end{itemize}

\medskip

\paragraph{Acknowledgements} We thank Brian Conrad, Rahul Dalal, Ildar Gaisin, Tasho Kaletha, Marcin Lara, Rachel (Nakyung) Lee, and Alexander Sherman for fruitful discussions. The last named author would particularly like to thank Piotr Achinger for several helpful discussions related to algebraic geometry. Finally, we would like to thank the anonymous referee for their detailed comments which helped significantly improve the clarity of this paper.

The first named author was partially supported by NSF RTG grant DMS-1840234. The second named author was supported by JSPS KAKENHI Grant Number 18H01109. The last named author completed this work under the auspices of the JSPS fellowship.

\section{Some group theoretic preliminaries}

In this section we establish some notation, definitions, and basic well-known results that we shall often use without comment in the sequel. We encourage the reader to skip this section on first reading, referring back only when necessary.

\subsection{The nilpotent variety, unipotent variety, and exponential map}

Let us fix $k$ to be a field of characteristic $0$ and $H$ to be a reductive group over $k$. We denote by $\mf{h}$ the Lie algebra of $H$ thought of both as a vector $k$-space and as a $k$-scheme. 

Let $A$ be a $k$-algebra and $x$ an element of $\mf{h}_A$. Recall then that as in \cite[II, \S6, \textnumero 3]{DemazureGabriel} one may associate an element $\exp(Tx)$ in $H(A\llbracket T\rrbracket)$ to $x$. We then say that $x$ is \emph{nilpotent} if it satisfies any of the following equivalent conditions.

\begin{prop}\label{prop:nilp-equiv} The following are equivalent:
\begin{enumerate}
    \item for all finite-dimensional representations $\rho\colon H\to\GL(V)$ the endomorphism $d\rho(x)$ of $V_A$ is nilpotent,
    \item there exists a faithful finite-dimensional representation $\rho\colon H\to\GL(V)$ such that the endomorphism $d\rho(x)$ of $V_A$ is nilpotent,
    \item $\exp(Tx)$ belongs to $H(A[T])$,
    \item there exists a morphism of group $A$-schemes $\alpha\colon \mathbb{G}_{a,A}\to H_A$ such that $x=d\alpha(1)$,
\end{enumerate}
if $A$ is in addition reduced, then (1)-(4) are equivalent to
\begin{enumerate}
    \item[(5)] $x$ belongs to $\mf{h}^\der_A$ and $\ad(x)$ is a nilpotent transformation of $\mf{h}^\der_A$.
\end{enumerate}
\end{prop}
\begin{proof} The equivalence of (1)-(4) is given by \cite[II, \S6, \textnumero 3, Corollaire 3.5]{DemazureGabriel}. To see the equivalence of (1) and (5), in the case when $A$ is reduced, we may assume that $A$ is a field. Let $\sigma\colon H/Z(H^\der)\to \GL(W)$ be the faithful representation given by taking a direct sum of $\mathrm{Ad} \colon H \to \mathrm{GL}(\mathfrak{h}^{\mathrm{der}})$ and the composition of $H \to H^{\mathrm{ab}}$ with a faithful representation of $H^{\mathrm{ab}}$. It is clear that applying (1) to $\sigma$ shows that (5) holds. Conversely, suppose that (5) holds, so then $d\sigma (x)$ is nilpotent. 
Let $\rho$ be as in (1). We may assume that $\rho$ is irreducible. We put $n=\lvert Z(H^{\mathrm{der}}) \rvert$. Then $\rho^{\otimes n} \colon H \to \mathrm{GL}(V^{\otimes n})$ factors through $H/Z(H^{\mathrm{der}})$. Hence by \cite[Proposition 3.1 (a)]{DeligneHodge} $d\rho^{\otimes n} (x)$ is nilpotent. This implies that $d\rho(x)$ is nilpotent.
\end{proof}

Let us consider the symmetric algebra on $\mf{h}^\ast$ (resp.\@ the graded ideal of positive degree tensors)
\begin{equation*}
    S(\mf{h}^\ast)=\bigoplus_{d\geqslant 0}S^d(\mf{h}^\ast)=\Hom(\mf{h},\mathbb{A}^1_k),\qquad \bigg(\mathrm{resp.}\,\,S^+(\mf{h}^\ast)\defeq\bigoplus_{d>0}S^d(\mf{h}^\ast)\bigg).
\end{equation*}
 Let $S(\mf{h}^\ast)^H$ be the $k$-subalgebra of $S(\mf{h}^\ast)$ which is invariant for the adjoint action of $H$ on $\mf{h}$ (in the sense of \cite[Definition 0.5 i)]{Mumford}). Let us then consider the radical ideal
\begin{equation*} S^+(\mf{h}^\ast)^H\defeq S^+(\mf{h}^\ast)\cap S(\mf{h}^\ast)^H.
\end{equation*}
The \emph{nilpotent variety} of $H$ is the closed subshceme of $\mf{h}$ given by $\mc{N}\defeq V\left(S^+(\mf{h}^\ast)^H\right)$ (or $\mc{N}_H$ when we want to emphasize $H$). This name is justified as for any extension $k'$ of $k$ we have
\begin{equation*}
    \mc{N}(k')=\left\{x\in \mf{h}_{k'}: x\text{ is nilpotent}\right\}
\end{equation*}
(cf.\@ \cite[\S 6.1, Lemma]{Jantzen}). In particular, $\mc{N}$ is the unique reduced subscheme of $\mf{h}$ whose $\ov{k}$-points consist of the nilpotent elements of $\mf{h}_{\ov{k}}$.

The nilpotent variety $\mc{N}$ is an integral (cf.\@  \cite[\S6.2, Lemma]{Jantzen}) finite type affine $k$-scheme of dimension $\dim(H)-r$ where $r$ is the geometric rank of $H$ (see \cite[\S6.4]{Jantzen}). In fact, as $k$ is of characteristic $0$, it is normal by the results of \cite{KostantLieGroupReps}. Observe that the nilpotent variety is stable under the adjoint action of $H$. Also observe that if $f\colon H\to H'$ is a morphism of reductive groups over $k$ it induces a morphism $df\colon \mc{N}_H\to \mc{N}_{H'}$ and satisfies $df(\Ad(h)(x))=\Ad(f(h))(df(x))$. 

\begin{eg}\label{eg:gln-nilp}
Let $\mathrm{Mat}_{n,k}$ be the scheme of $n$-by-$n$ matrices over $k$, and let $I\subseteq \mc{O}(\mathrm{Mat}_{n,k})$ be generated by those polynomials corresponding to $(a_{ij})^n=0$. Then, $\mathcal{N}_{\GL_{n,k}}=V(\sqrt{I})$.
\end{eg}

From this example, and the functoriality of the nilpotent variety, it's easy to see that if $A$ is a $k$-algebra, then one has the containment
\begin{equation*}
    \mc{N}(A)\subseteq \{x\in \mf{h}_A : x\text{ is nilpotent}\},
\end{equation*}
which is an equality if $A$ is reduced, but can differ otherwise. From this containment we see that for any element $x$ of $\mathcal{N}(A)$ we may define an element $\exp(x)$ of $H(A)$ as in \cite[II, \S6, \textnumero 3, 3.7]{DemazureGabriel}. As this association is functorial we obtain an $H$-equivariant morphism of schemes $\mc{N}\to H$ called the \emph{exponential morphism} and denoted by $\exp$ (or $\exp_H$ when we want to emphasize $H$) which is functorial in $H$. We would now like to describe the image of $\exp$.

To this end, note that there exists a unique reduced closed subscheme $\mc{U}$ (or $\mc{U}_H$ when we want to emphasize $H$) of $H$ such that 
\begin{equation*}
    \mc{U}(k')=\left\{h\in H(k'): h\text{ is unipotent}\right\},
\end{equation*}
for all extensions $k'$ of $k$ (see \cite[Proposition 1.1]{SpringerUnipotent}). We call $\mc{U}$ the \emph{unipotent variety} associated to $H$. It is an integral finite type affine $k$-scheme of dimension $\dim(H)-r$ which is stable under the conjugation action of $H$ (see loc.\@ cit.\@). Moreover, as $k$ is of characteristic $0$, it is normal (see \cite[Proposition 1.3]{SpringerUnipotent}). We observe that $\mc{U}$ is stable under the conjugation action of $H$. 

Observe that $\exp$ factorizes through $\mc{U}$, as both are reduced, and so this may be checked on the level of $\ov{k}$-points. We have the following omnibus result concerning the exponential morphism.

\begin{prop}\label{prop:exp-omnibus}Let $H$ be a reductive group over a characteristic $0$ field $k$. Then, 
\begin{enumerate}
    \item the exponential map $\exp\colon \mc{N}_H\to \mc{U}_H$ is an $H$-equivariant isomorphism,
    \item for any $k$-algebra $A$ and any $x$ in $\mc{N}_H(A)$, $\Ad(\exp(x))$ is equal to $\sum_{i=0}^\infty \frac{1}{i!}\ad(x)^i$,
    \item for any $k$-algebra $A$ and any nilpotent Lie subalgebra $\mf{n}$ of $\mf{h}_A$ contained in $\mc{N}(A)$ the subset $\exp(\mf{n})\subseteq H(A)$ is a subgroup. If the functor $\mf{n}\mapsto \mf{n}\otimes_A B$ is representable by a closed subgroup scheme of $\mc{N}_A$ then $\exp(\mf{n})$ is actually a closed subgroup scheme of $H_A$ such that $\exp(\mf{n})_x$ is unipotent for all $x$ in $\Spec(A)$.
\end{enumerate}
\end{prop}
\begin{proof} For (1), as $\mc{N}_H$ and $\mc{U}_H$ are connected and normal, and $\exp$ may be checked to be a bijection on $\ov{k}$-points, this follows from Zariski's main theorem as $k$ is of characteristic $0$. Claim (2) follows by the functoriality of the exponential map (cf.\@ \cite[II, \S6, \textnumero 3, 3.7]{DemazureGabriel}). Finally, (3) may be deduced by the Campbell--Hausdorff series (see \cite[II, \S6, \textnumero 4, Th\'eor\`eme 2]{Bourbaki}).
\end{proof}

\subsection{The $L$-group and $C$-group}\label{ss:L-and-C}
Fix $F$ to be a non-archimedean local field, and let $G$ be a reductive group over $F$. In this subsection we define the $C$-group of $G$, which is a modification of the $L$-group of $G$ that will be used to construct a moduli space of Weil--Deligne parameters over $\mathbb{Q}$ without choosing a square root of $q$ (see \S 5.1).

To begin, let $\Psi(G)$ denote the canonical based root datum of $G_{\ov{F}}$ (see \cite[\S1.1]{KotStfcus} and \cite[\S21.42]{MilneGroups}) which comes equipped with an action of $\Gamma_F$. We fix once and for all a \emph{Langlands dual group of} $G$ by which we mean a pinned reductive group $(\widehat{G},\wh{B},\wh{T},\{x_\alpha\})$ over $\Q$ (see \cite[\S23.d]{MilneGroups}) together with an isomorphism between $\Psi(\widehat{G},\wh{B},\wh{T})$ and $\Psi(G)^\vee$. We denote by $\wh{\mf{g}}$ the Lie algebra of $\wh{G}$, and by $\wh{\mc{N}}$ the nilpotent variety of $\wh{G}$.

Next, let $\cW_F$ denote the \emph{Weil group scheme} over $\bQ$ associated to $F$ 
as in \cite[(4.1)]{TatNtb}. For a $\Q$-algebra $A$ one may identify $\mc{W}_F(A)$ with the set of continuous maps $f\colon \pi_0(\Spec(A))\to W_F$ where here $\pi_0(\Spec(A))$ is thought of as a profinite space (cf.\@ \stacks{0906}) and $W_F$ is given its usual topology. In particular, $\mc{W}_F(A)=\underline{W_F}(A)$  when $\pi_0(\Spec(A))$ is discrete (e.g.\@ if $A$ is connected or Noetherian), but can differ otherwise. For $w$ in $W_F$ we shall occasionally abuse notation and use $w$ to also denote its image in $\mc{W}_F(A)$. 

Note that if $d\colon W_F\to \Z$ is the degree map sending a lift of arithmetic Frobenius to $-1$, then there is a morphism of $\Q$-group schemes $d\colon \mc{W}_F\to \underline{\Z}$ which takes a map $f$ to $d\circ f$. Observe that $\underline{\Z}$ admits an embedding of group $\Q$-schemes into $\bb{G}_{m,\Q}$ corresponding to $1\mapsto q^{-1}$ and we denote the composition of $d$ with this map by $\|\cdot\|\colon \mc{W}_F\to \bb{G}_{m,A}$. We define $\cI_F =\ker( \| \cdot \|)$, which is an \emph{affine scheme} equal to $\varprojlim \underline{I_F/I_K}$ as $K$ travels over all finite extensions of $F$. Note that if $A$ is a $\Q$-algebra and $X$ an $A$-scheme locally of finite presentation then any morphism of $A$-schemes $\mc{I}_{F,A}\to X$ must factorize through $\underline{I_F/I_K}$ for some $K$ (cf.\@ \stacks{01ZC}).

\begin{rem} One reason to prefer $\mc{W}_F$ over the constant group scheme $\underline{W_F}$ is that the topological group $\pi_0(\mc{W}_F)$ is equal to $W_F$ with its usual topology, and similarly for $\mc{I}_F$.
\end{rem}

Returning to $G$, note that the action of $\Gamma_F$ on $\Psi(G)$ gives rise to an action of $\Gamma_F$ on $(\wh{G},\wh{B},\wh{T},\{x_\alpha\})$ and, in particular, on $\wh{G}$ as a group $\Q$-scheme. 
We define a finite Galois extension $F^\ast$ of $F$ characterized by the equality
$\Gamma_{F^\ast}=\ker (W_F \to \Aut (\wh{G}))$. Equivalently, $F^\ast$ is the minimal field splitting $G^\ast$, the quasi-split inner form of $G$. We write $\Gamma_\ast$ for $\Gamma_{F^\ast/F}$. As $\underline{\Gamma_\ast}$ acts on $\wh{G}$ and $\mc{W}_F$ admits $\underline{\Gamma_\ast}$ as a quotient, we obtain an action of $\cW_F$ on $\wh{G}$. Define the \emph{$L$-group scheme} of $G$ to be the group $\Q$-scheme $\LG =\wh{G} \rtimes \cW_F$. Observe that there is a natural inclusion $\wh{G}\hookrightarrow \LG$ which identifies $\wh{G}$ as a normal subgroup scheme of $\LG$. In particular, there is a natural conjugation action of $\LG$ on $\wh{G}$, which in turn induces an adjoint action of $\LG$ on $\wh{\mf{g}}$. 
 
As the action of $\mc{W}_F$ on $\wh{G}$ factorizes through a finite quotient, we see by Lemma \ref{lem:fixed-points-reductive} below that the group presheaf associating a $\Q$-algebra $A$ to  $Z_0(\wh{G})(A)\defeq Z(\wh{G})(A)^{\mc{W}_F(A)}$ is representable.

\begin{lem}\label{lem:fixed-points-reductive} Let $A$ be a $\Q$-algebra, $H$ a reductive group over $A$, and $\Sigma$ a finite group acting on $H$ by group $A$-scheme automorphisms. Then, the group functor
\begin{equation*}
    H^\Sigma\colon\cat{Alg}_A\to \cat{Grp},\quad B\mapsto H(B)^\Sigma
\end{equation*}
is represented by a subgroup scheme of $H$ smooth over $A$, with $(H^\Sigma)^\circ$ reductive over $A$, and such that for all $A$-algebras $B$ one has the equality $\mathrm{Lie}(H^\Sigma)(B)=\mathrm{Lie}(H)(B)^\Sigma$.
\end{lem}
\begin{proof} Write $H=\Spec(R)$, then one easily verifies that $\Spec(R_\Sigma)$, where $R_\Sigma$ is the ring of coinvariants, represents $H^\Sigma$. As $A$ is a $\Q$-algebra, it is evident that $R_\Sigma$ is a direct summand of $R$ and thus $H^\Sigma$ is flat over $A$, and thus smooth. By \cite[Expos\'{e} VIB, Corollaire 4.4]{SGA3-1} we know that $(H^\Sigma)^\circ$ is representable and smooth over $A$, and it is then reductive by \cite[Theorem 2.1]{PrasadYu}). The claim about Lie algebras is clear as the functor of $\Sigma$-invariants preserves kernels.
\end{proof}

Let $X^\ast$ denote the cocharacter component of $\Psi(G)$ and $R^+$ the positive root component, and define $\delta$ to be the element of $X^\ast$ given by the sum over the elements of $R^+$. By our identification between $\Psi(\wh{G},\wh{B},\wh{T})$ and  $\Psi(G)^\vee$ we see that $\delta$ corresponds to an element of $X_\ast(\wh{T})$ which we also denote by $\delta$. Let us set $z_G\defeq\delta(-1)\in \wh{T}(\Q)[2]$. By the proof of \cite[Proposition 5.39]{BuGeconjc}, $z_G$ lies in $Z_0(\wh{G})(\Q)$. Thus, the action of $\mc{W}_F$ on $\wh{G}\times \bb{G}_{m,\Q}$ (with trivial action on the second component) fixes the pair $(z_G,-1)$. Therefore, $\mc{W}_F$ acts on $\check{G}\defeq (\widehat{G}\times \bb{G}_{m,\Q})/\langle (z_G,-1)\rangle$. We then define the \emph{$C$-group scheme} of $G$ to be $\CG =\check{G} \rtimes \cW_F$. Note that by \cite[Proposition 5.39]{BuGeconjc} there exists a central extension $\wt{G}$ of $G$ such that $\CG$ is naturally isomorphic to ${^L}\!\wt{G}$, which is the definition of the $C$-group as in \cite[Definition 5.38]{BuGeconjc}.

The group $\wh{G}$ admits a natural embedding into $\check{G}$, with normal image, via the first factor, and therefore we obtain a conjugation action of $\CG$ on $\wh{G}$, and thus an adjoint action of $\CG$ on $\wh{\mf{g}}$. Also, the morphism
\begin{equation*}
    (\wh{G}\times \bb{G}_{m,\Q})\rtimes \cW_F\to \bb{G}_{m,\Q} \times \cW_F,\qquad (g,z,w)\mapsto (z^2,w)
\end{equation*}
annihilates $\langle (z_G,-1)\rangle$, and thus induces a morphism 
\begin{equation*}
    p_C =(p_{\bG_m},p_{\cW_F})\colon \CG \to \bb{G}_{m,\Q} \times \cW_F . 
\end{equation*}
Finally, we observe that if $k$ is an extension of $\Q$, and $c$ is any element of $k$ such that $c^2=q$, then there is a morphism $i_c\colon \LG_k\to \CG_k$
obtained as the composition
\begin{equation*}
    \LG_k \xrightarrow{(g,w)\mapsto (g,c^{-d(w)},w)}(\wh{G}_k \times \bb{G}_{m,k} )\rtimes \cW_{F,k}\to \CG_k . 
\end{equation*}

\subsection{Scheme of homomorphisms and cross-section homomorphisms} We establish here some terminology and basic results concering the scheme of homomorphisms as well as the scheme of cross-section homomorphisms (in the sense of \cite[Appendix A]{DHKMModLp}). Throughout the following we fix $k$ to be field of characteristic $0$.

\medskip

\paragraph{Scheme of homomorphisms}

Let $H$ and $H'$ be reductive groups over $k$ with Lie algebras $\mf{h}$ and $\mf{h'}$. For a $k$-algebra $A$ denote by $\Hom(H_A,H'_A)$ the set of group $A$-scheme morphisms $H_A\to H'_A$. Consider the following functor
\begin{equation*}
    \underline{\Hom}(H,H')\colon \cat{Alg}_k\to \cat{Set},\qquad A\mapsto \Hom(H_A,H'_A),
\end{equation*}
and define the functor $\underline{\Hom}(\mf{h},\mf{h}')$ similarly, both of which carry a natural $H'$-conjugation action.

\begin{prop}\label{prop:hom-schem-omnibus} The following statements hold true.
\begin{enumerate}
    \item The functor $\underline{\Hom}(H,H')$ is representable by a smooth $k$-scheme for which the map
    \begin{equation*}
        \varrho \colon H'\times\underline{\Hom}(H,H')\to\underline{\Hom}(H,H')\times \underline{\Hom}(H,H'),\quad (h,f)\mapsto (hfh^{-1},f)
    \end{equation*}
    is smooth,
    \item if $H$ is semi-simple then $\underline{\Hom}(H,H')$ is affine, and if $H$ furthemore simply connected then the map
    \begin{equation*}
        \underline{\Hom}(H,H')\to\underline{\Hom}(\mf{h},\mf{h}'),\qquad f\mapsto df,
    \end{equation*}
    is an $H'$-equivariant isomorphism,
    \item for any $k$-algebra $A$ the natural map
    \begin{equation*}
        \Hom(H_A,H'_A)\to \Hom(H(A),H'(A))
    \end{equation*}
   is injective.
\end{enumerate}
\end{prop}
\begin{proof} Statements (1) and (2) follow from \cite[Theorem 2]{Brion} and \cite[Exp.\@ XXIV, Proposition 7.3.1]{SGA3-3new}  respectively. Statement (3) follows from Proposition \ref{prop:unirational} below as $H$ and $H'$ are integral and unirational (see \cite[Summary 1.36, Theorem 3.23, and Theorem 17.93]{MilneGroups}).
\end{proof}

\begin{prop}\label{prop:unirational} Suppose that $X$ and $Y$ are finite type integral $k$-schemes with $X$ unirational. Then for any $k$-algebra $A$, the natural map
\begin{equation*}
    \Hom(X_A,Y_A)\to \Hom(X(A),Y(A))
\end{equation*}
is injective.
\end{prop}
\begin{proof} Let $f$ and $g$ be two different $A$-morphisms $X_A\to Y_A$. Note that the claim is clearly local on $Y$, and so we may assume that $Y$ is affine. It is also clear that by embedding $Y$ into $\mathbb{A}^n_k$, and checking coordinate-by-coordinate we may further assume that $Y=\mathbb{A}^1_k$. It also suffices to check locally on $X$, and so using
the unirationality of $X$ we may then further assume that $X=D(w)\subseteq \mathbb{A}^n_k$ for $w$ in $k[x_1,\ldots,x_n]$. With this we may interpret $f$ and $g$ as elements of $A[x_1,\ldots,x_n][w^{-1}]$. Taking the difference of $f$ and $g$ and multiplying by an appropriate power of $w$ allows us to further assume that $f$ lies in $A[x_1,\ldots,x_n]$ and $g$ is the zero map. By considering the map $X(k)\to X(A)$, we will be done if we can show that $f$ does not vanish identically on $D(w)(k)$. If $\{ a_i \}_{i \in I}$ is a basis of $A$ as a $k$-vector space then we may write $f=\sum_{i \in I} a_i f_i$ where $f_i \in k[x_1,\ldots, x_n]$. As $f$ is non-zero there exists some $i$ such that $f_i$ is non-zero. As $D(w)(k)$ is Zariski dense in $\mathbb{A}^n_k$ as $k$ is infinite, there then exists some $x$ in $D(w)(k)$ such that $f_i(x)\ne 0$. Then, by setup, $f(x)\ne 0$.
\end{proof}

In the future, we call a homomorphism of groups $H(A)\to H'(A)$ \emph{algebraic} if it is the map on $A$-points of a morphism (necessarily unique) of group $A$-schemes $H_A\to H'_A$.

\medskip

\paragraph{Schemes of cross-section homomorphisms} Fix an abstract group $\Sigma$ and a reductive group $H$ over $k$. We then consider the presheaf
\begin{equation*}
    \underline{\Hom}(\Sigma,H)\colon\cat{Alg}_k\to \cat{Set},\qquad \Hom(\Sigma,H(A))=\Hom(\underline{\Sigma}_A,H_A).
\end{equation*}
This presheaf clearly carries an $H$-conjugation action. If, in addition, $\Sigma$ acts on $H$ by group $k$-scheme morphisms then for a $k$-algebra $A$ we say a homomorphism $f\colon \underline{\Sigma}_A\to H_A\rtimes \underline{\Sigma}_A$ is a \emph{cross-section homomorphism} over $A$ if $p_2(f(\sigma))=\sigma$ for all $\sigma$, where $p_2\colon H_A\rtimes \underline{\Sigma}_A\to\underline{\Sigma}_A$ is the scheme-theoretic projection. We denote by $\underline{Z}^1(\Sigma,H)(A)$ the set of cross-section homomorphisms over $A$ which is clearly a presheaf on $k$-algebras which carries an $H$-conjugation action.

\begin{remn}
The notation $\underline{Z}^1(\Sigma,H)$ is used as this object is equal to the scheme of $1$-cocycles in \cite[Appendix A]{DHKMModLp}.
\end{remn}

\begin{prop}[{\cite[Lemma A.1 and Corollary A.2]{DHKMModLp}}]\label{prop:cocycle-scheme} Suppose that $\Sigma$ is finite. Then, $\underline{\Hom}(\Sigma,H)$ (resp.\@  $\underline{Z}^1(\Sigma,H)$) is represented by a finite type smooth affine $k$-scheme. Moreover, for all $k$-algebras $A$, and all $f$ in $\underline{\Hom}(\Sigma,H)(A)$ (resp.\@ $\underline{Z}^1(\Sigma,H)(A)$) the orbit map
\begin{equation*}
    \mu_f\colon H_A\to \underline{\Hom}(\Sigma,H)_A,\qquad \bigg(\text{resp.}\,\, \mu_f\colon H_A\to\underline{Z}^1(\Sigma,H)_A\bigg)
\end{equation*}
is smooth.
\end{prop}

\subsection{Transporter and centralizer schemes} Let $R$ be a ring, $H$ a group-valued functor on $\cat{Alg}_R$, and $X$ a set-valued functor on $\cat{Alg}_R$. Then, for an $R$-algebra $S$ and two elements $\alpha$ and $\beta$ of $X(S)$ we define the \emph{transporter set} to be
\begin{equation*}
    \mathrm{Transp}_H(\alpha,\beta)\defeq \left\{h\in H(S): h\cdot \alpha=\beta\right\}.
\end{equation*}
We then define the \emph{transporter presheaf} to be the presheaf
\begin{equation*}
    \Transp_H(\alpha,\beta)\colon \cat{Alg}_S\to \cat{Set},\qquad T\mapsto \mathrm{Transp}_H(\alpha_T,\beta_T).
\end{equation*}
We abbreviate $\Transp_H(\beta,\beta)$ to $Z_H(\beta)$ and call it the \emph{centralizer presheaf}, which is clearly a group presheaf. We then have the following obvious proposition.

\begin{prop}\label{prop:transp-rep} Suppose that $H$ is a group $R$-scheme and that $X$ is a separated $R$-scheme of finite presentation. Then, for any $R$-algebra $S$ and any elements $\alpha$ and $\beta$ of $X(S)$, the presheaves $\Transp_H(\alpha,\beta)$ and $Z_H(\beta)$ are representable by closed finitely presented subschemes of $H_S$. Moreover, for any $S$-algebra $T$ one has the natural equalities
\begin{equation*}
    \Transp_H(\alpha,\beta)_T=\Transp_H(\alpha_T,\beta_T),\qquad Z_H(\beta)_T=Z_H(\beta_T).
\end{equation*}
\end{prop}

\subsection{Eigenvalue decomposition of Lie algebras}

We record here the following result, which, for a cocharacter of a smooth group scheme, relates a character decomposition and an eigenvalue decomposition of the Lie algebra.

\begin{lem}\label{lem:Gm-Ad-ad}
Let $S$ be a scheme and $H$ a smooth group $S$-scheme with Lie algebra $\mf{h}$. 
Let $\rho \colon \bG_{m,S} \to H$ be a morphism of group $S$-schemes. Set $h=d\rho(1)$, and for an integer $i$ we set 
\begin{equation*}
    \mf{h}_{\rho,i}=\{ x \in \mf{h} : \Ad (\rho (z)) x=z^i x\text{ for all }z\},\qquad \mf{h}_{h,i}=\{ x \in \mf{h} : \ad (h)(x)=ix \}.
\end{equation*} Then we have $\mf{h}_{\rho,i} \subseteq \mf{h}_{h,i}$. 
This is an equality if $S$ is a $\Q$-scheme. 
\end{lem}
\begin{proof}
We have $d(\Ad \circ \rho)(1)=\ad (h)$ under the identification of the Lie algebra of $\GL (\mf{h})$ with $\End (\mf{h})$. 
By taking the weight decomposition of $\mf{h}$ under $\Ad \circ \rho$ (cf.\@ \cite[Lemma A.8.8]{CGP}), we obtain the claim from the fact that the derivative of the $i^\text{th}$-power map $\mathbb{G}_{m,S} \to \mathbb{G}_{m,S}$ is the multiplication-by-$i$ map. The last claim follows from $\mf{h}=\bigoplus_{i \in \bZ} \mf{h}_{\rho,i}$ and that $\mf{h}_{h,i}$ for $i \in \bZ$ are linearly independent if $S$ is a $\Q$-scheme. 
\end{proof}

\section{The classical setting}

In this section we recall the Jacobson--Morozov theorem and the Jacobson--Morozov theorem for parameters in their classical settings. This will not only serve to emphasize the results we wish to geometrize, but will play an important role in the proof of these more general results.

\subsection{The Jacobson--Morozov theorem}

Let $k$ be a field of characteristic $0$ and $H$ an algebraic group over $k$ such that $H^{\circ}$ is reductive. It will be useful to explicitly name the matrices
\begin{equation*}
    e_0=\begin{pmatrix}0 & 1\\ 0 & 0\end{pmatrix},\quad h_0=\begin{pmatrix}1 & 0\\ 0 & -1\end{pmatrix},\quad f_0=\begin{pmatrix}0 & 0\\ 1 & 0\end{pmatrix},
\end{equation*}
which form a $k$-basis of the Lie algebra $\mf{sl}_{2,k}$. We then have the Jacobson--Morozov Theorem as follows.

\begin{thm}[{cf.\cite[VIII, \S 11, \textnumero 2, Proposition 2 and Corollaire]{BourLie78}}]\label{thm:JM-classical} The map 
\begin{equation*}
    \JM\colon \Hom(\SL_{2,k},H)\to \mc{N}(k),\qquad \theta\mapsto d\theta(e_0)
\end{equation*}
is an $H(k)$-equivariant surjection, and induces a bijection 
\begin{equation*}
    \Hom(\SL_{2,k},H)/H(k)\to \mc{N}(k)/H(k).
\end{equation*}
\end{thm}

Let us call a triple $(e,h,f)$ of elements an \emph{$\mf{sl}_2$-triple} in $\mf{h}$ if the following equalities hold
\begin{equation*}
    [h,e]=2e,\quad [h,f]=-2f,\quad [e,f]=h.
\end{equation*}
Let us denote by $\mc{T}(k)$ (or $\mc{T}_H(k)$ when we want to emphasize $H$), the set of $\mf{sl}_2$-triples in $\mf{h}$. The natural adjoint action of $H(k)$ on $\mf{h}$ induces an action of $H(k)$ on $\mc{T}(k)$.

\begin{thm}\label{thm:rel-JM-triples-classical} The following diagram is commutative and each arrow is a bijection 
\begin{equation*}
    \xymatrixrowsep{3pc}\xymatrixcolsep{5pc}\xymatrix{\Hom(\SL_{2,k},H)/H(k)\ar[r]^{\theta\,\longmapsto \,d\theta}\ar[d]^{\JM} & \Hom(\mf{sl}_{2,k},\mf{h})/H(k)\ar[d]^{\nu\mapsto (\nu(e_0),\nu(h_0),\nu(f_0))}\\ \mathcal{N}(k)/H(k)
    & \mathcal{T}(k)/H(k). 
    \ar[l]_{e\,\longmapsfrom \,(e,h,f)}}
\end{equation*}
\end{thm}

We end this subsection by explaining the relationship between the centralizers of $\theta$ and $N=\JM(\theta)$. Namely, let us set 
\begin{equation*}
 \mf{u}^N = \mathrm{im}(\ad (N)) \cap \ker (\ad (N)),\qquad U^N=\exp(\mf{u}^N). 
\end{equation*}
Then, we have the following Levi decomposition statement.

\begin{prop}\label{prop:Zudec}
The equality $Z_{H}(N) = U^N\rtimes Z_{H}(\theta)$ holds. Further we have 
\begin{equation*}
    \Lie (Z_{H}(\theta))=\Lie (Z_{H}(N))_0,\qquad \Lie (U^N)=\bigoplus_{i >0} \Lie (Z_{H}(N))_i,
\end{equation*}
where for an integer $i$ we set
\begin{equation*}
    \Lie (Z_{H}(N))_i=\{ x \in \Lie Z_{H}(N) : \Ad \left( \theta \left( \left( \begin{smallmatrix}z & 0\\ 0 & z^{-1} \end{smallmatrix} \right) \right) \right)x=z^i x \}.
\end{equation*}
\end{prop}
\begin{proof}
The first claim is proved in the same way as \cite[Proposition 2.4]{BaVoUnipss}. The second follows from \cite[Lemma 5.1]{Elkington} by taking the derived group of $H^{\circ}$. 
\end{proof}

\subsection{The Jacobson--Morozov theorem for parameters}\label{ss:JM-for-params-classical}

We now recall the analogue of the Jacobson--Morozov theorem for parameters. We use the notation from \S\ref{ss:L-and-C}.

\begin{defn} Topologize $\LG(\C)$ by giving $\wh{G}(\C)$ the classical topology.
\begin{enumerate} 
\item A \emph{(complex) Weil--Deligne parameter} for $G$ is a pair $(\varphi,N)$ where 
\begin{itemize}
    \item $\varphi\colon W_F\to \LG(\C)$ is a continuous cross-section homomorphism,
    \item $N\in\wh{\mc{N}}(\C)$ is such that $\mathrm{Ad}(\varphi(w))(N)=\|w\|N$ for all $w\in W_F$.
\end{itemize}
\item A \emph{(complex) $L$-parameter} for $G$ is a map 
\begin{equation*}
    \psi\colon W_F\times \SL_2(\C)\to \LG(\C),
\end{equation*}
such that 
\begin{itemize}
    \item $\psi|_{W_F}\colon W_F\to \LG(\C)$ is a continuous cross-section homomorphism,
    \item $\psi|_{\SL_2(\C)}\colon \SL_2(\C)\to \LG(\C)$ takes values in $\widehat{G}(\C)$ and is algebraic.
\end{itemize}
\end{enumerate}
\end{defn}

For $\tau\in \{L,\WD\}$ let us denote by $\Phi^{\tau,\square}_G$ the set of complex $\tau$-parameters for $G$. Recall that a Weil--Deligne parameter $(\varphi,N)$ (resp.\@ an $L$-parameter $\psi$) is called \emph{Frobenius semi-simple} if for one (equiv.\@ for any) lift $w_0$ of arithmetic Frobenius the element $\varphi(w_0)$ (resp.\@ $\psi(w_0)$) is semi-simple (in the sense of \cite[\S8.2]{BorelCorvallis}). We denote by $\Phi^{\tau,\ss,\square}_G$ the subset of Frobenius semi-simple $\tau$-parameters. For each $\tau$ there is a natural action of $\wh{G}(\C)$ on $\Phi^{\tau,\square}_G$ which stabilizes the subset $\Phi^{\tau,\ss,\square}_G$. We then define $\Phi^\tau_G\defeq \Phi^{\tau,\square}_G/\wh{G}(\C)$ and $\Phi^{\tau,\ss}_G\defeq \Phi^{\tau,\ss,\square}_G/\wh{G}(\C)$. For an element $\psi$ of $\Phi^{L,\square}_G$ we denote by $\theta$ (or $\theta_\psi$ when we want to emphasize $\psi$) the morphism $\psi|_{\SL_2(\C)}\colon \SL_2(\C)\to \wh{G}(\C)$.

To upgrade Theorem \ref{thm:JM-classical} to the parameter setting, we need to associate a Weil--Deligne parameter to any $L$-parameter. To this end, let us define a morphism of groups
\begin{equation*}
    i=(i_1,i_2)\colon W_F\to W_F\times \SL_2(\C),\qquad w\mapsto \left(w,\left(\begin{smallmatrix}\|w\|^{\frac{1}{2}} & 0\\ 0 & \|w\|^{-\frac{1}{2}}\end{smallmatrix}\right)\right).
\end{equation*}
We then define the \emph{Jacobson--Morozov map} to be the $\wh{G}(\C)$-equivariant map
\begin{equation*}
   \JM\colon \Phi^{L,\square}_G\to \Phi^{\WD,\square}_G,\qquad \psi \mapsto (\psi\circ i,d\theta(e_0)).
\end{equation*}
It is easy to check that $\JM^{-1}(\Phi^{\WD,\ss,\square}_G)$ is precisely $\Phi^{L,\ss,\square}_G$. As the Jacobson--Morozov map is $\widehat{G}(\C)$-equivariant it induces maps $\Phi^{L}_G\to \Phi^{\WD}_G$ and $\Phi^{L,\ss}_G\to \Phi^{\WD,\ss}_G$.

The Jacobson--Morozov map is not a bijection as the following example illustrates.

\begin{eg}\label{eg:JM-not-surj} Set $G=\GL_4$ and as $G$ is split we may replace $\LG(\C)$ with $\wh{G}(\C)=\GL_4(\C)$. Consider the Weil--Deligne parameter $(\varphi,N)$ given as follows
\begin{equation*}
   \varphi\colon w\mapsto \begin{pmatrix} q^2 & 0 & 0 & 0 \\ 0 & q &1 &0 \\ 0 & 0& q & 0 \\ 0 & 0& 0& 1 \end{pmatrix}^{d(w)},\qquad N=\begin{pmatrix} 0 & 0 & 1 & 0 \\ 0 & 0 &0 &1 \\ 0 & 0& 0 & 0 \\ 0 & 0& 0& 0 \end{pmatrix}.
\end{equation*}
Suppose for contradiction that $(\varphi,N)=\JM(\psi)$ for some $\psi$ in $\Phi^{L,\square}_G$. Then, $\psi$ is of the form $\rho \boxtimes \mathrm{Std}$, where $\rho$ is a representation of $W_F$ and $\mathrm{Std}$ is the standard representation of $\mathrm{SL}_2(\C)$. Indeed, as $N$ is conjugate to 
\begin{equation*}
  \begin{pmatrix}0 & 1 & 0 & 0\\ 0 & 0 & 0 & 0\\ 0 & 0 & 0 & 1\\ 0 & 0 & 0 & 0\end{pmatrix},
\end{equation*}
we see from the Jacobson--Morozov theorem that as an $\mathrm{SL}_2(\C)$ representation $\mathbb{C}^4$ is isomorphic to $\mathrm{Std}^{\oplus 2}$. One may then check that the morphism 
\begin{equation*}
    \Hom_{\SL_2(\C)}(\mathrm{Std},\mathbb{C}^4)\boxtimes \mathrm{Std}\to \mathbb{C}^4
\end{equation*}
is an isomorphism of $W_F\times\SL_2(\C)$-representations. Note that by examining the relationship between the $W_F$-actions of $\varphi$ and $\psi$, the twist of $\rho$ by the unramified character $w \mapsto \|w\|^{-1/2}$ must be isomorphic to the representation on $\mathrm{Ker} N$ induced by $\varphi$. In particular $\rho$ is semi-simple. Hence the Weil--Deligne parameter attached to $\psi$ must be Frobenius semi-simple, but the original $(\varphi,N)$ is not Frobenius semi-simple. 
\end{eg}

However, we have the following Jacobson--Morozov theorem for parameters.

\begin{thm}[{see \cite[Proposition 2.2]{GRAinv} or \cite[Proposition 1.13]{ImaLLCell}}]\label{thm:JM-params-classical} The Jacobson--Morzov map $\JM\colon\Phi^{L,\ss,\square}_G\to \Phi^{\WD,\ss,\square}_G$ is a surjection and induces a bijection $\Phi^{L,\ss}_G\to \Phi^{\WD,\ss}_G$.
\end{thm}

\subsection{Bijection over reductive centralizer locus and applications}\label{ss:red-loc-classical}

The Jacobson--Morozov theorem for parameters is stated at the level of $\wh{G}(\C)$-orbits. While this is a non-issue for now, when we attempt to geometrize this result it becomes more problematic due to the subtle nature of quotients in algebraic geometry. So, we wish to upgrade the Jacobson--Morozov theorem for parameters to a bijectivity statement before quotienting by $\wh{G}(\C)$. 

To begin, we give an analogue of Proposition \ref{prop:Zudec} for parameters. To state it, let $(\varphi,N)$ be an element of $\Phi^{\WD,\square}_G$ and set $U^N(\varphi) \defeq U^N(\C) \cap Z_{\widehat{G}(\C)}(\varphi)$. 

\begin{prop}{\label{prop:Zphidec}}
Let $\psi$ be an element of $\Phi^{L,\square}_G$ and set $(\varphi, N)=\JM(\psi)$. Then, the equality $Z_{\widehat{G}(\C)}(\varphi, N)=U^N(\varphi)\rtimes Z_{\wh{G}(\C)}(\psi)$ holds.
\end{prop}
\begin{proof} Given Proposition \ref{prop:Zudec} it suffices to show that if $ua$ belongs to $Z_{\widehat{G}(\C)}(\varphi, N)$, where  $u$ is in $U^N(\C)$ and $a$ is in $Z_{\widehat{G}(\C)}(\theta)$, then in fact $u$ belongs to $U^N(\varphi)$ and $a$ belongs to $Z_{\wh{G}(\C)}(\psi)$. To prove this, we note that conjugation by an element in the image of $\varphi$ stabilizes both $U^N(\C)$ and $Z_{\widehat{G}(\C)}(\theta)$. Indeed, since $\Ad(\varphi(w))(N) = \|w\| N$, we have that conjugation by $\varphi(w)$ stabilizes $Z_{\widehat{G}(\C)}(N)$ and hence its unipotent radical $U^N$. On the other hand, as $\varphi(w)$ equals $\psi(w,1)\theta(i_2(w))$, and $\psi(w,1)$ commutes with $\theta$, one may easily check the claim that $\varphi(w)$ normalizes $Z_{\wh{G}(\C)}(\theta)$. Now for each $w \in W_F$, $ua$ equals $\Int(\varphi(w))(u)\Int(\varphi(w))(a)$. Therefore, $ \Int(\varphi(w))(a)a^{-1} $ equals $\Int(\varphi(w))(u)^{-1}u$. By what we have proven, the former is an element of $Z_{\widehat{G}(\C)}(\theta)$ and the latter is an element of $U^N(\C)$. Since $U^N(\C)$ and $Z_{\widehat{G}(\C)}(\theta)$ have trivial intersection, we have that both sides are trivial and so $a$ and $u$ commute with $\varphi(w)$ as desired. 
\end{proof}

We may use this decomposition to exhibit an example of a semi-simple $L$-parameter $\psi$ whose associated Weil--Deligne parameter has strictly larger centralizer. 

\begin{eg}
Let $G=\GL_3$ and consider the element $\psi$ in $\Phi^{L,\ss,\square}_G$ given by the following
\begin{equation*}
 \psi \left(w, \begin{pmatrix}
 a & b \\ 
 c & d 
 \end{pmatrix} \right) = 
 \left( 
 \begin{pmatrix}
 a & b & 0  \\ 
 c & d & 0  \\ 
 0 & 0 & 1  
 \end{pmatrix} 
 \begin{pmatrix}
 \| w \|^{-\frac{1}{2}} & 0 & 0 \\ 
 0 & \| w \|^{-\frac{1}{2}} & 0 \\ 
 0 & 0 & 1 
 \end{pmatrix} 
 , w 
 \right) . 
\end{equation*}
and set $(\varphi,N)=\JM(\psi)$. In this case, we have 
\[
 \fu^N =\left\{ 
 \begin{pmatrix}
 0 & * & *  \\ 
 0 & 0 & 0  \\ 
 0 & * & 0  
 \end{pmatrix} 
 \right\} . 
\]
Hence 
\[
 \begin{pmatrix}
 1 & 0 & 1  \\ 
 0 & 1 & 0  \\ 
 0 & 0 & 1  
 \end{pmatrix} \in 
 Z_{\wh{G} (\bC)}(\varphi,N) \cap U^N(\C) , 
\]
but it does not belong to 
$Z_{\wh{G} (\bC)}(\psi)$ 
by Proposition \ref{prop:Zphidec}.
\end{eg}

\begin{rem}
We remark that although $Z_{\widehat{G}(\C)}(\psi)$ need not equal $Z_{\widehat{G}(\C)}(\JM(\psi))$, these groups are the same for the purposes of parametrizing $L$-packets as in \cite{KalLLCnqs} as they have the same component groups by Proposition \ref{prop:Zphidec}. More generally, one can consider the group $S^{\natural}_{\psi}$ (resp.\@ $S^{\natural}_{\JM(\psi)}$) that is related to \cite[Conjecture F]{KalLLCnqs} and is defined by 
\begin{equation*}
    Z_{\wh{G}(\C)}(\psi)/[Z_{\wh{G}(\C)}(\psi)\cap \widehat{G}(\C)^{\der}]^{\circ},\qquad \bigg(\text{resp.}\,\,Z_{\wh{G}(\C)}(\JM(\psi))/[Z_{\wh{G}(\C)}(\JM(\psi)) \cap \widehat{G}(\C)^{\der}]^{\circ}\bigg).
\end{equation*} 
These groups are equal by Proposition \ref{prop:Zphidec} as $U^N(\varphi)$ is contained in $[Z_{\wh{G}(\C)}(\JM(\psi)) \cap \widehat{G}(\C)^{\der}]^{\circ}$.
\end{rem}

This decomposition also allows us to give an algebraic condition for when a Weil--Deligne parameter is the image under the Jacobson--Morozov map of a semi-simple $L$-parameter with the same centralizer.
In the rest of this section, we use Proposition \ref{prop:red-cent-ss}, but the proof of the proposition does not depend on the rest of this section. 

\begin{prop}\label{prop:red-cent-equiv} The group $Z_{\widehat{G}(\C)}(\varphi,N)^\circ$ is reductive if and only if $(\varphi,N)=\JM(\psi)$ for a Frobenius semi-simple Weil--Deligne parameter $\psi$ such that $Z_{\wh{G}(\C)}(\psi)=Z_{\wh{G}(\C)}(\varphi,N)$.
\end{prop}
\begin{proof} Suppose first that $Z_{\wh{G}(\C)}(\varphi,N)^\circ$ is reductive. We shall show in Proposition \ref{prop:red-cent-ss} that this implies that $(\varphi,N)$ is Frobenius semi-simple. Let $\psi$ be any element of $\Phi^{L,\ss,\square}_G$ such that $\JM(\psi)=(\varphi,N)$. By Proposition \ref{prop:Zphidec} the reductivity of $Z_{\wh{G}(\C)}(\varphi,N)^\circ$ implies that $U^N(\varphi)$ is trivial, and thus $Z_{\wh{G}(\C)}(\psi)=Z_{\wh{G}(\C)}(\varphi,N)$ as desired. Conversely, if $(\varphi,N)=\JM(\psi)$ for an element of $\Phi^{L,\ss,\square}_G$ and $Z_{\wh{G}(\C)}(\psi)=Z_{\wh{G}(\C)}(\varphi,N)$, then $Z_{\wh{G}(\C)}(\varphi,N)^\circ$ is reductive by \cite[Proposition 3.2]{SilbergerZink}
\end{proof}

Let $\Phi^{\WD,\mathrm{rc},\square}_G$ consist of those $(\varphi,N)$ with $Z_{\wh{G}(\C)}(\varphi,N)^\circ$ reductive. We call this the \emph{reductive centralizer locus} of $\Phi^{\WD,\square}_G$.

\begin{cor}\label{cor:JM-rd-bij-classical} The map $\JM\colon \JM^{-1}\left(\Phi^{\WD,\mathrm{rc},\square}_G\right)\to \Phi^{\WD,\mathrm{rc},\square}_G$ is a $\wh{G}(\C)$-equivariant bijection.
\end{cor}
\begin{proof}
This follows from Theorem \ref{thm:JM-params-classical}, Proposition \ref{prop:red-cent-equiv} and that $\psi$ is Frobenius semi-simple if and only if $\JM (\psi)$ is for $\psi \in \Phi^{L,\square}_G$. 
\end{proof}

\subsection{Essentially tempered parameters} To make Corollary \ref{cor:JM-rd-bij-classical} useful, we now show that $\JM^{-1}(\Phi^{\WD,\mathrm{rc},\square}_G)$ contains a large class of important $L$-parameters. To this end, let us call an element $\psi$ of $\Phi^{L,\square}_G$ \emph{essentially tempered} if the projection of $\psi(W_F)$ to $\wh{G}(\C)/Z_0(\wh{G})(\C)$ is relatively compact. Let $\Phi^{L,\mathrm{est},\square}_G$ be the set consisting of essentially tempered $L$-parameters. We will soon show that every essentially tempered $L$-parameter maps into the reductive centralizer locus, but first we must establish some results concerning Frobenius semi-simple parameters.

\begin{prop}\label{prop:ess-temp-Fss} Any element $\psi$ of $\Phi^{L,\mathrm{est},\square}_G$ is Frobenius semi-simple. 
\end{prop}
\begin{proof}
The map $\psi'$ obtained by composing $\psi|_{W_{F^\ast}}$ with the projection to $\wh{G}(\C)/Z_0(\wh{G})(\C)$ is a homomorphism. By Lemma \ref{lem:L-group-ss} below it suffices to show that if $w_0$ is an arithmetic Frobenius lift and $m$ is divisible by $[F^\ast:F]$, then $\psi'(w_0^m)$ is semi-simple. But, by essentially temperedness we know that the image of $\psi'(w_0^m)$ in $\wh{G}(\C)/Z_0(\wh{G})(\C)$ is contained in a maximal compact subgroup $K$ of $\wh{G}(\C)/Z_0(\wh{G})(\C)$. Up to conjugation, we may then assume that $K=H(\mathbb{R})$ for $H$ a compact form of $\wh{G}(\C)/Z_0(\wh{G})(\C)$ (see \cite[Theorem D.2.8, Proposition D.3.2, and Example D.3.3]{ConRgrsch}). But, as $H(\mathbb{R})$ consists only of semi-simple elements, the claim follows.
\end{proof}

\begin{lem}\label{lem:L-group-ss} Let $(s,w)$ be an element of $\LG(\C)$ and write $(s,w)^m= (s_m,w^m)$. Then, $(s,w)$ is Frobenius semi-simple if and only if $s_m$ is semi-simple for some non-zero integer $m$ divisible by $[F^\ast:F]$.
\end{lem}
\begin{proof} Fix any representation $r\colon {^L}G\to \mathrm{GL}_n$. As $r((s,w)^k)=r(s,w)^k$ we see that $(s,w)$ is semi-simple if and only if $(s,w)^k$ is for some $k> 0$. But, if $m$ is divisible by $[F^\ast:F]$ then as $r((s,w)^{mk})=r(s_m^k,1)$ for some $k>0$, the conclusion follows.
\end{proof}

The following shows that the naming of essentially tempered $L$-parameters is reasonable. 

\begin{prop}\label{prop:ess-temp-twist} For $\psi \in \Phi^{L,\square}_G$, 
the following conditions are equivalent:
\begin{enumerate}
    \item $\psi \in \Phi^{L,\mathrm{est},\square}_G$,
    \item there is a continuous character $\chi \colon W_F\times \SL_2(\C) \to Z_0(\wh{G})(\C)$ such that the projection of $(\chi \psi)(W_F)$ to $\wh{G}(\C)$ is relatively compact. 
\end{enumerate}
\end{prop}
\begin{proof}
It is clear that (2) implies (1). We show that (1) implies (2). Fix a Frobenius lift $w_0 \in W_F$. Set $H=Z_{\widehat{G}(\C)}(\psi)$, which has reductive identity component by Proposition \ref{prop:ess-temp-Fss} and \cite[Proposition 3.2]{SilbergerZink}. Let $\wh{\psi}$ be the $\wh{G}$-component of $\psi$. Taking a positive integer $m$ to be divisible by $|\Aut (\psi(I_F))|$ and $[F^\ast:F]$ we see that $\wh{\psi}(w_0^m) \in H$, and thus in fact $\wh{\psi}(w_0^m) \in Z(H)$. By replacing $m$ by a power, we may assume that  $\wh{\psi}(w_0^m) \in Z(H)^{\circ}$. Since $\psi \in \Phi^{L,\mathrm{est},\square}_G$, there is a compact subgroup $C \subseteq Z(H)^{\circ}$ such that $\wh{\psi}(w_0^m) \in C \cdot (Z(H)^{\circ} \cap Z(\widehat{G})(\C))$. We write $\wh{\psi}(w_0^m)=c z$ for $c \in C$ and $z \in Z(H)^{\circ} \cap Z(\widehat{G})(\C)$. Since elements of $Z(H)^{\circ} \cap Z(\widehat{G})(\C)$ commute with $\psi(W_F)$, we have $Z(H)^{\circ} \cap Z(\widehat{G})(\C)=Z(H)^{\circ} \cap Z_0(\wh{G})(\C)$. Replacing $m$ again, we may assume that $z \in (Z(H)^{\circ} \cap Z_0(\wh{G})(\C))^\circ$. We take $z_0 \in (Z(H)^{\circ} \cap Z_0(\widehat{G})(\C))^\circ$ such that $z_0^m =z$, which exists as $(Z(H)^\circ\cap Z_0(\widehat{G})(\mathbb{C}))^\circ$ is a torus since it is a connected algebraic subgroup of  $Z_0(\widehat{G})^\circ$, which is a torus. Further we define $\chi$ as the unramified character sending $w_0$ to $z_0^{-1}$. Then the image of $(\chi \psi)(W_F)$ in $\wh{G}(\C)$ is contained in the 
image of $\bigcup_{i=0}^{m-1} \psi (I_F) (\chi\psi)(w_0^i) C$ in $\wh{G}(\C)$, which is compact. 
\end{proof}

We now relate $\Phi^{L,\mathrm{est},\square}_G$ to the reductive centralizer locus of $\Phi^{\WD,\square}_G$. 

\begin{prop}\label{prop:temp-cent-equal} The containment $\Phi^{L,\mathrm{est},\square}_G\subseteq \JM^{-1}(\Phi^{\WD,\mathrm{rc},\square}_G)$ holds.
\end{prop}
\begin{proof} Let $\psi$ be an element of $\Phi^{L,\mathrm{est},\square}_G$ and set $(\varphi,N)=\JM(\psi)$. Then $\psi$ is Frobenius semi-simple by Proposition \ref{prop:ess-temp-Fss}. We claim that $Z_{\wh{G}(\C)}(\psi)=Z_{\wh{G}(\C)}(\varphi,N)$, from where we will be done by Proposition \ref{prop:red-cent-equiv}. By Proposition \ref{prop:Zphidec}, it suffices to show that $U^N(\varphi)$ is trivial. 
We assume that $U^N(\varphi)$ is non-trivial 
and take a non-trivial weight vector $v$ of 
$\Lie (U^N(\varphi))$ with respect to the adjoint action of $\theta|_{T_2}$, where $T_2$ is the standard maximal torus of $\SL_{2,\C}$. We put $u= \exp(v)$. For each $w \in W_F$ we have that $\varphi(w)= \psi(w,1)\theta(i_2(w))$. Since $\varphi(w)$ commutes with $u$, we see that $\Int(\psi(w,1)^{-1})(u)$ is equal to $\Int(\theta(i_2(w)))(u)$, and therefore
\begin{equation*}
   \Ad(\psi(w,1)^{-1})(v) =  \Ad(\theta(i_2(w)))(v). 
\end{equation*}
But, observe that if $w_0$ is a lift of arithmetic Frobenius in $W_F$ then $i_2(w_0^{2n})=\left(\begin{smallmatrix}q^n & 0\\ 0 & q^{-n}\end{smallmatrix}\right)$. By Proposition \ref{prop:Zudec}, we deduce that $\Ad(\theta(i_2(w_0^{2n})))(v)=q^{jn} v$ for some $j\geqslant 1$. Letting $n$ tend towards infinity, and using the fact that $u$ is non-trivial, we deduce that the adjoint orbit of $W_F$ on $v$ is non-compact, which is a contradiction.
\end{proof}

We now state a corollary to Proposition \ref{prop:temp-cent-equal}. Before doing so, we recall an even smaller subset of $\Phi^{L,\mathrm{est},\square}_G$ that will feature prominently below. Namely, recall that $(\varphi,N)$ in $\Phi^{\WD,\square}_G$ (resp.\@ $\psi$ in $\Phi^{L,\square}_G$) is called \emph{discrete} if the quotient
\begin{equation*}
    Z_{\widehat{G}(\C)}(\varphi,N)/Z_0(\widehat{G})(\C)\qquad \bigg(\text{resp.}\,\,Z_{\widehat{G}(\C)}(\psi)/Z_0(\widehat{G})(\C)\bigg)
\end{equation*}
is finite. Denote by $\Phi^{\WD,\disc,\square}_G$ (resp.\@ $\Phi^{L,\disc,\square}_G$) the set of discrete parameters and $\Phi^{\WD,\disc}_G$ (resp.\@ $\Phi^{L,\disc}_G$) its $\wh{G}(\C)$-quotient. Note that $\Phi^{L,\disc,\square}_G$ is contained in $\Phi^{L,\mathrm{est},\square}_G$ (cf.\@ \cite[Lemma 3.1]{GRAinv} and \cite[Lemma 5.2]{SilbergerZink}), and thus $\psi$ is discrete if and only if $\JM(\psi)$ discrete as they have the same centralizers by Proposition \ref{prop:temp-cent-equal} and its proof.

\begin{cor}\label{cor:bij-et-disc} The map
\begin{equation*}
    \JM\colon \Phi^{L,\mathrm{est},\square}_G\to \Phi^{\WD,\square}_G,\qquad \bigg(\text{resp.}\,\, \JM\colon \Phi^{L,\disc,\square}_G\to\Phi^{\WD,\disc,\square}_G\bigg)
\end{equation*}
is a $\wh{G}(\C)$-equivariant injection (resp.\@ bijection).
\end{cor}

Note that implicit in the above is the following result of independent interest.

\begin{prop}\label{prop:disc-Frob-ss-classical}
Any element of $\Phi^{\WD,\disc,\square}_G$ (resp.\@ $\Phi^{L,\disc,\square}_G$) is Frobenius semi-simple. 
\end{prop}
\begin{proof}
The first claim is a special case of Proposition \ref{prop:red-cent-ss}. 
The second claim follows from 
$\Phi^{L,\disc,\square}_G \subseteq \Phi^{L,\mathrm{est},\square}_G$ and Proposition \ref{prop:ess-temp-Fss}. 
\end{proof}

We end this subsection by showing that one may apply Corollary \ref{cor:bij-et-disc} to show that the association of $\psi\circ i$ to $\psi$ is injective when restricted to the set of discrete $L$-parameters. This result plays an important technical role in \cite{Characterization}.

\begin{prop}\label{prop:disc-ss-prop} The maps
\begin{equation*}
    \Phi^{\WD,\disc}_G\xrightarrow{(\varphi,N)\mapsto \varphi} \Hom(W_F,\LG(\C))/\wh{G}(\C),\qquad  \Phi^{L,\disc}_G\xrightarrow{\psi\mapsto \psi\circ i} \Hom(W_F,\LG(\C))/\wh{G}(\C)
\end{equation*}
are injective.
\end{prop}
\begin{proof} By Corollary \ref{cor:bij-et-disc} it suffices to show that the former map is injective. Fix $\lambda$ in the set $\Hom(W_F,\LG(\C))$. By Proposition \ref{prop:disc-Frob-ss-classical} it then suffices to show that (if non-empty) the set
\begin{equation*}
    P(G,\lambda)\defeq\left\{(\varphi,N)\in \Phi^{\WD,\ss,\square}_G: \varphi=\lambda\right\}
\end{equation*}
intersects at most one $\wh{G}(\C)$-orbit of discrete parameters. As in \cite[\S4]{VoganLL}, set $\wh{G}(\C)^\lambda$ to be $Z_{\wh{G}(\C)}(\lambda)$, and
\begin{equation*}
    \widehat{\mf{g}}^{\,\lambda(I_F)}_q\defeq\left\{x\in \widehat{\mf{g}}_\C:
    \begin{aligned}(1)&\quad \mathrm{Ad}(\lambda(w))(x)=x\text{ for all }w\in I_F\\ (2)&\quad \mathrm{Ad}(w_0)(x)=qx\end{aligned}\right\}
\end{equation*}
where $w_0$ is any lift of arithmetic Frobenius. Both $P(G,\lambda)$ and $\widehat{\mf{g}}^{\,\lambda(I_F)}_q$ carry an action of $\widehat{G}(\C)^\lambda$, and \cite[Proposition 4.5]{VoganLL} establishes a $\widehat{G}(\C)^\lambda$-equivariant bijection $P(G,\lambda)\to \widehat{\mf{g}}^{\lambda(I_F)}_{q}$, and that the latter space has only finitely many orbits. Therefore, $P(G,\lambda)$ carries the structure of a vector space on which $\widehat{G}(\C)^\lambda$ acts algebraically and with only finitely many orbits.

Suppose then that $(\lambda,N)$ is a discrete element of $P(G,\lambda)$ and let $\mc{O}\subseteq P(G,\lambda)$ denote its $\widehat{G}(\C)^\lambda$-orbit. Now, $\mc{O}$ is a locally closed subscheme of $P(G,\lambda)$ (see \cite[Proposition 1.65 (2)]{MilneGroups}) of dimension $\dim(\widehat{G}(\C)^\lambda)-\dim(H)$ where $H$ is the isotropy subgroup of $(\lambda,N)$ in $\widehat{G}(\C)^\lambda$ (\cite[Proposition 5.23 and Proposition 7.12]{MilneGroups}). But, note that $H=Z_{\widehat{G}(\C)}(\lambda,N)$ and so contains $Z_0(\wh{G})(\C)$ as a finite index subgroup. We deduce that $\dim(\mc{O})$ is equal to $\dim(\widehat{G}(\C)^\lambda)-\dim(Z_0(\widehat{G})(\C))$. But, as $\wh{G}(\C)^\lambda$ acts through $\wh{G}(\C)^\lambda/Z_0(\wh{G})(\C)$, and has finitely many (locally closed) orbits, we see that $\dim P(G,\lambda)$ is at most $\dim(\widehat{G}(\C)^\lambda)-\dim(Z_0(\widehat{G})(\C))$. Thus, we deduce that $\dim(\mc{O})=\dim(P(G,\lambda))$. As $\mc{O}$ is locally closed in $P(G,\lambda)$ we deduce that $\mc{O}$ is open. As $P(G,\lambda)$ is a vector space it is irreducible, so open orbits are unique, and the conclusion follows.
\end{proof}

\section{The geometric and relative Jacobson--Morozov theorems}

Before we can geometrize the Jacobson--Morozov theorem for parameters, we now first geometrize the Jacobson--Morozov theorem. After doing so, we derive a version of the Jacobson--Morozov on the level of $A$-points. We fix for the remainder of this section a field $k$ of characteristic $0$ and $H$ a reductive group over $k$.

\begin{rem} In this section we often assume that $H$ is split. This will be sufficient for us as $\wh{G}$ is a split group. Most of these statements admit obvious generalizations to arbitrary reductive $H$, with similar proofs. The exception is Theorem \ref{thm:relative-jm}, but we suspect that the statement is still true and that one can employ a similar strategy to prove it.
\end{rem}

\subsection{The orbit separation space} Pivotal to our formulation of a geometric version of the Jacobson--Morozov theorem is a certain construction which, in a precise sense, replaces a variety with group action with the disjoint union of its orbits. Throughout this subsection we fix a reduced quasi-projective scheme $X$ over $k$ equipped with an action of $H$. We also assume that the map
\begin{equation*}
    X(k)/H(k)\to X(\ov{k})/H(\ov{k})
\end{equation*}
is surjective (although one may deal with the general case by Galois descent). Whenever we speak of the class of $x$ in $X(\ov{k})/H(\ov{k})$ we assume without loss of generality that $x$ is in $X(k)$.

For each element $x$ of $X(k)$ let us denote by $\mathcal{O}_x$ the \emph{orbit scheme} given as the fppf sheafification of the presheaf 
\begin{equation*}
    \cat{Alg}_{k}\to\cat{Set},\qquad A\mapsto \left\{g\cdot x: g\in H(A)\right\}\subseteq X(A).
\end{equation*}
Since $X$ is itself an fppf sheaf, we see that $\h_x$ is an $H$-stable subsheaf of $X$. There is a natural map $\mu_x\colon H\to \mc{O}_x$ called the \emph{orbit map} given on $R$-points by sending $h\in H(R)$ to $h\cdot x\in \mc{O}_x(R)$.

\begin{prop}\label{prop:orbit-map} The orbit scheme is representable by a reduced locally closed subscheme of $X$ smooth over $k$. Moreover, the orbit map $\mu_x\colon H\to \mc{O}_x$ is smooth and surjective and identifies $\mc{O}_x$ as the fppf sheaf quotient $H/Z_H(x)$.
\end{prop} 
\begin{proof} Clearly the orbit map identifies $\mc{O}_x$ as the fppf sheaf quotient $H/Z_H(x)$. In \cite[Proposition 1.65]{MilneGroups} it is shown that $\mu_x(H)$ is a locally closed subset of $H$, which one may endow with the reduced scheme structure. In \cite[Proposition 7.17]{MilneGroups} it is shown that $\mu_x(H)$ represents $\mathcal{O}_x$. The smoothness of the orbit map is then confirmed by \cite[Proposition 7.15]{MilneGroups}, and the smoothness of $\mc{O}_x$ over $k$ is handled by \cite[Corollary 5.26]{MilneGroups}.
\end{proof}

It will be useful to have a more explicit description of the $A$-points of $\mc{O}_x$ for a $k$-algebra $A$.

\begin{prop}\label{prop:orbit-desc-gen} For any $k$-algebra $A$, there are identifications
\begin{equation*}
    \mc{O}_x(A)=\left\{\mathbf{x}\in X(A):x\emph{ and }\mathbf{x}\emph{ lie in the same }H(A)\emph{-orbit \'etale locally on }A\right\},
\end{equation*}
and
\begin{equation*}
    \mc{O}_x(A)/H(A)=\ker\bigg(H^1_\et(\Spec(A),Z_{H}(x))\to H^1_\et(\Spec(A),H)\bigg),
\end{equation*}
functorial in $(X,x)$, where a map $(X,x)\to (Y,y)$ is an $H$-equivariant map $X\to Y$ sending $x$ to $y$.
\end{prop}
\begin{proof} The first claim follows from the fact that the orbit map $\mu_x\colon H\to \mc{O}_x$ is a smooth surjection and \cite[Corollaire 17.16.3.(ii)]{EGA4-4}. The second claim follows by combining  \cite[Chaptire III, Corollaire 3.2.3]{Giraud} with the fact that as $H_A$ and $Z_H(x)_A$ are smooth over $A$, their \'etale cohomology functorially agrees with their fppf cohomology (cf.\@ \cite[Th\'{e}or\`{e}me 11.7]{GrothendieckBrauerIII}).
\end{proof}

When $A$ is a reduced $k$ algebra, one may give a simpler description. Say an element $\mb{x}$ of $X(A)$ is \emph{everywhere geometrically conjugate (egc)} to $x$ if for all geometric points $\Spec(k')\to\Spec(A)$ one has that $x$ and $\mb{x}$ have images in $X(k')$ belonging to the same $H(k')$-orbit.

\begin{prop}\label{prop:ever-geom-conj} For a reduced $k$-algebra $A$ there is a functorial identification 
\begin{equation*}
    \h_x(A)=\left\{\mb{x}\in X(A):\mb{x}\emph{ is egc to }x\right\}.
\end{equation*}
\end{prop}
\begin{proof} Evidently any element of $\h_x(A)$ is egc to $x$. If $\mb{x}$ is egc to $x$ then the morphism $\mb{x}\colon \Spec(A)\to X$ has the property that $\mb{x}(|\Spec(A)|)\subseteq |\h_x|$. As $\Spec(A)$ is reduced this implies that $\mb{x}$ factorizes through $\h_x$ as desired.
\end{proof}

We then assemble the spaces $\mc{O}_x$ into one as follows.

\begin{defn} We define the \emph{orbit separation} of $X$, denoted by $X^\sqcup$, to be the space 
\begin{equation*}
    X^\sqcup\defeq \bigsqcup_{x\in X(\ov{k})/H(\ov{k})}\mc{O}_x.
\end{equation*}
\end{defn}

We have a tautological $H$-equivariant map $X^\sqcup\to X$. Observe that the orbit separation space is a functorial construction. Namely, if $Y$ is another quasi-projective scheme over $k$ equipped with an action of $H$ with the same properties, then  for any $H$-equivariant morphism $X\to Y$, the composition $X^\sqcup\to X\to Y$ factorizes uniquely through $Y^\sqcup\to Y$.

We end this section with the following omnibus result concerning its properties in the case when $X(\ov{k})/H(\ov{k})$ is finite, which is the case of most interest to us. Below, and in the sequel, we call a morphism of schemes $f\colon Y\to X$ \emph{weakly birational} if there exists a dense open subset $U$ of $X$ such that $f^{-1}(U)\to U$ is an isomorphism. 

\begin{prop}\label{prop:sqcup-omnibus} The following statements are true.
\begin{enumerate}
    \item The map $X^\sqcup\to X$ is an isomorphism if and only if the map
    \begin{equation*} \varrho\colon H\times X\to X\times X,\qquad (h,x)\mapsto (hx,x),
    \end{equation*}
    is smooth.
    \item In general, if $X(\ov{k})/H(\ov{k})$ is finite then the map $X^\sqcup\to X$ is a weakly birational surjective monomorphism.
\end{enumerate}
\end{prop}
\begin{proof} To prove (1) suppose first that $\varrho$ is smooth, then for any $x_0$ in $X(\ov{k})$ the pullback of $\varrho$ along the map $X\to X\times X$ given by $x\mapsto (x,x_0)$ is also smooth. This pullback map can be identified with the composition of $\mu_{x_0}\colon H\to \mc{O}_{x_0}$ with the inclusion $\mc{O}_{x_0}\to X$. Thus, $\mc{O}_{x_0}$, being the image of a smooth morphism, is open by \stacks{056G}. Since $x_0$ was arbitrary we conclude that $X^\sqcup\to X$ is an isomorphism by applying part (3) of Lemma \ref{lem:stratification-isom} below with the set $\{Y_i\}=\{\mc{O}_{x_0}\}_{x_0\in X(\ov{k})/H(\ov{k})}$. Conversely, suppose that $X^\sqcup\to X$ is an isomorphism. Then, it is easy to see that $\varrho$ is smooth if and only if for all $x_0$ in $X(\ov{k})$ the map 
\begin{equation*}
    \varrho_{x_0}\colon H\times \mc{O}_{x_0}\to\mc{O}_{x_0}\times\mc{O}_{x_0},\qquad (h,x)\mapsto (hx,x)
\end{equation*}
is smooth. By Proposition \ref{prop:orbit-map} and \stacks{0429} it suffices to prove that the pullback of $\varrho_{x_0}$ along the map $\mu_{x_0}\times \mu_{x_0}\colon H\times H\to \mc{O}_{x_0}\times \mc{O}_{x_0}$ is smooth. This pullback, as an $H\times H$-scheme, may be identified with the map 
\begin{equation*}
    H_{x_0}\times H\times H\to H\times H,\qquad (h,h',h'')\mapsto (h',h'')
\end{equation*}
which is evidently smooth. Explicitly, this identification is given by the map
\begin{equation*}
    (H\times\mathcal{O}_{x_0})\times_{\mathcal{O}_{x_0}\times\mathcal{O}_{x_0}}(H\times H)\to H_{x_0}\times H\times H,\quad ((h,x),(h',h''))\mapsto ((h')^{-1}hh'',h',h''),
\end{equation*}
which is easily verified to be an isomorphism of $H\times H$-schemes.

To prove (2) observe that the map the map $X^\sqcup(\ov{k})\to X(\ov{k})$ is surjective, and so by applying Lemma \ref{lem:stratification-isom} parts (1) and (2) and \cite[Corollary 3.36]{GortzWedhorn} we know that $X^\sqcup\to X$ is a weakly birational monomorphism. As the image of $X^\sqcup\to X$ is a finite union of locally closed subsets it is locally closed and as it contains $X(\ov{k})$ it must be all of $X$ by loc.\@ cit.\@, and so surjectivity also follows.
\end{proof}

\begin{lem}\label{lem:stratification-isom} Let $f\colon Y\to X$ be a morphism of reduced schemes locally of finite type over $k$, with $X$ quasi-compact. Suppose that $Y_{\ov{k}}$ admits a decomposition $Y=\bigsqcup_i Y_i$, with each $Y_i$ clopen, such that $f|_{Y_i}$ is a locally closed immersion, and $f(Y_i(\ov{k}))\cap f(Y_j(\ov{k}))$ is empty for $i\ne j$. Then,
\begin{enumerate}
    \item $f$ is a monomorphism,
    \item if $\{Y_i\}$ is a finite set, then $f$ is weakly birational if and only if $f(Y(\ov{k}))$ is dense in $X$,
    \item $f$ is an isomorphism if and only if $f(Y(\ov{k}))=X(\ov{k})$ and each $Y_i$ is open in $X_{\ov{k}}$.
\end{enumerate}
\end{lem}
\begin{proof} As all of these claims may be checked over $\ov{k}$ we may assume without loss of generality that $k$ is algebraically closed. The final claim is clear, thus we focus on the first two claims. For the first claim, as each $f|_{Y_i}$ is a monomorphism, it suffices to show that $f(Y_i)$ and $f(Y_j)$ are disjoint for $i\ne j$. But, as $f(Y_i)\cap f(Y_j)$ is locally closed, if non-empty it would contain a $k$-point which is a contradiction.

To see the second claim, it suffices to show the if direction. For each irreducible component $Z$ of $X$ note that $\{Y_i\cap Z\}$ is a finite set of locally closed subsets with dense union. This implies that there exists some $i_0$ such that $Y_{i_0}\cap Z$ is open.  Let $C$ be the union of irreducible components of $X$ which intersect $Z$ at a proper non-empty subset. Set $U_Z\defeq (Y_{i_0}\cap Z)-C$. Then, it is clear that if $U$ is the union of the $U_Z$, then $U$ is a dense open subset of $X$ and as $X$ is reduced that $f\colon f^{-1}(U)\to U$ is an isomorphism.
\end{proof}

\subsection{The geometric Jacobson--Morozov theorem}

We now move to the geometrization of the Jacobson--Morozov theorem. Let us now assume that $H$ is split. To begin, observe that one has a \emph{Jacobson--Morozov morphism}
\begin{equation*}
    \JM\colon \underline{\Hom}(\SL_{2,k},H)\to \mc{N},\qquad \theta\mapsto d\theta(e_0).
\end{equation*}
We would like to apply the orbit separation construction from the last subsection to this map, but before we do so, we should first observe that the actions of $H$ on $\underline{\Hom}(\SL_{2,k},H)$ and $\mc{N}$ satisfy the properties used in the last section.

\begin{prop}\label{prop:split-nilp-desc} The maps 
\begin{equation*}
    \mc{N}(k)/H(k)\to \mc{N}(\ov{k})/H(\ov{k}),\qquad \Hom(\SL_{2,k},H)/H(k)\to \Hom(\SL_{2,\ov{k}},H_{\ov{k}})/H(\ov{k})
\end{equation*} 
are surjections.
\end{prop}
\begin{proof} By Theorem \ref{thm:JM-classical} it suffices to show the first map is a surjection. Let $N$ be an element of $\mc{N}(\ov{k})$. Bala--Carter theory (see \cite[Proposition 4.7 and Theorem 4.13]{Jantzen}) says that there exists a Levi subgroup $\ov{L}$ of $H_{\ov{k}}$ and a parabolic subgroup $\ov{P}$ of $\ov{L}$ such that $N$ is conjugate to an element contained in the unique open orbit of $\ov{P}$ acting on $\Lie(R_u(\ov{P}))$. Now, as $H$ is split, we may assume up to conjugacy, that $\ov{L}=L_{\ov{k}}$ for a Levi subgroup $L$ of $H$ (see \cite{Solleveld}). As $L$ is also split we may also assume, up to conjugacy, that $\ov{P}=P_{\ov{k}}$ for a parabolic subgroup $P$ of $L$. As the unique open orbit of $P$ acting on $\Lie(R_u(P))$ has a $k$-point, being a Zariski open of a vector $k$-space, we are done.
\end{proof}

\begin{rem} The morphism $\mc{N}(k)/H(k)\to \mc{N}(\ov{k})/H(\ov{k})$ is rarely injective. As a concrete example, if $H=\SL_{2,\Q}$ then $\left(\begin{smallmatrix}0 & 1\\ 0 & 0\end{smallmatrix}\right)$ and $\left(\begin{smallmatrix}0 & -1\\ 0 & 0\end{smallmatrix}\right)$ are $H(\ov{\Q})$-conjugate, but not $H(\Q)$-conjugate.
\end{rem}

Before we show that our two spaces with $H$-action have finitely many $H(\ov{k})$-orbits, we observe the following.

\begin{prop}\label{prop:SL2-Hom-open-orbits} The morphism $\underline{\Hom}(\SL_{2,k},H)^\sqcup\to\underline{\Hom}(\SL_{2,k},H)$ is an isomorphism.
\end{prop}
\begin{proof} This follows immediately from combining Proposition \ref{prop:hom-schem-omnibus} and Proposition \ref{prop:sqcup-omnibus}.
\end{proof}

\begin{prop} The sets $\Hom(\SL_{2,\ov{k}},H_{\ov{k}})/H(\ov{k})$ and $\mc{N}(\ov{k})/H(\ov{k})$ are finite.
\end{prop}
\begin{proof} By Theorem \ref{thm:JM-classical} these two sets are in bijection, so it suffices to prove the finiteness of either. The finiteness of the latter set is a classical result (e.g.\@ see \cite[\S2.8, Theorem 1]{Jantzen}). Alternatively, one may prove the finiteness of the former set by observing that by Proposition \ref{prop:SL2-Hom-open-orbits} the sets $\Hom(\SL_{2,\ov{k}},H_{\ov{k}})/H(\ov{k})$ and $\pi_0(\underline{\Hom}(\SL_{2,{\ov{k}}},H_{\ov{k}}))$ have the same cardinality. But, by Proposition \ref{prop:hom-schem-omnibus} the scheme $\underline{\Hom}(\SL_{2,{\ov{k}}},H_{\ov{k}})$ is finite type over $k$ and thus  $\pi_0(\underline{\Hom}(\SL_{2,{\ov{k}}},H_{\ov{k}}))$ is finite.
\end{proof}

By the functoriality of the orbit separation construction the Jacobson--Morozov morphism factors uniquely through $\mc{N}^\sqcup$ and we also denote the resulting map $\underline{\Hom}(\SL_{2,k},H)\to \mc{N}^\sqcup$ by $\JM$. But, unlike $\underline{\Hom}(\SL_{2,k},H)$, the orbit separation space $\mc{N}^\sqcup$ is essentially never equal to $\mc{N}$.

\begin{prop} The morphism $\mc{N}^\sqcup\to\mc{N}$ is an isomorphism if and only if $H$ is abelian.
\end{prop}
\begin{proof} If $H$ is abelian then $\mc{N}$ is a single point. If $\mc{N}^\sqcup\to\mc{N}$ is an isomorphism then the orbit of $0$ is open, but as it is also closed and $\mc{N}$ is connected we deduce that it is equal to $\mc{N}$. As $\dim(\mc{N})$ is equal to $\dim(H)-r(H)$, we see that $H$ is a torus as desired.
\end{proof}

\begin{eg}\label{eg:naive-JM-fails} The element $\mathbf{N}=\left(\begin{smallmatrix}0 & t\\ 0 & 0\end{smallmatrix}\right)$ defines a point of $\mc{N}_{\GL_{2,k}}(k[t])$ not in $\mc{N}_{\GL_{2,k}}^\sqcup(k[t])$. 
\end{eg}

To state our geometric Jacobson--Morozov theorem, note that by Theorem \ref{thm:JM-classical} the map 
\begin{equation*}
    \JM\colon \Hom(\SL_{2,k},H)/H(k)\to \mc{N}(k)/H(k),
\end{equation*}
is a bijection. For each $\theta$, writing $N=\JM(\theta)$, define $\JM_\theta$ to be the map $\mc{O}_\theta \to\mc{O}_N$ which may be described as the quotient map $H/Z_H(\theta)\to H/Z_H(N)$.

\begin{thm}[Geometric Jacobson--Morozov]\label{thm:geom-JM-split} Suppose that $H$ is split. The morphism $\JM\colon  \underline{\Hom}(\SL_{2,k},H)\to \mc{N}$ factorizes through $\mc{N}^\sqcup$, where it may be described as $\bigsqcup_\theta \JM_\theta$.
\end{thm}

\subsection{The relative Jacobson--Morozov theorem} 

We now apply the geometric Jacobson--Morozov theorem to obtain a more concrete result on the level of $A$-points.

\begin{thm}[Relative Jacobson--Morozov]\label{thm:relative-jm} Let $A$ be a $k$-algebra. Then, the map 
\begin{equation*}
    \JM\colon \Hom(\SL_{2,A},H_A)/H(A)\to \mc{N}(A)/H(A)
\end{equation*}
is a bijection onto $\mc{N}^\sqcup(A)/H(A)$.
\end{thm}
\begin{proof} Assume first that $\Spec(A)$ is connected. By Theorem \ref{thm:geom-JM-split}, it suffices to show that for each $\theta$ the map $\JM_\theta$ induces a bijection $\mathcal{O}_\theta(A)/H(A)\to \mathcal{O}_N(A)/H(A)$. But, by Proposition \ref{prop:orbit-desc-gen} it suffices to show that the natural map $H^1_\et(\Spec(A),Z_H(\theta))\to H^1_\et(\Spec(A),Z_H(N))$ is a bijection. But, this follows from Proposition \ref{prop:Zudec} and \cite[Lemma 4.14]{GillePianzola}. For the general case we reduce to the Noetherian case by standard approximation arguments, and then working on each component to the case when $\Spec(A)$ is connected.
\end{proof}

We now pursue the analogue of Theorem \ref{thm:rel-JM-triples-classical} in the relative setting. 

\begin{defn} Let $A$ be a $k$-algebra and $\mf{a}$ a Lie algebra over $A$. We call a triple of elements $(e,h,f)$ in $\mf{a}^3$ such that
\begin{equation*}
    [h,e]=2e,\quad [h,f]=-2f,\quad [e,f]=h,
\end{equation*}
an \emph{$\mf{sl}_2$-triple} in $\mf{a}$. 
\end{defn}

Denote by $\mc{T}(A)$ (or $\mc{T}_H(A)$ when we want to emphasize $H$) the set of $\mf{sl}_2$-triples in $\mf{h}_A$. Evidently $\mc{T}(A)$ carries a natural conjugation action by $H(A)$. 

\begin{thm}\label{thm:rel-JM-triples} The following diagram is commutative and each arrow is a bijection 
\begin{equation*}
    \xymatrixrowsep{3pc}\xymatrixcolsep{5pc}\xymatrix{\Hom(\SL_{2,A},H_A)/H(A)\ar[r]^{\theta\,\longmapsto \,d\theta}\ar[d]^{\JM} & \Hom(\mf{sl}_{2,A},\mf{h}_A)/H(A)\ar[d]^{\nu\mapsto (\nu(e_0),\nu(h_0),\nu(f_0))}\\ \mc{N}^\sqcup(A)/H(A) & \mathcal{T}(A)/H(A).
    \ar[l]_{e\,\longmapsfrom \,(e,h,f)}}
\end{equation*}
\end{thm}
\begin{proof} By Theorem \ref{thm:relative-jm} the left vertical arrow is a bijection. The right vertical arrow is clearly a bijection, and the top horizontal arrow is a bijection by Proposition \ref{prop:hom-schem-omnibus}. We thus deduce that the bottom horizontal arrow is well-defined (i.e.\@ takes values in $\mc{N}^\sqcup(A)$) and is bijective.
\end{proof}

\subsection{A relative version of Kostant's characterization of $\mf{sl}_2$-triples}

This final subsection is dedicated to giving a proof of the following relative version of \cite[Corollary 3.5]{KostdsB}.

\begin{prop}\label{prop:Kostant-triples-prop} Let $A$ be a $k$-algebra and $\mf{a}$ a Lie subalgebra of $\mf{h}_A$. Then, for a pair $(e,h)$ in $\mf{a}^2$, there exists an $\mf{sl}_2$-triple of the form $(e,h,f)$ in $\mf{a}$ if and only if the following conditions hold:
\begin{enumerate}
    \item $e\in\mc{N}^\sqcup(A)$,
    \item $h$ is in the image of $\mathrm{ad}(e)\colon \mf{a}\to \mf{a}$,
    \item $[h,e]=2e$.
\end{enumerate}
\end{prop}

Let us set
\begin{equation*}
    \mf{h}^e_A\defeq\ker\left(\ad(e)|\mf{h}_A\to\mf{h}_A\right),\qquad \mf{a}^e\defeq\ker\left(\ad(e)|\mf{a}\to\mf{a}\right).
\end{equation*}
If $\ad(e)(x)$ is zero then $\ad(e)(\ad(h)(x))$ is also zero. Thus, $\ad(h)$ stabilizes $\mf{h}_A^e$ and $\mf{a}^e$.

\begin{lem}\label{lem:ad(x)+2-invertible}The $A$-linear map $\ad(h)+2\colon \mf{a}^e\to\mf{a}^e$ is an isomorphism.
\end{lem}
\begin{proof} It suffices to show this result after passing to an etale cover $\Spec(B)\to\Spec(A)$. Indeed, since $A\to B$ is faithfully flat we have that $(\mf{a}^e)_B=\mf{a}_B^e$, and moreover that $\ker(\ad(h)+2)$ and $\mathrm{coker}(\ad(h)+2)$ are trivial if and only if they are so after tensoring with $B$. Thus, we may assume without loss of generality that $e$ is an element of $\mc{N}(k)$. Indeed, the statement of the lemma is insensitive to conjugating the pair $(e,h)$, and so this follows by the definition of $\mc{N}^\sqcup$ and Proposition \ref{prop:split-nilp-desc}. With notation as in Lemma \ref{lem:eigenvalues-ad(x)+2} below, the $A$-algebra map $A[T]\to \End_A(\mf{a}^e)$ sending $T$ to $\ad(h)$ factorizes through $A[T]/(p(T))$. But, by the Chinese remainder theorem $T+2$ is a unit in this ring.
\end{proof}

\begin{lem}[{cf.\@ \cite[Lemma 3.4]{KostdsB}}]\label{lem:eigenvalues-ad(x)+2} Suppose that $e$ is an element of $\mc{N}(k)$. Let $m$ be the smallest element such that $\ad(e)^{m+1}$ is trivial on $\mf{h}$. Then, $p(\ad(h)|_{\mf{h}^e_A})=0$ where 
\begin{equation*}
    p(T)=\prod_{i=0}^m\left(T-i \right).
\end{equation*}
Thus, a fortiori, we see that $p\left(\ad(h)|_{\mf{a}^e}\right)=0$.
\end{lem}
\begin{proof} For each $i=0,\ldots,m+1$ let us set 
\begin{equation*}
    \mf{d}_i\defeq(\ad(e)^i(\mf{h})\cap \mf{h}^{e})\otimes_k A.
\end{equation*}
Observe that 
\begin{equation*}
    \mf{h}^{e}_A=\mf{d}_0\supseteq\cdots\supseteq \mf{d}_{m+1}=0.
\end{equation*}
We claim then that $(\ad(h)-i)(\mf{d}_i)\subseteq \mf{d}_{i+1}$. Note that $\mf{d}_i$ is generated as an $A$-algebra by elements of the form $\ad(e)^i(z)$ for $z$ in $\mf{h}$. The exact same algebra as in \cite[Lemma 3.4]{KostdsB} then shows the desired containment, from where the claim is clear.
\end{proof}

Returning to the proof of Proposition \ref{prop:Kostant-triples-prop}, let us write $h=\ad(e)(f)$. Note that $[[h,f]+2f,e]$ vanishes and thus $[h,f]+2f$ is in $\mf{a}^e$. By Lemma \ref{lem:ad(x)+2-invertible} we may write $[h,f]+2f=[h,g]+2g$ for some $g$ in $\mf{a}^e$. So then, if we take $f''=f-g$ then 
\begin{equation*}
    [h,e]=2e,\qquad [h,f'']=[h,f]-[h,g]=-2f'',\qquad [e,f'']=[e,f]-[e,g]=h-0=h,
\end{equation*}
as desired.

\section{Moduli spaces of Weil--Deligne parameters}

To give a geometrization of the results of \S\ref{ss:JM-for-params-classical} it is useful to first develop a space intermediary between the moduli space of $L$-parameters (see \S\ref{s:L-param}) and the moduli space of Weil--Deligne parameters. We give such a space in this section which, in short, parameterizes Weil--Deligne parameters whose monodromy operator lies in $\mathcal{N}^\sqcup$.

\subsection{The moduli space of Weil--Deligne parameters}\label{ss:WD-params}

We first recall the moduli space of Weil--Deligne parameters roughly following the presentation as in \cite{ZhuCohLp}. In particular, we use the $C$-group of Buzzard--Gee in lieu of the $L$-group for our definition of parameters. The theory of $C$-groups is better suited to dealing with $L$-parameters valued in arbitrary $\mathbb{Q}$-algebras $R$ as many constructions involving $L$-parameters in terms of $L$-groups require a choice of square root of $q$ in $R$. Also, the authors find that many necessary arguments involving $L$-parameters in terms of $L$-groups (e.g.\@ see Proposition \ref{prop:red-cent-ss} and Proposition \ref{prop:dense-tor-cent}) require the consideration of certain ancillary groups which, ultimately, end up being equivalent to working in the $C$-group.

\medskip

\paragraph{Initial definitions} We begin by defining the relative analogue of a Weil--Deligne parameter. 

\begin{defn}
For a $\Q$-algebra $A$, we define a \emph{Weil--Deligne parameter over} $A$ to be a pair $(\varphi,N)$ where 
\begin{enumerate}[leftmargin=2cm,widest=iiiiii]
    \item[\textbf{(WDP1)}] $\varphi\colon \cW_{F,A}\to \CG_A$ is a morphism of group $A$-schemes such that $p_C \circ \varphi=(\| \cdot \|,\id)$,
    \item[\textbf{(WDP2)}] $N$ is an element of $\wh{\mc{N}}(A)$ such that $\Ad(\varphi(w))(N)=\|w\| N$ for all $w \in \cW_F (A)$.
\end{enumerate}
\end{defn}
We denote the set of Weil--Deligne parameters over $A$ by $\WDP_G(A)$ which clearly constitutes a presheaf on $\Q$-algebras. The presheaf $\WDP_G$ has a natural action by $\widehat{G}$ given by 
\begin{equation*}
    g(\varphi,N)g^{-1}\defeq (\Int(g)\circ \varphi,\Ad(g)(N)).
\end{equation*}
So, for a Weil--Deligne parameter $(\varphi,N)$ we may consider the centralizer group presheaf $Z_{\wh{G}}(\varphi,N)$.

We define the morphism $\check{\varphi}\colon \cW_{F,A}\to \check{G}_A$ of schemes 
as the composition of $\varphi$ with the projection to $\check{G}_A$. We denote by $\ov{\varphi}$ the homomorphism $\cW_{F,A}\to (\wh{G}\rtimes \underline{\Gamma_\ast})_A$ obtained by composing $\varphi$ with the quotient map $\CG_A\to (\wh{G}\rtimes \underline{\Gamma_\ast})_A$. Observe that while $\check{\varphi}$ may not be a homomorphism, it becomes so after restriction to $\cW_{F^\ast,A}$. In particular, for any $w \in \cW_F (A)$ the restriction of $\check{\varphi}$ to $\langle w^m\rangle$ is a homomorphism whenever $[F^\ast:F]$ divides $m$. 

Let $K$ be a finite extension of $F^\ast$ Galois over $F$, and let us define for a $\Q$-algebra $A$ the set
\begin{equation*}
    \WDP_G^K(A)\defeq \left\{(\varphi,N)\in \WDP_G(A): \cI_{K,A}\subseteq \ker(\check{\varphi}|_{\cW_{F^\ast,A}})\right\}. 
\end{equation*} 
We observe that $\WDP_G^K$ forms a $\wh{G}$-stable subfunctor of $\WDP_{G}$. In fact, one sees that there is an equality of functors $\WDP_G=\varinjlim \WDP_G^K$ as $K$ travels over all such extensions.

We finally observe that $\WDP_G$ has a more familiar form over an extension $k$ of $\Q$ containing an element $c$ such that $c^2=q$. 
More precisely, for a $k$-algebra $A$, we equip $\wh{G}(A)$ with the discrete topology and put 
\begin{equation*}
    \WDP_{G,k}'(A)\defeq \left\{(\varphi,N):\begin{aligned} (1)&\quad \varphi\colon W_F\to \widehat{G}(A) \rtimes W_F \text{ is a a continuous cross-section homomorphism}, 
    \\ (2)&\quad N\in\widehat{\cN}(A)\text{ is such that }\Ad(\varphi(w))(N)=\|w\|N\text{ for all }w\in W_F\end{aligned}\right\}.
\end{equation*}
It is clear that $\WDP'_G$ is a functor on the category of $k$-algebras and comes equipped with a natural action of $\wh{G}_k$. Let us also observe that if $i_c$ is the map from \S\ref{ss:L-and-C} then there is a morphism $i_c^{\mathrm{WD}} \colon \WDP_{G,k}'\to \WDP_{G,k}$ which on $A$-points is given by sending $(\varphi',N)$ to the unique element of $\WDP_G(A)$ of the form $(\varphi,N)$ which is equal to $(i_c\circ \varphi',N)$ on $A$-points.

\begin{prop}\label{prop:WD-C-L-comparison} The morphism of functors $i_c^{\mathrm{WD}} \colon \WDP'_{G,k}\to \WDP_{G,k}$ is an isomorphism.
\end{prop}
\begin{proof}
This follows from the cartesian diagram 
\begin{equation*}
\xymatrixcolsep{4pc}
    \xymatrix{{}^L G_k\ar[r]^{i_c}\ar[d] & {}^C G_k \ar[d]^{p_C}\\ \cW_{F,k}\ar[r]^-{(\| \cdot \|,\id )} & \bG_{m,k} \times \cW_{F,k} } 
\end{equation*}
and that any morphism 
$\cW_{F,A}\to \check{G}_A$ of schemes over $A$ factors through $(\cW_{F}/\cI_{K})_A$ for a 
finite extension $K$ of $F$.  
\end{proof}

Our discussion of the space $\WDP'_{G,k}$ is not strictly necessary for
the paper. However, while the $C$-group is better suited to the technical
setting of the paper, the $L$-group setting is more common. Thus, we
include $\WDP'_{G,k}$, and the space $\LP'_{G,k}$ as in \S\ref{s:L-param},
and so we discuss both to reconcile these two points of view.

\medskip

\paragraph{Representability} We now establish the representability of the functor $\WDP_{G}$. To this end, let us fix $K$ a finite extension of $F^\ast$ Galois over $F$. Note that for a $\Q$-algebra $A$ and an element $(\varphi,N)$ of $\WDP_G^K(A)$ we may define an element $\phi$ of $\underline{Z}^1(I_F/I_K,\wh{G})(A)$ as follows. 
First observe that condition \textbf{(WDP1)} implies that $\varphi|_{\cI_{F,A}}$ takes values in $\wh{G}_A \rtimes \cI_{F,A}$. Then, as $(\varphi,N)$ is in $\WDP_G^K(A)$, the composition of $\varphi|_{\cI_{F,A}}$ with the projection $\wh{G}_A\rtimes \cI_{F,A}\to \wh{G}_A\rtimes (\cI_F/\cI_K)_A$ factorizes through a cross-section homomorphism 
$(\cI_F/\cI_K)_A \to \wh{G}_A\rtimes (\cI_F/\cI_K)_A$. 
This gives an element $\phi$ of $\underline{Z}^1(I_F/I_K,\wh{G})(A)$ since $\cI_F/\cI_K \cong \underline{I_F/I_K}$. This association defines a morphism of presheaves $\WDP_G^K\to \underline{Z}^1(I_F/I_K,\wh{G})$. 

Let us now fix a lift $w_0$ of arithmetic Frobenius in $W_F$. Define a morphism of presheaves
\begin{equation*}
    j_{w_0}\colon \WDP_G^K\to \check{G}\times \underline{Z}^1(I_F/I_K,\wh{G})\times \wh{\mc{N}},\qquad (\varphi,N)\mapsto (\check{\varphi}(w_0),\phi,N).
\end{equation*}
On the other hand, we have a diagram
\begin{equation*}
    \mc{D}^\WD\colon \xymatrix{\check{G}\times \underline{Z}^1(I_F/I_K,\wh{G})\times \wh{\mc{N}} \ar@<-.5ex>[r] \ar@<.5ex>[r]& \underline{\Hom}(I_F/I_K,\wh{G})\times \bb{G}_{m,\Q}\times \wh{\mc{N}}^{[I_F:I_K]+1}}
\end{equation*}
given by 
\begin{equation*}
    \begin{aligned} (g,f,M)\mapsto & \bigg(\Int(g,w_0)\circ f,p_{\bG_m}(g), (\mathrm{Ad}(f(i))(M))_{i\in I_F/I_K},\Ad(g,w_0)(M)\bigg)\\  (g,f,M) \mapsto &\bigg(f\circ \Int(w_0),q,(M)_{i\in I_F/I_K},qM\bigg).\end{aligned}
\end{equation*} 
We then have the following explicit description of $\WDP_G^K$.

\begin{prop}\label{prop:wd-finite-level-rep} The morphism $j_{w_0}$ identifies $\WDP_G^K$ with the equalizer $\mathsf{Eq}(\mc{D}^\WD)$. Thus, $\WDP_G^K$ is representable by a finite type affine $\Q$-scheme and $j_{w_0}$ is a closed embedding.
\end{prop}

Observe that for an extension $K\subseteq K'$ of Galois extensions of $F$ containing $F^\ast$ there is a restriction morphism $\underline{Z}^1(I_F/I_{K'},\wh{G})\to \underline{Z}^1(I_K/I_{K'},\wh{G})$. By Proposition \ref{prop:cocycle-scheme} and Lemma \ref{lem:stratification-isom} the subspace consisting of only the trivial homomorphism is a clopen subset of the target, and thus so is its preimage in $\underline{Z}^1(I_F/I_{K'},\wh{G})$, but this is precisely $\underline{Z}^1(I_F/I_K,\wh{G})$. We deduce that the natural inclusion of functors $\WDP_G^K\to \WDP_G^{K'}$ is a clopen embedding. From the identification $\WDP_G=\varinjlim_K \WDP_G^K$ we deduce from Proposition \ref{prop:wd-finite-level-rep} that $\WDP_G$ is representable by a scheme locally of finite type over $\Q$, all of whose connected components are affine.

The following non-trivial result will play an important technical role below (e.g.\@ in the proof of Theorem \ref{thm:JM-omnibus}).

\begin{thm}[{\cite[Corollary 2.3.7]{BeGeGdef} and \cite[Corollary 3.1.10]{ZhuCohLp}}]\label{thm:WD-reduced} The schemes $\WDP_G^K$ are reduced for all $K$, and thus, a fortiori, $\WDP_G$ is reduced.
\end{thm}

\subsection{Semi-simplicity of parameters}

As in the Theorem \ref{thm:JM-params-classical} one requires Frobenius semi-simplicity conditions to get a Jacobson--Morozov result in the relative setting. Therefore, we now develop a sufficient notion of Frobenius semi-simplicity for a Weil--Deligne and $L$-parameter over a $\Q$-algebra $A$.

\begin{defn}\label{defn:s-s-elem} Let $R$ be a $\Q$-algebra and $H$ is a smooth group $R$-scheme such that $H^\circ$ is reductive. We then say that an element $h$ of $H(R)$ is \emph{semi-simple} if there exists some $m\geqslant 1$, an \'etale cover $\Spec(S)\to\Spec(R)$, and a torus $T$ of $H^\circ_S$ such that $h^m$ is in $T(S)$.
\end{defn} 

By \cite[Expos\'{e} VIB, Corollaire 4.4]{SGA3-1} $H^\circ$ is representable so the above makes sense. Moreover, by \cite[Proposition B.3.4]{ConRgrsch} we may assume that $T$ is split in the above definition.

\begin{prop}\label{prop:ss-properties}  Let $R$ be a $\Q$-algebra and $H$ is a smooth group $R$-scheme such that $H^\circ$ is reductive, and let $h$ be an element of $H(R)$. Then, the following statements are true.
\begin{enumerate}
    \item If $h$ is semi-simple, there exists an \'etale cover $\Spec(S)\to\Spec(R)$, an integer $m\geqslant 1$, and a split maximal torus $T$ of $H_S^\circ$ such that $h^m$ is in $T(S)$.
    \item If $Z$ is a closed subgroup $R$-scheme of $Z(H^\circ)$ which is flat over $R$, then $h$ is semi-simple if and only if its image in $(H/Z)(R)$ is semi-simple.
\end{enumerate}
\end{prop}
\begin{proof} To show (1) let $\Spec(S')\to \Spec(R)$ be an \'etale cover and $T'$ a torus of $H^\circ_{S'}$ such that $h^m$ is in $T'(S')$. Note that $Z_{H^\circ}(T')$ is a reductive group (combine \cite[Lemma 2.2.4]{ConRgrsch} and \cite[Corollary 17.59]{MilneGroups}). By \cite[Corollary 3.2.7]{ConRgrsch} there exists an \'etale cover $\Spec(S)\to\Spec(S')$ and a maximal torus $T$ of $Z_{H^\circ}(T')_S$. Observe that $T$ is also a maximal torus of $H^\circ_S$. Indeed, it is evidently a torus, and its maximality can be checked over each point $x$ of $\Spec(S)$. Then, as $T_x$ is contained in a maximal torus of $H^\circ_x$, and $T'_x\subseteq T_x$, we see that this maximal torus is actually contained in $Z_{H_x^\circ}(T'_x)=Z_{H^\circ}(T')_x$, and so must be equal to $T_x$. As $T'_S$ is central in $Z_{H_{S'}^\circ}(T')_S$ it is clear that $T$ contains $T'_S$ and thus $h^m$ is contained in $T(S)$. As we may pass to a futher \'etale extension to split $T$ the claim follows. 

Let $f\colon H^\circ\to H^\circ/Z$ be the tautological map. To prove (2) it is sufficient to note that for any $R$-algebra $S$ one has that the map $T\mapsto T/Z$ and $T'\mapsto f^{-1}(T')$ are mutually inverse bijections between the maximal tori of $H^\circ_S$ and $(H^\circ/Z)_S$ by \cite[Corollary 3.3.5]{ConRgrsch}. 
\end{proof}

Consider a representation $\rho\colon H\to \GL(M)$ where $M$ is a finitely generated $R$-module. Let $h$ be an element of $H(R)$ and $I$ a finite subgroup of $H(R)$ that is stable under conjugation by $h$. For any $R$-algebra $S$ and any $\lambda$ in $S^\times$ let us set
\begin{equation*}
    M_S^I(h,\lambda)\defeq\ker\left(\rho(h)-\lambda|M_S^{\rho(I)}\to M_S^{\rho(I)}\right).
\end{equation*}
Abbreviate $M_R^I(h,\lambda)$ to $M^I(h,\lambda)$, and further abbreviate to $M^I(\lambda)$ if $h$ is clear from context. Finally, we omit $I$ from the notation if $I$ is trivial. Evidently $M_S^I(h,\lambda)\otimes_S S'$ is equal to $M_{S'}^I(h,\lambda)$ for any flat map of $R$-algebras $S\to S'$.

\begin{prop}\label{prop:eigen-decomp} Assume that $h$ is semi-simple. Then, there exists a unique decomposition
\begin{equation*}
    M^{\rho(I)}=\bigoplus_{\lambda \in R^\times}M^I(h,\lambda)\oplus M'
\end{equation*}
such that for any flat map $R\to S$ one has that 
\begin{equation*}
    \bigoplus_{\lambda\in S^\times-R^\times}M_S^I(h,\lambda)
\end{equation*}
is a direct summand of $M'_S$, and such that this is an equality if for some $m\geqslant 1$:
\begin{enumerate}
    \item $h^m$ is contained in a split torus of $H^\circ_S$ and commutes with $I$, 
    \item $S$ is a $\Q(\zeta_r)$-algebra, where $r\defeq [\langle h\rangle : \langle h^m\rangle]$ and $\zeta_r$ is a primitive $r^\text{th}$-root of unity, 
    \item and $S$ contains an $r^\text{th}$-root of all $\lambda$ such that $M(h^r,\lambda)\ne 0$.
\end{enumerate}
\end{prop}
\begin{proof} Take an \'etale cover $\Spec(S)\to \Spec(R)$ and $m\geqslant 1$ such that $h^m$ is contained in a split torus $T$ of $H^\circ_S$ and $h^m$ commutes with $I$. Then $h^r \in \langle h^m\rangle$ is contained in $T$ and commutes with $I$. By \cite[Lemma A.8.8]{CGP} one may decompose $M_S$ into character spaces $M_S(\chi)$. One then observes that $M_S(h^r,\lambda)$ is precisely the direct sum of those character spaces $M_S(\chi)$ such that $\chi(h^r)=\lambda$. So, $M_S$ admits a direct sum decomposition with respect to the spaces $M_S(h^r,\lambda)$. 

As $M_S$ is finitely generated, we know that $M_S(h^r,\lambda)$ is trivial for all but finitely many $\lambda_1,\ldots,\lambda_e$, as a finitely generated module can have only finitely many non-zero direct summands. In particular, we may further pass to the \'etale extension $S'\defeq S[\lambda_1^{\nicefrac{1}{r}},\ldots,\lambda_e^{\nicefrac{1}{r}},\zeta_r]$. We extend the action of $I$ on each nontrivial $M_{S'}(h^r,\lambda)$ by $\rho$ to the action of the finite group $I \rtimes (\langle h\rangle/\langle h^m\rangle)$ letting $h$ act $\lambda^{-\nicefrac{1}{r}}\rho (h)$. As $S'$ is a $\Q (\zeta_r)$-algebra, we have a decomposition of $M_{S'}^I(h^r,\lambda)$ into character spaces $M_{S'}(h^r,\lambda)[\nu]$ where $\nu$ travels over the characters $I \rtimes (\langle h\rangle/\langle h^m\rangle) \to \langle h\rangle/\langle h^m\rangle\to S'$. We then see that for each $\tau \in (S')^\times$ such that $\tau^r=\lambda$ the space $M_{S'}^I(h,\tau)$ admits a direct decomposition into the spaces $M_{S'}(h^r,\lambda)[\nu]$ as $\nu$ ranges over those characters with $\nu(h)=\lambda^{-\nicefrac{1}{r}}\tau$. 

One may then check that the module $\bigoplus_{\tau}M_{S'}^I(h,\tau)$ as $\tau$ ranges over those elements of $(S')^\times-R^\times$ is stabilized under the \'etale descent data associated to $M_{S'}^{\rho(I)}$, and therefore (see \stacks{023N}) descends to a submodule $M'$ of $M^{\rho(I)}$. One sees that $M'$ is a complement of $\bigoplus_\lambda M^I(h,\lambda)$ as $\lambda$ travels over the elements of $R^\times$, as this may be checked over the faithfully flat extension $S'$. One may then check that $M'$ is independent of all choices, and satisfies the desired conditions.
\end{proof}

The following proposition will be helpful to define Frobenius semi-simple in a way that does not require the choice of an explicit arithmetic Frobenius lift.

\begin{prop}\label{prop:Frob-factor}
Let $\varphi\colon \cW_{F,A}\to \CG_A$ be a morphism of group schemes over a $\mathbb{Q}$-algebra $A$. 
Then there is a positive integer $m$ divisible by $[F^\ast:F]$ such that the morphism $\cW_{F,A}\to \check{G}_A$ given by $w \mapsto \check{\varphi}(w^m)$ admits a factorization
\begin{equation*}
    \cW_{F,A} \stackrel{d}{\lra} \underline{\mathbb{Z}}_A \stackrel{\check{\varphi}_m}{\lra} \check{G}_A
\end{equation*}
and $\check{\varphi}_m$ takes values in $Z_{\check{G}}(\varphi)$. 
\end{prop}
\begin{proof}
Take a finite extension $K$ of $F^\ast$ Galois over $F$ such that 
$\check{\varphi}|_{\cI_{K,A}}$ is trivial. 
Take a lift $w_0 \in W_F$ of arithmetic Frobenius and choose $m_0$ such that 
the image of $w_0^{m_0}$ in $W_F/I_K$ is central. 
Let $m$ be the order of $W_F/(I_K \langle w_0^{m_0} \rangle)$. Then for any $w \in W_F$, since $w^m$ is trivial in $W_F/(I_K\langle w^{m_0}_0 \rangle)$, we have that $w^m=iw^{d(w)m}_0$ for some $i \in I_K$. Hence, the images of $w^m$ 
and $w_0^{md(w)}$ in $W_{F^\ast}/I_K$ are the same. 
Since $\check{\varphi}|_{\cW_{F^\ast,A}}$ factors through 
$(\cW_{F^\ast}/\cI_{K})_A$, we have  $\check{\varphi}(w^m)=\check{\varphi}(w_0^m)^{d(w)}$ 
for any point $w$ of $\cW_{F,A}$. Hence we have the factorization  $\check{\varphi}_m \colon \underline{\mathbb{Z}}_A \to \check{G}_A$. 
The composition 
\begin{equation*}
    \cW_{F,A} \stackrel{\varphi}{\lra} \CG \lra \check{G}_A \rtimes (\cW_{F}/\cI_{K})_A 
\end{equation*}
factors through $\varphi_K \colon (\cW_{F}/\cI_{K})_A \to \check{G}_A \rtimes (\cW_{F}/\cI_{K})_A$. To show that $\check{\varphi}_m$ factors through $Z_{\check{G}}(\varphi)$, it suffices to show $\check{\varphi}(w_0^m) \in Z_{\check{G}}(\varphi_K)$. Since the image of $w_0^m$ in $W_F/I_K$ is central, we have $\varphi_K(w_0^m) \in Z_{\check{G}_A \rtimes (\cW_{F}/\cI_{K})_A}(\varphi_K)$. Since the image of $(1,w_0^m)$ in $\check{G}_A \rtimes (\cW_{F}/\cI_{K})_A$ is central, we obtain $\check{\varphi}(w_0^m) \in Z_{\check{G}_A}(\varphi_K)$. 
\end{proof}

To define the notion of Frobenius semi-simple parameters, it is useful to have the following analogue of Lemma \ref{lem:L-group-ss}.

\begin{prop}\label{prop:Frob-ss-equiv}Let $(\varphi,N)$ be an element of $\WDP_G(A)$. Then, the following are equivalent:
\begin{enumerate}
    \item for any (equiv.\@ one) lift $w_0\in W_F$ of arithmetic Frobenius, $\ov{\varphi}(w_0)$ is semi-simple,
    \item for some $m$ as in Proposition \ref{prop:Frob-factor}, the morphism $\check{\varphi}_{m}$ \'etale locally factorizes through a torus of $\check{G}_A$.
\end{enumerate}
\end{prop}
\begin{proof} By definition, (1) holds if and only if $\ov{\varphi}(w_0)$ has the property that $\ov{\varphi}(w_0)^{m}$ \'etale locally lies in a torus of $(\check{G}\rtimes \underline{\Gamma_\ast})_A^\circ=\check{G}_A$ for some $m$  as in Proposition \ref{prop:Frob-factor}. But, as an element of $\check{G}_A$, one easily sees that $\ov{\varphi}(w_0)^{m}$ is precisely $\check{\varphi}_{m}(1)$. As it is clear that (2) is equivalent to claim that \'etale locally on $A$ there exists a torus containing $\check{\varphi}_{m}(1)$ the claim follows.
\end{proof}

\begin{defn}\label{defn:Frob-ss-WD-param} For a $\Q$-algebra $A$, we call an element $(\varphi,N)$ of $\WDP_G(A)$ \emph{Frobenius semi-simple} if it satisfies any of the equivalent conditions of Proposition \ref{prop:Frob-ss-equiv}.
\end{defn}

For each $\Q$-algebra $A$, let us denote by $\WDP^\ss_G(A)$ (resp.\@ $\WDP_G^{K,\ss}(A)$) the subset of $\WDP_G(A)$ (resp.\@ $\WDP_G^K(A)$) consisting of Frobenius semi-simple parameters. It is clear that this forms a $\wh{G}$-stable subpresheaf of $\WDP_G$ (resp.\@ $\WDP_G^K$). Note that one does not expect this presheaf to be representable as the semi-simple elements in algebraic group form a constructible, but not locally closed, subset. Note also that by Proposition \ref{prop:ss-properties}, under the bijection of $\WDP_G(\C)$ with $\Phi^{\WD,\square}_G$ the set $\WDP_G^\ss(\C)$ corresponds to $\Phi^{\WD,\ss,\square}_G$.

The following technical result will play an important role later in the paper.

\begin{prop}\label{prop:red-cent-ss} If $A$ is a reduced $\Q$-algebra and $(\varphi,N)$ is an element of $\WDP_G(A)$ such that $Z_{\widehat{G}}(\varphi,N)^\circ_x$ is reductive of dimension $n$ for all $x$ in $\Spec(A)$, then $(\varphi,N)$ is Frobenius semi-simple.
\end{prop}
\begin{proof}
Define $S(N)$ to be the closed subgroup scheme of $\check{G}_A$ cut out by the closed condition $gNg^{-1} =p_{\bG_m}(g)N$. We have the equality $Z_{\widehat{G}}(\varphi,N)=\ker(p_{\bG_m}|_{Z_{S(N)}(\varphi)})$. Note that for all $x$ in $\Spec(A)$ one has a short exact sequence
\begin{equation*}
    1\to Z_{\wh{G}}(\varphi,N)_x\to Z_{S(N)}(\varphi)_x\to \bb{G}_{m,x}\to 1,
\end{equation*}
and as $Z_{\wh{G}}(\varphi,N)^\circ_x$ is assumed to be reductive of dimension $n$ for all $x$ in $\Spec(A)$, that $Z_{S(N)}(\varphi)^\circ_x$ is reductive of dimension $n+1$, and thus $Z_{S(N)}(\varphi)^\circ$ is representable and smooth over $A$, and thus reductive over $A$, by \cite[Expos\'{e} VIB, Corollaire 4.4]{SGA3-1} and \cite[Theorem 3.23]{MilneGroups}.

We take $m$ as Proposition \ref{prop:Frob-factor}. 
Then $\check{\varphi}_m$ factors through 
$Z_{S(N)}(\varphi)$. 
Further it factors through $Z(Z_{S(N)}(\varphi))$, 
since $\varphi(w^m)$ and $(1,w^m)$ commutes with $Z_{S(N)}(\varphi)$ for any point $w$ of $\cW_{F,A}$. 
Then there is an $m'$ such that 
$\check{\varphi}_m^{m'}=\check{\varphi}_{mm'}$ 
factors through $Z(Z_{S(N)}(\varphi)^{\circ})^{\circ}$. 
As $Z_{S(N)}(\varphi)^\circ$ is reductive, $Z(Z_{S(N)}(\varphi)^\circ)^\circ$ is a torus. Hence $(\varphi,N)$ is Frobenius semi-simple. 
\end{proof}

\subsection{\texorpdfstring{The space $\WDP^\sqcup_G$}{The space Locsqcup}} In this section we study the moduli space of Weil--Deligne parameters $(\varphi,N)$ where $N$ lies in $\mc{N}^\sqcup$ and show that this moduli space has an exceedingly simple structure.

\begin{defn} We denote by $\WDP_G^{K,\sqcup}$ (resp.\@ $\WDP^\sqcup_{G}$) the space $\WDP_G^K\times_{\wh{\mc{N}}}\wh{\mc{N}}^\sqcup$ (resp.\@ $ \WDP_{G}\times_{\wh{\mc{N}}}\wh{\mc{N}}^\sqcup=\varinjlim_K \WDP_G^{K,\sqcup}$).
\end{defn}

Now, let us fix a finite extension $K$ of $F^\ast$ Galois over $F$ and a lift $w_0$ of arithmetic Frobenius. Then, by Proposition \ref{prop:wd-finite-level-rep} we have an identification $j_{w_0}$ of $\WDP_G^K(\ov{\Q})$ with
\vspace*{1 pt}
\begin{equation*}
    \left\{(\gamma,\phi,N)\in \check{G}(\ov{\Q})\times \underline{Z}^1(I_F/I_K,\wh{G})(\ov{\Q})\times \wh{\mc{N}}(\ov{\Q}):\begin{aligned}(1)&\quad \Int(\gamma,w_0)\circ \phi=\phi\circ\Int(w_0),\\ (2)&\quad p_{\bG_m}(\gamma)=q,\\ (3)&\quad \mathrm{Ad}(\phi(i))(N)=N\text{ for all }i\in I_F/I_K,\\ (4)&\quad \Ad(\gamma,w_0)(N)=qN\end{aligned}\right\}.
\end{equation*}
\vspace*{1 pt}
Now, for $(\gamma,\phi,N)$ in $\WDP_G^K(\ov{\Q})$ let us define $Z_{\phi,N}\defeq Z_{\wh{G}}(\phi,N)$. 

\begin{defn}
An element $(\gamma',\phi',N')$ in $\WDP_G^K(A)$, for a $\ov{\Q}$-algebra $A$, is \emph{locally movable to $(\gamma,\phi,N)$} if there exists an \'etale cover $\Spec(A')\to\Spec(A)$ and $(g,h)\in(\wh{G}\times Z_{\phi,N}^\circ)(A')$ such that $(\gamma',\phi',N')=g(h\gamma,\phi,N)g^{-1}$.
\end{defn}

As this definition is clearly functorial, we observe that we may define a subpresheaf $U(\gamma,\phi,N)$ of $\WDP_{G,\ov{\Q}}^{K,\sqcup}$ whose $A$-points are given by
\begin{equation*}
    U(\gamma,\phi,N)(A)\defeq \left\{(\gamma',\phi',N')\in\WDP_{G,\ov{\Q}}^{K,\sqcup}(A): (\gamma',\phi',N')\text{ is locally movable to }(\gamma,\phi,N)\right\}.
\end{equation*}
We then have the following.

\begin{prop}\label{prop:wd-loc-mov-is-open} The morphism of presheaves $U(\gamma,\phi,N)\to \WDP_{G,\ov{\Q}}^{K,\sqcup}$ is representable by an open immersion. Moreover, the $\ov{\Q}$-scheme $U(\gamma,\phi,N)$ is smooth and irreducible.
\end{prop}

Before we prove this proposition, we observe its major consequence. To this end, let us define an equivalence relation on $\WDP_G^K(\ov{\Q})$ by declaring that $(\gamma,\phi,N)$ is equivalent to $(\gamma',\phi',N')$ if there exists some $(g,h)\in (\wh{G}\times Z_{\phi,N})(\ov{\Q})$ such that $(\gamma',\phi',N')$ is equal to $g(h\gamma,\phi,N)g^{-1}$. Let us denote an equivalence class under this relation by $[(\gamma,\phi,N)]$. Observe that as we do not require that $h$ to actually lie in $Z^\circ_{\phi,N}(\ov{\Q})$ that $[(\gamma,\phi,N)]$ differs from $U(\gamma,\phi,N)(\ov{\Q})$. For each such equivalence class, let us choose an element $(\gamma,\phi,N)$. We consider $\pi_0(Z_{\phi,N})$ as a finite abstract group, and we define an equivalence relation on it by declaring that $c$ is equivalent to $c_1 c \gamma c_1^{-1} \gamma^{-1}$ for any $c_1$ in $\pi_0(Z_{\phi,N})$. We denote by $[c]$ an equivalence class for this relation.
\begin{rem}
The group $\langle \gamma \rangle$ acts on $\pi_0(Z_{\phi,N})$ by $\gamma \cdot c = \gamma c \gamma^{-1}$. Note that $\langle \gamma \rangle \cong \Z$ since $p_{\mathbb{G}_m}(\gamma)=q$. Hence, the map $z \mapsto z(\gamma)$ for $z \in Z^1(\langle \gamma \rangle, \pi_0(Z_{\phi,N}))$ induces a bijection between $H^1( \langle \gamma \rangle, \pi_0(Z_{\phi,N}))$ and equivalence classes in $\pi_0(Z_{\phi, N})$.
\end{rem}

We then have the following decomposition of $\WDP_{G,\ov{\Q}}^{K,\sqcup}$ into explicit connected components. 

\begin{thm}\label{thm:WD-const-decomp} The choice of $(\gamma,\phi,N)$ in each class $[(\gamma,\phi,N)]$ of $\WDP_G^K(\ov{\Q})$ gives a scheme-theoretic decomposition
\begin{equation*}
    \WDP_{G,\ov{\Q}}^{K,\sqcup}=\bigsqcup_{[(\gamma,\phi,N)]}\,\bigsqcup_{[c]}\,\,U(c\gamma,\phi,N).
\end{equation*}
\end{thm}
\begin{proof} From Proposition \ref{prop:wd-loc-mov-is-open} we know that each $U(c\gamma,\phi,N)$ is an open subset of $\WDP_{G,\ov{\Q}}^{K,\sqcup}$. As $\WDP_{G,\ov{\Q}}^{K,\sqcup}$ is a finite type $\ov{\Q}$-scheme, it thus suffices to prove this claim at the level of $\ov{\Q}$-points. But, note that by Proposition \ref{prop:wd-finite-level-rep}, if $(\gamma,\phi,N)$ satisfies the conditions to be in $\WDP_G^K(\ov{\Q})$ then $(\gamma',\phi,N)$ does if and only if $\gamma'=h\gamma$ for $h$ in $Z_{\phi,N}(\ov{\Q})$. Thus,  we have a decomposition
\begin{equation*}
    \WDP_{G,\ov{\Q}}^{K,\sqcup}=\bigsqcup_{[(\gamma,\phi,N)]} \,\bigcup_{c\in \pi_0(Z_{\phi,N})}U(c\gamma,\phi,N).
\end{equation*}
 Next observe that an element $(h\gamma,\phi,N)$ may be written in the form $g(h'\gamma,\phi,N)g^{-1}$ if and only if $g$ is in $Z_{\phi,N}(\ov{\Q})$ and $h\gamma=gh'\gamma g^{-1}$ which implies that $h=gh'\gamma g^{-1}\gamma^{-1}$. With this, it is easy to see that 
\begin{equation*}
    \bigcup_{c\in \pi_0(Z_{\phi,N})}U(c\gamma,\phi,N)=\bigsqcup_{[c]}\,\,U(c\gamma,\phi,N)
\end{equation*}
from where the desired equality follows.
\end{proof}

From this we deduce the following non-trivial result. Let us denote the set of equivalence classes for $\WDP_G^K(\ov{\Q})$ (resp.\@ $\pi_0(Z_{\phi,N})$) by $[\WDP_G^K(\ov{\Q})]$ (resp.\@ $[\pi_0(Z_{\phi,N})]$).

\begin{cor}\label{cor:WDP-pi0} The $\Q$-scheme $\WDP_G^{K,\sqcup}$ is smooth, and there is a non-canonical $\Gamma_\Q$-equivariant bijection
\begin{equation*}
    \pi_0\left(\WDP_{G,\ov{\Q}}^{K,\sqcup}\right)\isomto \left\{([(\gamma,\phi,N)],[c]):\begin{aligned}(1)&\quad [(\gamma,\phi,N)]\in [\WDP_G^K(\ov{\Q})]\\ (2)&\quad [c]\in [\pi_0(Z_{\phi,N})] \end{aligned}\right\}
\end{equation*}
where the $\Gamma_\Q$ action on the target is inherited from $\WDP_G^{K,\sqcup}$ and $\wh{G}$.
\end{cor}

\medskip

\paragraph{The proof of Proposition \ref{prop:wd-loc-mov-is-open}} Define the morphism $\pi_K\colon \WDP_G^{K,\sqcup}\to \underline{Z}^1(I_F/I_K,\wh{G})\times\wh{\mc{N}}^\sqcup$ by $\pi_K(\varphi,N)= (\phi,N)$. This morphism is $\wh{G}$-equivariant when the target is endowed with the diagonal $\wh{G}$-action. Now, by Proposition \ref{prop:cocycle-scheme} there is a decomposition
\begin{equation*}
    \underline{Z}^1(I_F/I_K,\wh{G})_{\ov{\Q}}\times\wh{\mc{N}}^\sqcup_{\ov{\Q}}=\bigsqcup_{[(\phi_0,N_0)]\in\mc{J}}\mc{O}_{\phi_0}\times \mc{O}_{N_0}
\end{equation*}
where $\mc{J}
$ is the set of $\wh{G}(\ov{\Q})^2$ orbits of $(\underline{Z}^1(I_F/I_K,\wh{G})\times\wh{\mc{N}}^\sqcup)(\ov{\Q})$. Observe though that if $(\varphi,N)$ is in $\WDP_G^{K,\sqcup}(\ov{\Q})$ with $\pi_K(\varphi,N)=(\phi,N)$ then $\phi$ centralizes $N$. So, if we set $\mc{J}'$ to be the subset of $\mc{J}$ consisting of those $[(\phi_0,N_0)]$ with $\phi_0$ centralizing $N_0$ then we may produce a factorization
\begin{equation*}
    \pi_K\colon \WDP_{G,\ov{\Q}}^{K,\sqcup}\longrightarrow \bigsqcup_{[(\phi_0,N_0)]\in\mc{J}'}\mc{O}_\phi\times \mc{O}_N
\end{equation*}
which is $\wh{G}$-equivariant. For each $[(\phi_0,N_0)]$ in $\mc{J}'$ let us set $X(\phi_0,N_0)\defeq \pi_K^{-1}(\h_{\phi_0}\times \h_{N_0})$, which is a clopen subset of $\WDP_{G,\ov{\Q}}^{K,\sqcup}$. 

Set $L\defeq Z_{\wh{G}}(\phi)$ which, by Lemma \ref{lem:fixed-points-reductive} applied to the image of $\phi$ in $\wh{G}(\ov{\Q})\rtimes (I_F/I_K)$, is a closed subgroup scheme of $\wh{G}_{\ov{\Q}}$ with reductive identity component. Let $\mf{l}$ be the Lie algebra of $L$. Define $\mc{O}_N\cap\mc{N}_{L}\defeq \mc{O}_N\times_{\wh{\mc{N}}}\mc{N}_L$. For each $M$ in $(\mathcal{O}_N\cap \mc{N}_L)(\ov{\Q})$ we denote by $\mc{O}_{L,M}$ the locally closed $L$-orbit subscheme of $(\mc{O}_N\cap\mc{N}_L)_\red$.

\begin{lem}\label{lem:intersection-decomp} There exists a finite set $\{N=N_1,N_2,\ldots,N_m\}$ in $(\mc{O}_N\cap \mc{N}_L)(\ov{\Q})$ such that one has an equality of schemes $\mathcal{O}_N\cap \mathcal{N}_L=\bigsqcup_i \mathcal{O}_{L,N_i}$. In particular, $\mc{O}_N\cap\mc{N}_L$ is reduced.
\end{lem}
\begin{proof} We first show that the claimed decomposition holds for $(\mc{O}_N\cap\mc{N}_L)_\red$. Now, there are only finitely many $L(\ov{\Q})$ orbits in $(\mc{O}_N\cap\mc{N}_L)(\ov{\Q})$ as there are only finitely many $L^\circ$-orbits in $\mathcal{N}_L(\ov{\Q})$. Let $N=N_1,\ldots,N_m$ represent these orbits. By Lemma \ref{lem:stratification-isom} it suffices to show that each $\mc{O}_{L,N_i}$ is open or, as they form a set-theoretic partition of $(\mc{O}_N\cap \mc{N}_L)_\red$, that each is closed. Then, by the Noetherian valuative criterion for properness (see \stacks{0208}) it suffices to show if $R$ is a discrete valuation ring and $f\colon \Spec(R)\to (\mathcal{O}_N\cap\mc{N}_L)_\red$ is a morphism with $f(\eta)\in \mc{O}_{N_i,L}$ then $f(\Spec(R))\subseteq \mc{O}_{N_i,L}$. Assume not, and let $f\colon\Spec(R)\to (\mc{O}_N\cap\mc{N}_L)_\red$ be a morphism such that $f(\eta)\in \mathcal{O}_{L,N_i}(k(\eta))$ and $f(s)\in\mathcal{O}_{L,N_j}(k(s))$ with $i\ne j$. Note that $f$ corresponds to an element $\mb{N}$ in $\mc{N}_L(R)$ which is, as an element of $\wh{\mc{N}}(R)$, lies in $\mc{O}_N(R)$. Let us consider $Z_{L}(\mb{N})$. On the one hand, $Z_{L}(\mb{N})$ cannot be flat, as its generic fiber (resp.\@ special fiber) is a twisted form of $Z_L(N_i)$ (resp.\@ $Z_L(N_j)$) which has dimension $\dim(L)-\dim(\mc{O}_{L,N_i})$ (resp.\@ $\dim(L)-\dim(\mc{O}_{L,N_j})$). Note though that as $f(s)$ lies in $\ov{\mathcal{O}_{L,N_i}}$, whose $\ov{\Q}$-points are unions of $\ov{\Q}$-points of orbits of smaller dimension (cf.\@ \cite[Proposition 1.66]{MilneGroups}), $\dim(\mc{O}_{L,N_j})$ is strictly less than $\dim(\mc{O}_{N_i,L})$, and thus the fibers of $Z_{L}(\mb{N})$ have different dimensions, and so it cannot be flat over $R$ (see \cite[Corollary 14.95]{GortzWedhorn}). On the other hand, $Z_{\widehat{G}}(\mb{N})$ is flat as it is \'etale locally isomorphic to $Z_{\wh{G}}(N)=Z_{\wh{G}}(N)_R$. But, by Lemma \ref{lem:fixed-points-reductive} this implies that $Z_{\widehat{G}}(\mb{N})^{\phi(I_F)}=Z_{L}(\mb{N})$ is flat, which is a contradiction.

As $(\mc{O}_N\cap \mc{N}_L)_\red\to \mc{O}_N\cap\mc{N}_L$ is a homeomorphism, there is a scheme-theoretic decomposition $\mc{O}_N\cap\mc{N}_L=\bigsqcup_i U_i$ where $U_i$ is the open subscheme of $\mc{O}_N\cap\mc{N}_L$ with underlying space $\mc{O}_{L,N_i}$. As these schemes are Noetherian, to finish it suffices to show that for all $i$ and all Noetherian $\ov{\Q}$-algebras $A$ every morphism $\Spec(A)\to U_i$ factorizes through $\mc{O}_{L,N_i}$. As $\mc{O}_N=\mc{O}_{N_i}$ we may assume without loss of generality that $i=1$, and so $N_i=N$. Let $\mb{N}$ be the element of $\mf{l}_A$ coresponding to $\Spec(A)\to U_i$. We must then show that \'etale locally on $A$, $\mb{N}$ is conjugate to $N$. Let $I$ denotes the nilradical of $A$, and write $A_0=A/I$. As $A$ is Noetherian, $I^m=(0)$ for some $m$, and thus by inducting we may assume that $I^2=(0)$. Now, as $A_0$ is reduced the map $\Spec(A_0)\to U_i$ factorizes through $\mc{O}_{L,N}$ and thus $\mb{N}_{A_0}$ is \'etale locally conjugate to $N$. As the \'etale covers of $A$ and $A_0$ are equivalent (see \stacks{04DY}), and we are free to work \'etale locally on $A$, we may assume without loss of generality that $\Ad(l_0)(\mb{N}_{A_0})=N$ for some $l_0$ in $L(A_0)$. As $L$ is smooth, we may apply the infinitesimal lifting criterion to find a lift $l$ in $L(A)$ of $l_0$. Replacing $\mb{N}$ by $\Ad(l)(\mb{N})$ we may assume without loss of generality that $\mb{N}_{A_0}=N$. Now, as $\Transp_{\wh{G}}(\mb{N},N)\to \Spec(A)$ is a $Z_{\wh{G}}(N)$-torsor, and thus smooth, we know by the infinitesimal lifting criterion that there exists some $g$ in $\Transp_{\wh{G}}(\mb{N},N)(A)$ lifting the identity. Using the notation of \cite[II, \S4, \textnumero 3, 3.7]{DemazureGabriel}, we may write $g=e^x$ for $x$ in $I\wh{\mf{g}}_A$. Then, by \cite[II, \S4, \textnumero 4, 4.2]{DemazureGabriel} we have
\begin{equation*}
    N=\Ad(g)(\mb{N})=\mb{N}+\ad(x)(\mb{N}). 
\end{equation*} 
As $N$ and $\mb{N}$ lie in $\mf{l}_A$, they are invariant for the action of the finite group $\phi(I_F/I_K)$, and so if $y$ denotes the average of $x$ over the action of $\phi(I_F/I_K)$ then 
\begin{equation*}
    N=\mb{N}+\ad(y)(\mb{N}). 
\end{equation*}
But, by loc.\@ cit.\@ this right-hand side is equal to $\Ad(e^y)(\mb{N})$. By Lemma \ref{lem:fixed-points-reductive} we see that $e^y$ lies in $L(A)$, from where the claim follows.\end{proof}

Let us now denote by $(\gamma^\univ,\phi^\mathrm{univ},N^\mathrm{univ})$ the universal object over $X(\phi,N)$. Consider the transporter scheme $ \Transp_{\widehat{G}}(\phi^\mathrm{univ},\phi)\to \underline{Z}^1(I_F/I_K,\wh{G})$ and set $T$ to be the pullback to $X(\phi,N)$. Set $b\colon T\to X(\phi,N)$ to be the tautological map, which is smooth as $T$ is visibly an $L$-torsor. Note that we have a morphism $a\colon T\to \mathcal{O}_N\cap \mathcal{N}_L$ given by $a(g)= \mathrm{Ad}(g)(N^\mathrm{univ})$ and observe then that we have a scheme-theoretic decomposition $T=\bigsqcup_i a^{-1}(\mathcal{O}_{L,N_i})$. But, for each $i$ we also have a map $\kappa_i\colon a^{-1}(\mathcal{O}_{L,N_i})\to \pi_0(Z_{\widehat{G}}(\phi,N))$ given by sending $g$ to the component containing $\mathrm{Int}(g)(\gamma^\univ)\gamma^{-1}$, and we define for each $i$ and each $c\in \pi_0(Z_{\widehat{G}}(\phi,N))$  the open subscheme $U_{i,c}\defeq\kappa_i^{-1}(c)$ of $a^{-1}(\mathcal{O}_{L,N_i})$. We then obtain a decomposition $T=\bigsqcup_{i,c}U_{i,c}$. 

As $b\colon T\to X(\phi,N)$ is smooth, we see that $b(U_{1,\mathrm{id}})$ is an open subset of $X(\phi,N)$ whose $A$-points are precisely (by \cite[Corollaire 17.16.3.(ii)]{EGA4-4}) the set of $A$-points $(\gamma',\phi',N')$ of $X(\phi,N)$ which are \'etale locally in the image of $b$. It is simple to see that this implies that $U(\gamma,\phi,N)=b(U_{1,\mathrm{id}})$, which implies $U(\gamma,\phi,N)$ is representable by an open immersion. 

Finally, to show that $U(\gamma,\phi,N)$ is smooth and irreducible consider the natural morphism $\wh{G}\times Z_{\phi,N}^\circ\to U(\gamma,\phi,N)$. To simplify notation let us write $S=\wh{G}\times Z_{\phi,N}^\circ$. Note that, by definition, $S\to U(\gamma,\phi,N)$ is surjective as \'etale sheaves and thus a fortiori surjective as schemes, and thus $U(\gamma,\phi,N)$ is irreducible. To see that $U(\gamma,\phi,N)$ is smooth, note that as $S\to U(\gamma,\phi,N)$ is surjective as \'etale sheaves there exists an \'etale cover $V\to U(\gamma,\phi,N)$ such that $p\colon S_V\to V$ admits a section. Note though that as $S_V\to S$ is \'etale and the target is reduced, so is the source (see \stacks{025O}). But, as $p$ has a section, this implies that $V$ is reduced as the morphism of sheaves of rings $\mc{O}_V\to p_\ast \mc{O}_S$ has a section and thus is injective. This implies that $U(\gamma,\phi,N)$ is reduced by \stacks{033F}. But, as we're in characteristic $0$, this implies that $U(\gamma,\phi,N)$ is generically smooth over $\ov{\Q}$ (see \stacks{056V}). But, as $S(\ov{\Q})$ acts $U(\gamma,\phi,N)$ by scheme automorphisms acting transitively on $U(\gamma,\phi,N)(\ov{\Q})$ we deduce that every point of $U(\gamma,\phi,N)(\ov{\Q})$ has regular local ring, and thus $U(\gamma,\phi,N)$ is smooth over $\ov{\Q}$ as desired (see \stacks{0B8X}). This completes the proof of Proposition \ref{prop:wd-loc-mov-is-open}.

\section{The moduli space of $L$-parameters and the Jacobson--Morozov morphism}\label{s:L-param}

In this  section we define the moduli space $\LP_G^K$ of $L$-parameters for $G$, show it has favorable geometric properties, construct the Jacobson--Morozov morphism $\LP_G^K\to \WDP_G^{K,\sqcup}$, and show that an analogue of Theorem \ref{thm:JM-params-classical} holds for any $\Q$-algebra $A$.

\subsection{The moduli space of $L$-parameters}\label{ss:L-param-def} We begin with a slight modification of the Langlands group scheme $\mc{W}_F\times \SL_{2,\Q}$ better suited to arithmetic discussions over $\Q$. Specifically, as in the case of the $C$-group, this concept allows us to avoid extraneous choices of a square root of $q$ (e.g\@ the embedding $\iota$ below).
\begin{defn}We call the $\Q$-scheme representing the functor
\begin{equation*}
        \cat{Alg}_\Q\to\cat{Grp},\qquad A\mapsto \left\{(w,g)\in \mc{W}_F(A)\times \GL_2(A) :  \|w\|=\det(g)\right\}
    \end{equation*}
the \emph{twisted Langlands group scheme} and denote it $\mc{L}^\tw_F$.
\end{defn}
    
To justify the naming of $\mc{L}^\tw_F$, note that if $k$ is any extension of $\Q$ and $c$ is any element of $k$ such that $c^2=q$, then the morphism 
\begin{equation*}
    \eta_c\colon \mc{W}_{F,k}\times \SL_{2,k}\to \mc{L}^\tw_{F,k},\qquad (w,g)\mapsto \left(w,g\left(\begin{smallmatrix}c^{-d(w)} & 0
     \\ 0 & c^{-d(w)}\end{smallmatrix}\right)\right),
\end{equation*}
is an isomorphism. For future reference, we observe that we have a morphism
\begin{equation*}
    p_{\tw} \colon \mc{L}^\tw_F \to \bG_{m,\bQ}\times \cW_F  , \qquad 
    (w,g) \mapsto (\|w\|,w) . 
\end{equation*}
Let us also observe that there is a natural embedding of group schemes $\SL_{2,\Q}\to \mc{L}^\tw_F$ given by sending $g$ to $(1,g)$, as well as an embedding 
\begin{equation*}
    \iota\colon \mc{W}_F\to \mc{L}^\tw_F\qquad w\mapsto\left(w,\left(\begin{smallmatrix}\|w\| & 0\\ 0 & 1\end{smallmatrix}\right)\right).
\end{equation*} 
With these embeddings, we shall implicitly think of $\SL_{2,\Q}$ and $\mc{I}_F$ as subfunctors of $\mc{L}^\tw_F$. Finally, we observe that the embedding of $\mc{W}_{K}$ into $\mc{W}_F$ for any finite extension $K$ of $F$ gives rise to an embedding of $\mc{L}^\tw_{K}\to \mc{L}^\tw_F$ which we implicitly use to think of $\mc{L}^\tw_{K}$ as a subgroup scheme of $\mc{L}^\tw_F$. 

\begin{defn} For a $\Q$-algebra $A$ we define an \emph{$L$-parameter over $A$} to be a homomorphism of group $A$-schemes $\psi\colon \mc{L}^\tw_{F,A}\to \CG_A$ such that $p_C \circ \psi=p_{\tw}$.
\end{defn}

Denote by $\LP_G(A)$ the set of $L$-parameters over $A$, which is functorial in $A$. Note that $\LP_G$ has a natural conjugation action by $\widehat{G}$ and so one has the centralizer group presheaf $Z_{\wh{G}}(\psi)$.

For an $L$-parameter $\psi$ over $A$ we define the morphism $\check{\psi}\colon \mc{L}^\tw_{F,A}\to \check{G}_A$ as the composition of $\psi$ with the projection $\CG_A\to \check{G}_A$. We denote by $\ov{\psi}$ the homomorphism of group $A$-schemes $\mc{L}^\tw_{F,A}\to (\wh{G}\rtimes\underline{\Gamma_\ast})_A$ obtained by composing $\psi$ with the quotient homomorphism $\check{G}_A\to (\wh{G}\rtimes\underline{\Gamma_\ast})_A$. Let us observe that while $\check{\psi}$ may not be a homomorphism, it becomes so after restriction to $\mc{L}^\tw_{F^\ast,A}$. Finally, by our assumptions on $\psi$ the restriction to $\SL_{2,A}$ takes values in $\wh{G}_A$ and we denote this resulting morphism $\SL_{2,A}\to\wh{G}_A$ by $\theta$ (or $\theta_\psi$ when we want to emphasize $\psi$).

To relate this to more familiar objects, fix $k$ to be an extension of $\Q$ containing  an element $c$ such that $c^2=q$. For a $k$-algebra $A$, we endow $\widehat{G}(A)$ with the discrete topology and set
\begin{equation*}
    \LP'_{G,k}(A)\defeq \left\{W_F\times \SL_2(A) \xrightarrow{\psi} \wh{G}(A) \rtimes W_F : \begin{aligned}(1)&\,\, \psi\text{ is a  homomorphism over $W_F$},\\
    (2)&\,\, W_F \stackrel{\psi|_{W_F}}{\to} \wh{G}(A) \rtimes W_F \to \wh{G}(A) \textrm{ is continuous,}\\
    (3)&\,\, \psi|_{\SL_2(A)}\colon \SL_2(A) \to \wh{G}(A) \text{ is algebraic}\end{aligned}\right\}.
\end{equation*}
There is a morphism $i_c^{\mathrm{L}} \colon \LP'_{G,k}\to \LP_{G,k}$ constructed as follows. Fix a $k$-algebra $A$ and an element $\psi'$ of $\mathsf{LP}'_{G,k}(A)$. Note that the restriction $\psi'|_{\SL_2(A)}$ has an algebraization by assumption, call it $\theta_{\psi'}$ which we interpret as a map to $\LG_A$. The projection of $\psi'|_{W_F}$ onto the first coordinate factorizes set-theoretically through $W_F/N$ for some open normal subgroup $N$. This map $W_F/N\to \widehat{G}(A)$ induces a map $\underline{W_F/N}\to \wh{G}_A$ and as $\underline{W_F/N}$ is a quotient of $\mc{W}_{F,A}$, we obtain a morphism $\alpha_{\psi'} \colon \mc{W}_{F,A}\to \LG_A$ whose first projection is the composition $\mc{W}_F\to \underline{W_F/N}\to \wh{G}_A$ and whose second projection is the identity. We can then define $\psi_1 \colon \mc{W}_{F,A}\times\SL_{2,A}\to \LG_A$ as the map $(w,g)\mapsto \alpha_{\psi'}(w)\theta_{\psi'}(g)$, which is well-defined as $\alpha_{\psi'}$ and $\theta_{\psi'}$ commute by Proposition \ref{prop:hom-schem-omnibus}. We then define $i^{\mathrm{L}}_c(\psi')$ to be $i_c\circ \psi_1\circ \eta_c^{-1}$. This construction is independent of all choices, and is functorial. We can show the following proposition in the same way as Proposition \ref{prop:WD-C-L-comparison}.

\begin{prop}\label{L-C-L-comparison} The morphism $i_c^{\mathrm{L}} \colon \LP'_{G,k}\to \LP_{G,k}$ is an isomorphism.
\end{prop}

For a finite extension $K$ of $F^\ast$ Galois over $F$ define 
\begin{equation*}
    \LP_G^K(A)\defeq \left\{\psi\in\LP_G(A): \mc{I}_K\subseteq \ker\left(\check{\psi}|_{\mc{L}^\tw_{F^\ast,A}}\right)\right\},
\end{equation*}
which clearly forms a subpresheaf of $\LP_G$. We have the equality of presheaves $\LP_G=\varinjlim_K \LP_G^K$. As in the case of Weil--Deligne parameters, may associate to an $L$-parameter $\psi$ in $\LP_G^K(A)$ an element $\phi$ of $\underline{Z}^1(I_F/I_K,\wh{G})(A)$ and thus obtain a morphism of presheaves $\LP_G^K\to \underline{Z}^1(I_F/I_K,\wh{G})$.

Fix a lift $w_0$ of arithmetic Frobenius in $W_F$ and define a morphism of presheaves
\begin{equation*}
    j_{w_0}\colon \LP_G^K\to \check{G}\times \underline{Z}^1(I_F/I_K,\wh{G})\times \underline{\Hom}(\SL_{2,\Q},\wh{G}),\qquad \psi\mapsto \left(\check{\psi}\left(w_0,\left(\begin{smallmatrix}q & 0\\ 0 & 1\end{smallmatrix}\right)\right),\phi,\theta\right).
\end{equation*}
On the other hand, we have a diagram
\begin{equation*}
   \mc{D}^L\colon \xymatrix{\check{G}\times \underline{Z}^1(I_F/I_K,\wh{G})\times \underline{\Hom}(\SL_{2,\Q},\wh{G}) \ar@<-.5ex>[r] \ar@<.5ex>[r]& \underline{\Hom}(I_F/I_K,\wh{G})\times \bb{G}_{m,\Q}\times \underline{\Hom}(\SL_{2,\Q},\wh{G})^{[I_F:I_K]+1}}
\end{equation*}
given by the two maps 
\begin{equation*}
    \begin{aligned} (g,f,\nu)\mapsto & \bigg(\Int(g,w_0)\circ f,p_{\bG_m}(g), (\mathrm{Int}(f(i))\circ \nu)_{i\in I_F/I_K},\Int(g,w_0)\circ \nu\bigg)\\  (g,f,\nu) \mapsto &\bigg(f\circ \Int(w_0),q,(\nu)_{i\in I_F/I_K},\nu\circ \Int\left(\left( w_0,\left(\begin{smallmatrix}q & 0\\ 0 & 1\end{smallmatrix}\right)\right)\right)\bigg).\end{aligned}
\end{equation*} 

We then have the following explicit description of $\LP_G^K$.

\begin{prop}\label{prop:L-finite-level-rep} The morphism $j_{w_0}$ gives an identification of $\LP_G^K$ with $\mathrm{Eq}(\mc{D}^L)$. In particular, $\LP_G^K$ is representable by a finite type affine $\Q$-scheme and $j_{w_0}$ is a closed embedding.
\end{prop}

As already observed, for an extension $K\subseteq K'$ of finite extensions of $F^\ast$ Galois over $F$, there is a restriction morphism $\underline{Z}^1(I_F/I_{K'},\wh{G})\to \underline{Z}^1(I_K/I_{K'},\wh{G})$ which is a clopen embedding, and thus $\LP_G^K\to\LP_{G}^{K'}$ is also a clopen embedding. As we have the identification of presheaves $\LP_G=\varinjlim_K \LP_G^K$ we deduce from Proposition \ref{prop:wd-finite-level-rep} that $\LP_G$ is representable by a scheme locally of finite type over $\Q$, all of whose connected components are affine.

\subsection{Decomposition into connected components}\label{ss:L-param-conn-comp-decomp} We now establish the analogue of Theorem \ref{thm:WD-const-decomp} for $\LP_G$. Let us fix $K$ a finite extension of $F^\ast$ Galois over $F$, and a lift $w_0$ of arithmetic Frobenius. Then, by Proposition \ref{prop:L-finite-level-rep} we have an identification $j_{w_0}$ of $\LP_G^K(\ov{\Q})$ with
\begin{equation*}
    \left\{(\gamma,\phi,\theta)\in \check{G}(\ov{\Q})\times \underline{Z}^1(I_F/I_K,\wh{G})(\ov{\Q})\times \underline{\Hom}(\SL_{2,\Q},\wh{G})(\ov{\Q}):\begin{aligned}(1)&\quad \Int(\gamma,w_0)\circ \phi=\phi\circ\Int(w_0),\\ (2)&\quad p_{\bG_m}(\gamma)=q,\\ (3)&\quad \Int(\phi(i))\circ \theta=\theta\text{ for all }i\in I_F/I_K,\\ (4)&\quad \Int((\gamma,w_0))\circ \theta = \theta\circ \Int\left(\left( w_0, \left(\begin{smallmatrix}q & 0\\ 0 & 1\end{smallmatrix}\right)\right)\right)
    \end{aligned}\right\}.
\end{equation*}
Now, for $(\gamma,\phi,\theta)$ in $\LP_G^K(\ov{\Q})$ let us define $Z_{\phi,\theta}$ to be $Z_{\wh{G}}(\phi,\theta)$. This is a linear algebraic group over $\ov{\Q}$ whose identity component is reductive. Let us then say that an element $(\gamma',\phi',\theta')$ in $\LP_G^K(A)$, for a $\ov{\Q}$-algebra $A$, is \emph{locally movable to $(\gamma,\phi,\theta)$} if there exists an \'etale cover $\Spec(A')\to\Spec(A)$ and $(g,h)\in(\wh{G}\times Z_{\phi,\theta}^\circ)(A')$ such that $(\gamma',\phi',\theta')=g(h\gamma,\phi,\theta)g^{-1}$. As this definition is clearly functorial, we obtain a subpresheaf of $\LP_{G,\ov{\Q}}^K$ as follows:
\begin{equation*}
    U(\gamma,\phi,\theta)(A)\defeq \left\{(\gamma',\phi',\theta')\in\LP_{G,\ov{\Q}}^K(A): (\gamma',\phi',\theta')\text{ is locally movable to }(\gamma,\phi,\theta)\right\}.
\end{equation*}
We then have the following, whose proof is identical to Proposition \ref{prop:wd-loc-mov-is-open} except the analogue of Lemma \ref{lem:intersection-decomp} is simpler since for any closed subgroup scheme $L$ of $\widehat{G}_{\overline{\mathbb{Q}}}$ with reductive identity component, $\underline{\Hom}(\SL_{2,\overline{\mathbb{Q}}},L)$ is the disjoint union of the orbit schemes under the conjugation action of $L$ by Proposition \ref{prop:SL2-Hom-open-orbits}.

\begin{prop}\label{prop:l-loc-mov-is-open} The morphism of presheaves $U(\gamma,\phi,\theta)\to \LP_{G,\ov{\Q}}^K$ is representable by an open immersion. Moreover, the $\ov{\Q}$-scheme $U(\gamma,\phi,\theta)$ is smooth and irreducible.
\end{prop}

Define an equivalence relation on $\LP_G^K(\ov{\Q})$ by declaring that $(\gamma,\phi,\theta)$ is equivalent to $(\gamma',\phi',\theta')$ if there exists some $(g,h)\in (\wh{G}\times Z_{\phi,\theta})(\ov{\Q})$ such that $(\gamma',\phi',\theta')=g(h\gamma,\phi,\theta)g^{-1}$. Let us denote an equivalence class under this relation by $[(\gamma,\phi,\theta)]$. Observe that here we do not require $h$ to lie in $Z^\circ_{\phi,\theta}(\ov{\Q})$, so that these equivalence classes differ from $U(\gamma,\phi,\theta)(\ov{\Q})$. For each such equivalence class, let us choose an element $(\gamma,\phi,\theta)$. We consider $\pi_0(Z_{\phi,\theta})$ as a finite abstract group, and we define an equivalence relation on it by declaring that $c$ is equivalent to $c_1 c \gamma c_1^{-1} \gamma^{-1}$ for any $c_1$ in $\pi_0(Z_{\phi,\theta})$. We denote by $[c]$ an equivalence class for this relation. 

We then have the following decomposition of $\LP_{G,\ov{\Q}}^K$ into explicit connected components, whose proof is exactly the same as that of Theorem \ref{thm:WD-const-decomp}.

\begin{thm}\label{thm:L-const-decomp} The choice of $(\gamma,\phi,\theta)$ in each class $[(\gamma,\phi,\theta)]$ of $\LP_G^K(\ov{\Q})$ gives an identification
\begin{equation*}
    \LP_{G,\ov{\Q}}^K=\bigsqcup_{[(\gamma,\phi,\theta)]}\bigsqcup_{[c]}\,\,U(c\gamma,\phi,\theta).
\end{equation*}
\end{thm}

We derive from this two corollaries neither of which is a priori obvious.

\begin{cor} For all $(\gamma,\phi,\theta)$ in $\LP_G^K(\ov{\Q})$ the $\ov{\Q}$-scheme $U(\gamma,\phi,\theta)$ is affine.
\end{cor}
\begin{proof} By Proposition \ref{prop:L-finite-level-rep} the scheme $\LP_{G,\ov{\Q}}^K$ is affine, and thus so is the clopen subset $U(\gamma,\phi,\theta)$.
\end{proof}

Denote the set of equivalence classes for $\LP_G^K(\ov{\Q})$ (resp.\@ $\pi_0(Z_{\phi,\theta})$) by $[\LP_G^K(\ov{\Q})]$ (resp.\@ $[\pi_0(Z_{\theta,N})]$).

\begin{cor}\label{cor:LP-pi0} The affine $\Q$-scheme $\LP_G^K$ is smooth, and there is a non-canonical $\Gamma_\Q$-equivariant bijection
\begin{equation*}
    \pi_0\left(\LP_{G,\ov{\Q}}^K\right)\isomto \left\{([(\gamma,\phi,\theta)],[c]):\begin{aligned}(1)&\quad [(\gamma,\phi,\theta)]\in \left[\LP_G^K(\ov{\Q})\right]\\ (2)&\quad [c]\in [\pi_0(Z_{\phi,\theta})] \end{aligned}\right\}
\end{equation*}
where the $\Gamma_\Q$ action on the target is inherited from $\LP_G^K$ and $\wh{G}$.
\end{cor}
\begin{proof} By Proposition \ref{prop:l-loc-mov-is-open}, each $U(\gamma,\phi,\theta)$ is smooth and connected, and thus the disjoint union, which is $\LP^K_{G,\ov{\Q}}$, is smooth and the claim concerning connected components follows.
\end{proof}

\subsection{The Jacobson--Morozov morphism}\label{ss:JM-mor}
We now come to the definition of the Jacobson--Morozov map in the geometric setting. 

\begin{defn} The morphism $\JM\colon \LP_G \to \WDP_G$ given by sending $\psi$ to $(\psi\circ \iota,d\theta_\psi(e_0))$ is called the \emph{Jacobson--Morozov morphism}.
\end{defn}

It is clear that $\JM$ is $\wh{G}$-equivariant. By Theorem \ref{thm:geom-JM-split} it is also clear that $\JM$ factorizes uniquely through $\WDP^\sqcup_G$. Moreover, for any finite extension $K$ of $F^\ast$ Galois over $F$, one sees that $\JM^{-1}(\WDP_G^K)$ is precisely $\LP_G^K$ and so we get factorizations $\LP_G^K\to \WDP_G^K$ and $\LP_G^K\to\WDP_G^{K,\sqcup}$. We denote all these factorizations also by $\JM$.

Observe that over $\ov{\Q}$ we may give a simpler description of the Jacobson--Morozov morphism on each connected component. Namely, let us fix $(\gamma,\phi,\theta)$ in $\LP_G^K(\ov{\Q})$ as in the notation of \S\ref{ss:L-param-conn-comp-decomp}. Then, first observe that $\JM(\gamma,\phi,\theta)$ is equal to $\left(\gamma,\phi,N\right)$ where $N= \JM(\theta)$. We may then observe that $\JM$ restricted to $U(\gamma,\phi,\theta)$ maps into $U(\gamma,\phi,N)$ and is the \'etale sheafification of the map which on $A$-points is the map
\begin{equation*}
    \left\{g(h\gamma,\phi,\theta)g^{-1}:(g,h)\in \wh{G}(A)\times Z_{\phi,\theta}^\circ(A)\right\}\to \left\{g(h'\gamma,\phi,N)g^{-1}:(g,h')\in \wh{G}(A)\times Z_{\phi,N}^\circ(A)\right\}
\end{equation*}
given by sending $g(h\gamma,\phi,\theta)g^{-1}$ to $g(h\gamma,\phi,N)g^{-1}$. 

We also observe that if $k$ is an extension of $\Q$ and $c$ is an element of $k$ such that $c^2=q$ then under the isomorphisms described in Proposition \ref{prop:WD-C-L-comparison} and Proposition \ref{L-C-L-comparison} the Jacobson--Morozov morphism corresponds to the morphism $\LP'_{G,k}\to \WDP'_{G,k}$ which on $A$-points sends $\psi$ to $(\psi\circ \iota'_A,d\theta_\psi(e_0))$ where $\iota'_A$ is the map
\begin{equation*}
    \iota'_A\colon W_F\to W_F\times \SL_2(A) ,\qquad w\mapsto \left(w,\left(\begin{smallmatrix} c^{-d(w)} & 0\\ 0 & c^{d(w)}\end{smallmatrix}\right) \right),
\end{equation*}
and by $\theta_\psi$ we mean the map $\SL_{2,A}\to \wh{G}_A$ associated to the (unique) algebraization of $\psi|_{\SL_2(A)}$. So, on the level of $\C$-points we see that our Jacobson--Morozov map agrees with that from \S\ref{ss:JM-for-params-classical}.

We now move towards stating the analogue of Theorem \ref{thm:JM-params-classical} at the level of $A$-points. To begin, we must define the notion of semi-simplicity for $L$-parameters in the relative setting.

\begin{prop}\label{prop:Frob-factor-psi}
Let $\psi$ be an L-parameter over a $\mathbb{Q}$-algebra $A$. 
Then there is a positive integer $m$ divisible by $[F^\ast:F]$ such that the morphism 
\begin{equation*}
    \cW_{F,A}\to \check{G}_A,\qquad w \mapsto \check{\psi} \left(w^{2m},\left(\begin{smallmatrix}q^{-md(w)} & 0\\ 0 & q^{-md(w)}\end{smallmatrix}\right)\right)
\end{equation*} admits a factorization 
\begin{equation*}
    \cW_{F,A} \stackrel{d}{\lra} \underline{\mathbb{Z}}_A \stackrel{\check{\psi}_m}{\lra} \check{G}_A .
\end{equation*}
\end{prop}
\begin{proof}
This is proved in the same way as Proposition \ref{prop:Frob-factor}. 
\end{proof}

\begin{defn}\label{defn:Frob-ss-L-param} For $A$ a $\Q$-algebra, we call an element $\psi$ of $\LP_G(A)$ \emph{Frobenius semi-simple} if there exists an integer $m$ as in Proposition \ref{prop:Frob-factor-psi}  such that $\check{\psi}_m$ factors through a subtorus of $\check{G}_A$ etale locally on $A$. 
\end{defn}

Let us denote by $\LP^\ss_G(A)$ (resp.\@ $\LP_G^{K,\ss}(A)$) the subset of Frobenius semi-simple elements of $\LP_G(A)$ (resp.\@ $\LP_G^K(A)$). This evidently forms a $\wh{G}$-stable subfunctor of $\LP_G$ (resp.\@ $\LP_G^K$). 

\begin{rem}To understand the reasoning for this definition, observe that under the isomorphism in Proposition \ref{L-C-L-comparison}, this condition corresponds to an element $\psi'$ of $\LP'_{G,k}(A)$ satisfying the property that the projection of $\psi'(w_0^{2m},1)$ to $\wh{G}(A)$ is semi-simple for some $m$ as in Proposition \ref{prop:Frob-factor-psi}. In particular, this notion of semi-simple agrees with that from \S\ref{ss:JM-for-params-classical} for $\C$-points by Lemma \ref{lem:L-group-ss}. 
\end{rem}

We now prove the following surprisingly subtle semi-simplicity preservation property for the Jacobson--Morozov morphism.

\begin{prop}\label{prop:JM-preserves-ss} Let $A$ be a $\Q$-algebra and $\psi$ an element of $\LP_G(A)$. Then, $\psi$ is Frobenius semi-simple if and only if $\mathsf{JM}(\psi)$ is.
\end{prop}
\begin{proof}Suppose that $\psi$ is Frobenius semi-simple. As the conclusion is insensitive to passing to an \'etale extension and conjugating, we do so freely. Take $m$ as in Proposition \ref{prop:Frob-factor-psi} and a split maximal torus $T$ of $\check{G}_A$ such that $\check{\psi}_m$ factors through $T$. 
Note that the eigenspace $\check{\mf{g}}_A(1)$ with respect to $\check{\psi}_m(1)$ is the Lie algebra of a Levi subgroup $L$ of $\check{G}_A$ such that $\check{\psi}_m$ factors through $Z(L)$. Indeed, we may assume that $T=(T_0)_A$ for a maximal torus $T_0$ of $\check{G}$. Let $L'$ be the Levi subgroup of $\check{G}$ generated by $T_0$ and the root groups for the roots $\alpha$ which annihilate $\check{\psi}_m(1)$. Then, we may take $L= L'_A$, where $\check{\psi}_m$ factors through $Z(L)$ by \cite[Corollary 3.3.6]{ConRgrsch}.

Note that $\theta$ factorizes through $L$ as by Proposition \ref{prop:hom-schem-omnibus} it suffices to check this on the level of Lie algebras, from where it is clear.
Let $T_2$ denote the standard diagonal subtorus of $\SL_{2,A}$. Since $\theta$ factorizes through $L$, by \cite[Lemma 5.3.6]{ConRgrsch} we may assume that the map $\theta|_{T_2}$ factorizes through a maximal torus $T'$ of $L$. But, as $Z(L)\subseteq T'$ both $\theta|_{T_2}$ and $\check{\psi}_m$ factorize through $T'$. Hence, if we write $\mathsf{JM}(\psi)=(\varphi,N)$ then the morphism $\mc{W}_{F,A} \to \check{G}_A$ given by $w \mapsto \varphi(w^m)$ factors through $T'$. This implies that  $\mathsf{JM}(\psi)$ is Frobenius semi-simple. 

Conversely, suppose that $\mathsf{JM}(\psi)=(\varphi,N)$ is Frobenius semi-simple. Let $m$ be any integer as Proposition \ref{prop:Frob-factor}. 
As above, we may build a reductive subgroup $L_m$ of $\check{G}_A$ such that $\mathrm{Lie}(L_m)$ is identified with $\check{\mf{g}}_A(1)$ with respect to $\check{\varphi}_m(1)$. We claim that the group $L_{km}$ stabilizes for $k$ sufficiently large. Indeed, the roots of $\alpha$ of $\check{G}$ relative to $T_0$ that annihilate $\check{\varphi}_{km}(1)=\check{\varphi}_m(1)^k$ stabilize for $k$ sufficiently large, from where the claim follows by the construction. Denote by $L$ the group $L_{km}$ for $k$ sufficiently large, say for $k\geqslant k_0$. Let us write $Z$ for the torus $Z(L)^{\circ}$ (see \cite[Theorem 3.3.4]{ConRgrsch}). Observe that as $\check{\varphi}_{km}$, for $k\geqslant k_0$, centralizes $\Lie(L)$ that $\check{\varphi}_{km}$ factors through $Z(L)$. So then, for some $k_1\geqslant k_0$ we have that $\check{\varphi}_{k_1m}$ factors through $Z$. We put $m_1=k_1 m$.  We will be done if we can show that $\theta|_{T_2}$ factorizes through the reductive group $A$-scheme $Z_{\check{G}}(Z)$ (see \cite[Lemma 2.2.4]{ConRgrsch} and \cite[Corollary 17.59]{MilneGroups}). Indeed, in this case by \cite[Lemma 5.3.6]{ConRgrsch}, we know that after passing to an \'etale extension, $\theta|_{T_2}$ factorizes through a maximal torus $T'$ of $Z_{\check{G}}(Z)$. Then $\theta|_{T_2}$ and $\check{\varphi}_{m_1}$ factor through $T'$. 
Hence 
\begin{equation*}
\mc{W}_{F,A} \to \check{G}_A,\qquad w \mapsto \check{\psi} \left(w^{2m_1},\left(\begin{smallmatrix}q^{-m_1d(w)} & 0\\ 0 & q^{-m_1d(w)}\end{smallmatrix}\right)\right) 
\end{equation*}
factors through $T'$. This implies that $\psi$ is Frobenius semi-simple. 

Working etale locally, and by passing to a $\check{G}(A)$-conjugate, we may assume that $Z$ is equal to $Z'_A$ for a split subtorus $Z'$ of $\check{G}$. Let $R_0$ be the set of nontrivial characters of $Z'$ appearing in the adjoint action of $Z'$ on $\check{\mf{g}}_A$. Note that these characters are already defined over $\mathbb{Q}$. Consider the functor on $\cat{Alg}_\Q$ with
\begin{equation*}
    Y(B)\defeq \left\{z\in Z'(B):\begin{aligned}(1)&\quad \chi(z)\ne 1\text{ for all }\chi\in R_0,\\ (2)&\quad \chi(z)=q^{m_1}\text{ for all }\chi\in R_0\text{ such that }\chi(\check{\varphi}_{m_1}(1))=q^{m_1}\end{aligned}\right\}.
\end{equation*}
Clearly $Y$ defines a locally closed subscheme of $Z'$ which is non-empty as $\check{\varphi}_{m_1}(1)$ is an element of $Y(A)$. Take $y \in Y(F)$ for a finite extension $F$ of $\mathbb{Q}$. By passing to an \'etale extension, we may assume that $A$ contains $F$. We claim that inclusion $Z_{\check{G}}(Z)\subseteq Z_{\check{G}}(y)^\circ_A$ is an equality. As $Z_{\check{G}}(Z)$ is flat over $\Spec(A)$, we know from the fibral criterion for isomorphism (see \cite[Corollaire 17.9.5]{EGA4-4}), that it suffices to check this after base change to every point of $\Spec(A)$. But, as $A$ is $\Q$-algebra, and $Z_{\check{G}}(Z)$ and $Z_{\check{G}}(y)^\circ_A$ are both connected, it then suffices to check they have the same Lie algebra (e.g.\@ see \cite[Corollary 10.16]{MilneGroups}), but this is true by construction.

In the following, we use the notation $\check{\mathfrak{g}}_A(\lambda)$ for $\lambda \in A^{\times}$ with respect to $\check{\varphi}_{m_1}(1)$. 
By construction, we know that $\mathrm{Int}(y)$ acts on $\check{\mathfrak{g}}_A(q^{\pm m_1})$ by multiplication by $q^{\pm m_1}$. Moreover, the $\SL_2$-triple $(N,f,h)$ associated to $\theta$ by Theorem \ref{thm:rel-JM-triples} satisfies $N \in \check{\mathfrak{g}}_A(q^{m_1})$, $f \in \check{\mathfrak{g}}_A(q^{-m_1})$ and $h \in \check{\mathfrak{g}}_A(1)$. Therefore, the $\mathfrak{sl}_2$-triple attached to $\mathrm{Int}(y) \circ \theta$ is $(q^{m_1}N,q^{-m_1}f,h)$. Thus, the $\mathfrak{sl}_2$-triple attached to $\Int(y)\circ \theta\circ \mu$ is $(N,f,h)$ where 
\begin{equation*}
\mu\colon \SL_{2,A}\isomto\SL_{2,A},\qquad \begin{pmatrix} a & b \\ c & d \end{pmatrix} \mapsto \begin{pmatrix} a & q^{-m_1}b \\ q^{m_1}c & d \end{pmatrix}.
\end{equation*}
By Theorem \ref{thm:rel-JM-triples} $\mathrm{Int}(y) \circ \theta\circ \mu=\theta$, so $\theta|_{T_2}$ factorizes through $Z_{\check{G}}(y)^{\circ}_A =Z_{\check{G}}(Z)$ as desired.
\end{proof}

We end this section by proving a relative version of Proposition \ref{prop:Zphidec}. Fix a $\Q$-algebra $A$ and let $N$ be an element of $\mc{N}^\sqcup(A)$. Let us denote by $\mf{u}^N$ the $A$-submodule $\mathrm{im}(\ad (N)) \cap \Ker (\ad (N))$ of $\wh{\mf{g}}_A$, which we also treat as a subfunctor of $\wh{\mf{g}}_A$ in the obvious way. Note that $\mf{u}^N$ is in fact a closed subscheme of $\wh{\mc{N}}_A$ and for all $A$-algebras $B$ there is an equality
\begin{equation*}
    \mf{u}^N(B) = \mathrm{im}(\ad (N\otimes 1)) \cap \Ker (\ad (N\otimes 1)).
\end{equation*}
As these claims are \'etale local, we may assume that $N=gN_0g^{-1}$ for some $N_0$ in $\wh{\mc{N}}(\Q)$ and $g$ in $\wh{G}(A)$. Observe then that $\mf{u}^N$ is equal to $g(\mf{u}^{N_0})_A g^{-1}$ where $\mf{u}^{N_0} \subseteq \wh{\mf{g}}$ is defined in the same way as $\mf{u}^N$. As $\wh{\mc{N}}_A$ is $\wh{G}(A)$-equivariant it suffices to show that $\mf{u}^{N_0}$ factorizes through $\wh{\mc{N}}$ which may be checked on $\ov{\Q}$-points which is then clear. One similarly proves the claimed equality. 

As $\mf{u}^N$ is a closed subscheme of $\wh{\mc{N}}_A$, we obtain a closed subscheme $U^N\defeq \exp(\mf{u}^N)$ of $\wh{G}_A$. We claim that $U^N$ is a closed subgroup scheme of $\wh{G}_A$ flat over $A$. As this may be checked \'etale locally we are again reduced to checking that $\exp(\mf{u}^{N_0})$ is a closed subgroup $\Q$-scheme of $\wh{G}$ (automatically flat over $\Q$), but this is true by Proposition \ref{prop:exp-omnibus}. For an element $(\varphi,N)$ of $\WDP^{\sqcup}_G(A)$ we set
\begin{equation*}
    U^N(\varphi)\defeq U^N\times_{\wh{G}_A}Z_{\wh{G}}(\varphi). 
\end{equation*}
Concretely this means that for every $A$-algebra $B$ one has an identification of $U^N(\varphi)(B)$ with $U^N(B)\cap Z_{\wh{G}}(\varphi)(B)$ where this intersection is taken in $\wh{G}(B)$.

Let us first establish the following relative version of Proposition \ref{prop:Zudec}, which follows easily (using the same reduction arguments as already used above) from Proposition \ref{prop:Zudec}

\begin{lem}\label{lem:zudec-rel} Let $\theta$ be an element of $\underline{\Hom}(\SL_{2,\Q},\wh{G})(A)$ and define $N=\JM(\theta)$. Then, $Z_{\wh{G}}(N)=U^N\rtimes Z_{\wh{G}}(\theta) $.
\end{lem}

\begin{prop}\label{prop:Zphidesc-rel} Let $A$ be a $\Q$-algebra, $\psi$ is an element of $\LP_G(A)$, and set $(\varphi,N)=\mathsf{JM}(\psi)$. Then, $Z_{\widehat{G}}(\varphi, N) = U^N(\varphi)\rtimes Z_{\wh{G}}(\psi)$.
\end{prop}
\begin{proof} Let $B$ be an $A$-algebra. Given Lemma \ref{lem:zudec-rel} it clearly suffices to show that conjugation by an element in the image of $\varphi$ stabilizes $U^N$, as the rest of the argument for Proposition \ref{prop:Zphidec} then goes through verbatim. Let $u=\exp(n)$ be an element of $U^N(B)$ and observe that  $\Int(\varphi(w))(u)$ is equal to $\exp(\Ad(\varphi(w))(n))$, and so we are done as clearly $\Ad(\varphi(w))(n)\in \mf{u}^N(B)$.
\end{proof}

\subsection{The relative Jacobson--Morozov theorem for parameters}

We now arrive at the relative analogue of Theorem \ref{thm:JM-params-classical}. Let us set $\WDP^{\sqcup,\ss}_G$ to be the presheaf whose $A$-points consist of Frobenius semi-simple Weil--Deligne parameters $(\varphi,N)$ such that $N$ lies in $\mc{N}^\sqcup(A)$.

\begin{thm}[Relative Jacobson--Morozov theorem for parameters]\label{thm:rel-JM-param} The Jacobson--Morozov morphism $ \mathsf{JM}\colon \LP^\ss_G\to \WDP^{\sqcup,\ss}_G$ is surjective, and induces an isomorphism of quotient presheaves
\begin{equation*}
    \JM\colon \LP^\ss_G/\wh{G}\isomto \WDP^{\sqcup,\ss}_G/\wh{G}.
\end{equation*}
\end{thm}

Let us fix a $\Q$-algebra $A$, an element $(\varphi,N)$ of $\WDP^{\sqcup,\ss}_G(A)$, and an arithmetic Frobenius lift $w_0 \in \cW_{F,A}$. 
In the notation from Proposition \ref{prop:eigen-decomp}, with $\rho\colon (\check{G}\rtimes \underline{\Gamma_\ast})_A\to \GL(\wh{\mf{g}}_A)$ the adjoint action, $h=\ov{\varphi}(w_0)$, and $I=\phi(I_F/I_K)$, let $\mf{h}$ and $\mf{h}(\lambda)$ be $\wh{\mf{g}}_A^I$ and $\wh{\mf{g}}_A^I(\lambda)$ respectively. 

\begin{prop}[{cf.\@ \cite[Lemma 2.1]{GRAinv}}]\label{prop:gr-prop}
 There exists an $\mf{sl}_2$-triple in $\widehat{\mf{g}}_A$ of the form $(N,h,f)$ where $N\in\mf{h}(q)$, $h\in\mf{h}(1)$, and $f\in\mf{h}(q^{-1})$. Moreover, any two such $\mf{sl}_2$-triples are conjugate by an element of $Z_{\widehat{G}}(\varphi,N)$ \'etale locally on $A$.
\end{prop}

\begin{proof} By Theorem \ref{thm:rel-JM-triples} there exists an $\mf{sl}_2$-triple $(N,h_{-1},f_{-1})$ in $\wh{\mf{g}}_A$. We take a finite extension $K$ of $F^\ast$ Galois over $F$ such that $\cI_{K,A} \subseteq \ker (\check{\varphi}|_{\cW_{F^\ast,A}})$. Observe that $N$ is in $\mf{h}$ by definition and if we set $h_0$ to be the average over the action of $\phi(I_F/I_K)$, then $h_0$ is also in $\mf{h}$ and $(N,h_0)$ satisfies the conditions of Proposition \ref{prop:Kostant-triples-prop} for $\mf{h}$. Therefore there exists an $\mf{sl}_2$-triple in $\mf{h}$ of the form $(N,h_0,f_0)$. Given this, the decomposition result from Proposition \ref{prop:eigen-decomp}, and Proposition \ref{prop:Kostant-triples-prop}, the existence argument as in \cite[Lemma 2.1]{GRAinv} goes through without further comment. 

To show the uniqueness part of the statement, let $(N,h,f)$ and $(N,h_1,f_1)$ be two $\mf{sl}_2$-triples satisfying the conditions of the proposition. We shall pass to an \'etale extension freely in the following. By Proposition \ref{prop:SL2-Hom-open-orbits}, we may assume that there exists a morphism $\theta\colon \SL_{2,\Q}\to\widehat{G}$ such that $(N,h,f)$ is the associated $\mf{sl}_2$-triple. Set $\mf{m}\defeq\mf{h}^N\cap \mf{h}(1)$, and for each $i\in\mathbb{N}$ set $\mf{m}_i$ to be $\left\{x\in\mf{m}:[h,x]=ix\right\}$. We can check that $\mf{m}=\bigoplus_{i}\mf{m}_i$ by using the adjoint action of $\theta|_{T_2}$ and Lemma \ref{lem:Gm-Ad-ad}, where $T_2$ is the diagonal subtorus of $\SL_{2,\Q}$. Let us now set $\mf{u}\defeq\bigoplus_{i>0}\mf{m}_i$. Then $\mf{u}$ is Lie subalgebra of $\wh{\mf{g}}_A$ contained in $\wh{\mc{N}}(A)$ as it is contained in $\bigoplus_{i>0}\wh{\mf{g}}_{i,A}$, the base change to $A$ of $\bigoplus_{i>0}\wh{\mf{g}}_i$ where $\wh{\mf{g}}_i=\{x\in \wh{\mf{g}}: [h,x]=ix\}$, and $\bigoplus_{i>0}\wh{\mf{g}}_i$ is quickly checked to be contained in $\wh{\mc{N}}(\Q)$. Consider $U\defeq\exp(\mf{u})$, which is a subgroup of $H(A)$ by (3) of Proposition \ref{prop:exp-omnibus}.

We claim that $\left\{\mathrm{Ad}(u)(h):u\in U\right\}$ is equal to $h+\mf{u}$. To see this, we note that if we write $u=\exp(x)$ for  $x\in\mf{u}$ then by (2) of Proposition \ref{prop:exp-omnibus} $\mathrm{Ad}(u)(h)$ is equal to $\sum_{n \geq 0} \frac{1}{n!}\mathrm{ad}(x)^n(h)$. We need to show that for any $x_0 \in \mf{u}$ there is $x \in \mf{u}$ such that 
$x_0=\sum_{n \geq 1} \frac{1}{n!}\mathrm{ad}(x)^n(h)$. 
We define a filtration $\mathrm{Fil}^i(\mathfrak{u})=\bigoplus_{j\geqslant i}\mathfrak{m}_j$ for $i \geq 1$. It suffices to prove that there is $x_i \in \mf{u}$ such that 
\begin{equation*}
    x_0 \equiv \sum_{n \geq 1} \frac{1}{n!}\mathrm{ad}(x_i)^n(h) \mod \mathrm{Fil}^i(\mathfrak{u})
\end{equation*} 
by induction on $i$. This is trivial for $i=1$. We assume that it is proved for $i$. We take $x_i' \in \mathrm{Fil}^i(\mathfrak{u})$ such that $[x_i',h]=x_0 - \sum_{n \geq 1} \frac{1}{n!}\mathrm{ad}(x_i)^n(h)$. Then $x_{i+1}=x_i +x_i'$ is seen to satisfy 
\begin{equation*}
    x_0 \equiv \sum_{n \geq 1} \frac{1}{n!}\mathrm{ad}(x_{i+1})^n(h) \mod \mathrm{Fil}^{i+1}(\mathfrak{u})
\end{equation*} 
since $[\mathfrak{u},\mathrm{Fil}^{i}(\mathfrak{u})] \subseteq \mathrm{Fil}^{i+1}(\mathfrak{u})$.

Note now that $y=h_1-h=[N,f_1-f]$ is in $\mf{u}$. Indeed, by inspection $[N,y]=0$ so that $y$ is in $\mf{h}^N$, but since $h_1$ and $h$ are both in $\mf{h}(1)$, so is their difference $y$. Note though that as $y=[N,f_1-f]$ we have $y$ is in $\mf{u}$. Indeed, it again suffices to show that $\widehat{\mf{g}}_A\cap [N,\widehat{\mf{g}}_A]$ is equal to $\bigoplus_{i>0}\widehat{\mf{g}}_{i,A}$ which, again, may be verified over $\Q$ in which case it is again classical (cf. \cite[Proposition 2.2]{GRAinv}). Thus, we know that there exists some $u$ in $U$ such that $\mathrm{Ad}(u)(h)=h+y=h_1$. One then verifies that $\Ad(u)(f)=f_1$ as in loc.\@ cit. 

Finally, we now observe that the inclusion $U\subseteq Z_{\widehat{G}}(\varphi,N)(A)$ holds. Indeed, writing $u=\exp(x)$ we see that $\Ad(u)(N)=N$ since $x$ is in $\mf{h}^N$ and using the formula from (2) of Proposition \ref{prop:exp-omnibus}. Similarly, as $\Int(\varphi(w))(\exp(x))$ is equal to $\exp(\Ad(\varphi(w))(x))$, this is just $\exp(x)$ as $x$ is in $\mf{h}(1)$.
\end{proof}

To show the surjectivity claim in Theorem \ref{thm:rel-JM-param} let $(N,f,h)$ be as in Proposition \ref{prop:gr-prop}, and consider the morphism $\theta\colon \SL_{2,A}\to \wh{G}_A$ associated by Theorem \ref{thm:relative-jm}. We then consider the morphism of schemes
\begin{equation*}
\psi\colon \mc{L}^\tw_{F,A}\to \CG_A,\quad (w,g)\mapsto \theta\left(g\left(\begin{smallmatrix}\|w\| & 0\\ 0 & 1\end{smallmatrix}\right)^{-1}\right)\varphi(w).
\end{equation*}
We claim that this a morphism of group $A$-schemes. To prove this, it suffices to show \begin{equation*}
    \Ad(\varphi(w))(\theta (g))=\theta \left( \Ad \left(\left(\begin{smallmatrix}\|w\| & 0\\ 0 & 1\end{smallmatrix}\right)\right)(g)\right)
\end{equation*} 
for $w \in \mc{W}_{F,A}(B)$ and $g \in \SL_2 (B)$, where $B$ is any $A$-algebra. This follows from Proposition \ref{prop:hom-schem-omnibus} and the construction of $\theta$. One then easily check that $\psi$ is an element of $\LP_G(A)$ such that $\mathsf{JM}(\psi)=(\varphi,N)$ as desired.

We now show that $\mathsf{JM}$ induces a bijection $\LP^\ss_G(A)/\wh{G}(A)\isomto \WDP^{\sqcup,\ss}_G(A)/\wh{G}(A)$, which now only requires the demonstration of injectivity. By the $\wh{G}(A)$-equivariance of $\mathsf{JM}$ it suffices to show that if $\psi_1$ and $\psi_2$ are elements of $\LP^\ss_G(A)$ such that $\mathsf{JM}(\psi_1)$ and $\mathsf{JM}(\psi_2)$ both equal $(\varphi,N)$, then $\psi_1$ and $\psi_2$ are $\wh{G}(A)$-conjugate. Note that the $\mf{sl}_2$-triples associated to $\theta_{\psi_i}$ for $i=1,2$ both satisfy the conditions of Proposition \ref{prop:gr-prop} for $(\varphi,N)$. Therefore, \'etale locally on $A$ the $\mf{sl}_2$-triples associated to $\psi_1$ and $\psi_2$ are conjugate in a way that centralizes $(\varphi,N)$ and so $\psi_1$ and $\psi_2$ are \'etale locally conjugate. From this we deduce that $\psi_2$ defines a class in $H^1_\et(\Spec(A),Z_{\widehat{G}}(\psi_1))$ given by $\Transp_{\wh{G}}(\psi_1,\psi_2)$. Note that we have a natural map
\begin{equation*}
    H^1_\et(\Spec(A),Z_{\widehat{G}}(\psi_1))\to H^1_\et(\Spec(A),Z_{\widehat{G}}(\varphi,N))
\end{equation*}
which maps $\Transp_{\wh{G}}(\psi_1,\psi_2)$ to the trivial element, and so $\Transp_{\wh{G}}(\psi_1,\psi_2)$ belongs to
\begin{equation*}
    \ker\bigg(H^1_\et(\Spec(A),Z_{\widehat{G}}(\psi_1))\to H^1_\et(\Spec(A),Z_{\widehat{G}}(\varphi,N))\bigg),
\end{equation*}
and so we are done if this kernel is trivial. But, this map on cohomology groups has a set-theoretic splitting from the semi-direct product decomposition of Proposition \ref{prop:Zphidesc-rel}, and so the claim follows.

\section{Geometric properties of the Jacobson--Morozov map} 

In this final section we use the material developed so far to prove that the Jacobson--Morozov morphism satisfies favorable geometric properties. Namely, we show that $\JM\colon \LP_G^K\to \WDP_G^{K,\sqcup}$ (resp.\@ $\JM\colon \LP_G^K\to \WDP_G^K$) is birational (resp.\@ weakly birational). We do this by exhibiting a more explicit space which embeds into all three moduli spaces weakly birationally. This is the geometric analogue of the reductive centralizer locus from \S\ref{ss:red-loc-classical}. We then finally show that as a particular application of these ideas one may prove that the Jacobson--Morozov map is an isomorphism between the discrete loci in $\LP_G^K$ and $\WDP_G^K$.

\subsection{Birationality properties} To begin, note that as the morphism $\mc{N}^\sqcup\to\mc{N}$ is surjective and satisfies the conditions of Lemma \ref{lem:stratification-isom}, $\WDP_G^{K,\sqcup}\to \WDP_G^K$ is then also surjective and satisfies the same conditions. We therefore deduce from Lemma \ref{lem:stratification-isom} the following.
\begin{prop}\label{prop:sqcup-to-square-bir} The morphism $\WDP_G^{K,\sqcup}\to \WDP_G^K$ is weakly birational.
\end{prop}

We now give a more explicit effective version of this result. To start, we observe the following where we denote by $(\varphi^\univ,N^\univ)$ the universal pair over $\WDP_G^K$.

\begin{prop}\label{prop:red-equi-dim-loc-closed} For each $n\geqslant 0$, the subset
\begin{equation*}
   \WDP_G^{K,n}\defeq \left\{x\in \WDP_G^K:Z_{\wh{G}}(\varphi^\univ,N^\univ)_x^\circ \emph{ is reductive of dimension }n+\dim(Z_0(\wh{G}))\right\}
\end{equation*}
of $\WDP_G^K$ is locally closed, is open if $n=0$, and is empty if $n>\dim(\wh{G}/Z_0(\wh{G}))$. 
\end{prop}
\begin{proof} Consider the quotient $Q\defeq Z_{\wh{G}}(\varphi^\univ,N^\univ)/Z_0(\wh{G})_{\WDP_G^K}$. By \cite[Expos\'{e} VIB, Proposition 4.1]{SGA3-1}, the function $f\colon \WDP_G^K\to \mathbb{N}$ given by $f(x)=\dim(Q_x)$ is upper semi-continuous. In particular the set $D_n=f^{-1}([0,n+1))\cap f^{-1}([n,\infty)]$ of points where $Q_x$ is of dimension $n$ is locally closed, and as $D_0=f^{-1}([0,1))$, $D_0$ is open. Let us endow $D_n$ with the reduced substructure. Let us then note that by \cite[Expos\'{e} VIB, Corollaire 4.4]{SGA3-1} for all $n\geqslant 0$ the identity component functor $Q_{D_n}^\circ$ is representable and is smooth over $D_n$. Thus, by \cite[Proposition 3.1.9]{ConRgrsch}, we deduce that the locus of $x$ in $D_n$ where $Q_x^\circ$ is reductive is open, and thus locally closed in $\WDP_G^K$ and open if $n=0$. But, evidently this locus is equal to $\WDP_G^{K,n}$. Finally, as $Z_{\widehat{G}}(\varphi^\mathrm{univ},N^\mathrm{univ})^\circ_x$ is a subgroup of $\widehat{G}_x$ we have that its dimension is at most $\mathrm{dim}(\widehat{G})$, and thus evidently $\mathsf{WDP}^{K,n}_G$ is empty for $n>\dim(\widehat{G}/Z_0(\widehat{G}))$.
\end{proof}

\begin{defn}\label{defn:locus-of-red} We define the \emph{reductive centralizer locus} in $\WDP_G^K$ to be the $\Q$-scheme  $\WDP_G^{K,\mathrm{rc}}\defeq \bigsqcup_n \WDP_G^{K,n}$ (where each $\WDP_G^{K,n}$ is given the reduced subscheme structure). We call the open subset $\WDP_G^{K,0}$ the \emph{discrete locus} and denote it by $\WDP_G^{K,\disc}$.
\end{defn}

Let us observe that by the proof of Proposition \ref{prop:red-equi-dim-loc-closed}, if $A$ is a reduced $\Q$-algebra and $(\varphi,N)$ is a Weil--Deligne parameter over $A$ such that the corresponding morphism $\Spec(A)\to \WDP_G^K$ factorizes through $\WDP_G^{K,\mathrm{rc}}$, then $Z_{\wh{G}}(\varphi,N)^\circ$ is representable and reductive over $A$.

Now we show that the reducedness of $\WDP_G^{K,\mathrm{rc}}$ implies that $N^\univ$ pulled back to this reductive centralizer locus lies in $\mc{N}^\sqcup$. More precisely, we have the following.

\begin{prop}\label{prop:red-cent-constant-N} The morphism $\WDP_G^{K,\mathrm{rc}}\to\WDP_G^K$ factorizes through $\WDP_G^{K,\sqcup,\mathrm{ss}}$. 
\end{prop}

Indeed, as $\WDP_G^{K,\mathrm{rc}}$ is reduced by definition, this follows from Proposition \ref{prop:red-cent-ss} and the following proposition.

\begin{prop}\label{prop:reductive-cent-constant-N} If $A$ is a reduced $\Q$-algebra, and $(\varphi,N)$ is an element of $\WDP_G(A)$ such that $Z_{\wh{G}}(\varphi,N)_x^\circ$ is a reductive group scheme of dimension $n$ for all $x$ in $\Spec(A)$, then $(\varphi,N)$ is an element of $\WDP^\sqcup_G(A)$.
\end{prop}

\begin{proof} We break the argument into several steps to make the structure clear.

\medskip

\noindent\textbf{Step 1:} It suffices to prove that if $A$ is a strictly Henselian discrete valuation ring, then $N$ is egc to some $N_0$ in $\wh{\mc{N}}(\Q)$. Indeed, we must show that the map $\Spec(A)\to \mathcal{N}$ induced by $(\varphi,N)$ factorizes through $\mc{N}^\sqcup$. By standard Noetherian approximation arguments we may assume that $A$ is Noetherian. We may then assume that $A$ is connected, in which case we must show that this morphism factorizes through some $\mc{O}_N$. As $A$ is reduced, it suffices to show that $\Spec(A)\to \mc{N}$ factorizes through some $\mc{O}_N$ set-theoretically. As $A$ is connected, any two points of $\Spec(A)$ may be connected by a finite chain of specialization and generalizations. This reduces us to showing that if $x$ is a generalization of $y$ in $\Spec(A)$ then these points map into a common $\mc{O}_N$. We are then reduced to the case of a discrete valuation ring by \stacks{054F}, and then trivially to the case of a strictly Henselian discrete valuation ring.

\medskip

\noindent\textbf{Step 2:} We claim we may assume that $(\varphi,N)$ is in $\WDP^{K,\disc}_G(A)$. Write $\eta$ (resp.\@ $s$) for the generic point (resp.\@ special) of $\Spec(A)$. As $Z_{\widehat{G}}(\varphi,N)$ has constant fiber dimension, the same is true for $Z_{\widehat{G}^{\mathrm{der}}}(\varphi,N)$. Indeed, for each point $x$ of $\Spec(A)$ the group $\wh{G}_x$ is the quotient $(\wh{G}^\mathrm{der}_x\times Z(\wh{G})_x)/Z(\wh{G}^\mathrm{der})_x$ where $Z(\wh{G}^\mathrm{der})_x$ is a finite group scheme (see \cite[Example 19.25]{MilneGroups}). As a pair $(g,z)$ in $\wh{G}^\mathrm{der}_x\times Z(\wh{G})_x$ centralizes $(\varphi,N)$ if and only if $g$ does, we deduce that $Z_{\wh{G}}(\varphi,N)_x$ is the quotient of $Z_{\wh{G}^\der}(\varphi,N)_x\times Z(\wh{G})_x$ by the finite group scheme $Z(\wh{G}^\der)_x\cap Z_{\wh{G}^\der}(\varphi,N)_x$. Thus, $\dim(Z_{\wh{G}}(\varphi,N)_x)$ is $\dim(Z_{\wh{G}^\der}(\varphi,N)_x)+\dim(Z(\wh{G})_x)$. As $\dim(Z(\wh{G})_x)$ is also constant, the claim follows. 

So, again \cite[Expos\'{e} VIB, Corollaire 4.4]{SGA3-1} shows that $Z_{\widehat{G}^{\mathrm{der}}}(\varphi,N)^\circ$ is representable and reductive over $A$. As $A$ is strictly Henselian, for any reductive group over $A$, all its tori are split, all its maximal tori are conjugate, and all its Borel subgroups are conjugate. Then, as $\CG$ is equal to ${^L}\!\wt{G}$ the arguments in the second paragraph of the proof of \cite[Lemma 3.5]{BorelCorvallis} show that if $T$ is a maximal torus of $Z_{\wh{G}^\der}(\varphi,N)^\circ$ then $Z_{\CG_A}(T)$ is a Levi subgroup of $\CG_A$ since $Z_{\CG_A}(T)$ contains the image of $\varphi$ and hence projects onto $\mathcal{W}_{F,A}$. Further, there exists some $g\in \check{G}(A)$ and a Levi subgroup $H$ of $G^\ast$ (where $G^\ast$ is the quasi-split inner form of $G$) such that $g Z_{\CG_A}(T)g^{-1}={^C}\!H_A$. Therefore $g(\varphi,N)g^{-1}$ factorizes through ${^C}\!H_A$ because $T \subset Z_{\wh{G}^\der}(\varphi,N)^\circ$. We claim then that $g(\varphi,N)g^{-1}$ is in $\WDP^{K,\disc}_H(A)$. By Proposition \ref{prop:red-cent-ss} $g(\varphi_\eta,N_\eta)g^{-1}$ and $g(\varphi_s,N_s)g^{-1}$ are Frobenius semi-simple. Moreover, the argument given in \cite[Proposition 3.6]{BorelCorvallis} shows that neither $g(\varphi_\eta,N_\eta)g^{-1}$ nor $g(\varphi_s,N_s)g^{-1}$ factorizes through a proper Levi (in the sense of loc.\@ cit.\@) which, as they are both Frobenius semi-simple, implies by the usual arguments (cf.\@ \cite[Lemma 10.3.1]{KotStfcus}) that they are discrete. As $N$ is in $\mc{N}^\sqcup(A)$ if and only if $gNg^{-1}$ is, the claimed reduction follows.

\medskip

\noindent\textbf{Step 3:} We now show that we may assume $N_s\ne 0$. If both $N_s$ and $N_\eta$ are zero we're done, and so it suffices to show that if $N_\eta\ne 0$ then $N_s\ne 0$. To see this, assume otherwise. But the inequality $\dim Z_{\check{G}}(\varphi_\eta)\leqslant \dim Z_{\check{G}}(\varphi_s)=\dim Z_{\check{G}}(\varphi_s,N_s)$ holds by \cite[Expos\'{e} VIB, Proposition 4.1]{SGA3-1}. Also, $\dim Z_{\check{G}}(\varphi_\eta,N_\eta)<\dim Z_{\check{G}}(\varphi_\eta)$. Indeed, it suffices to note that if $w_0$ is any lift of arithmetic Frobenius then (as in Proposition \ref{prop:Frob-factor}) for $m$ sufficiently large $\check{\varphi}_\eta(w_0^m)$ defines a point of $Z_{\check{G}}(\varphi_\eta)^\circ$ but, as $N_\eta \ne0$, does not define a point of $Z_{\check{G}}(\varphi_\eta,N_\eta)$ and thus $ Z_{\check{G}}(\varphi_\eta,N_\eta)^\circ \subsetneq Z_{\check{G}}(\varphi_\eta)^\circ$ from where the claim follows. But, observe that $\dim(Z_{\check{G}}(\varphi_\eta,N_\eta))$ (resp.\@ $\dim(Z_{\check{G}}(\varphi_s,N_s))$) is equal to  $\dim(Z_{\widehat{G}}(\varphi_\eta,N_\eta))+1$ (resp.\@ $\dim(Z_{\widehat{G}}(\varphi_s,N_s))+1$) and so we arrive at a contradiction with the assumption that the dimensions of the fibers of $Z_{\wh{G}}(\varphi,N)$ are constant.

\medskip

\noindent\textbf{Step 4:} Taking the image of $(\varphi,N)$ under $\CG \to {{^C}\!(G^\der)}$, we may replace $G$ with $G^\der$. Then $Z_0(\wh{G})$ is finite since $G$ is semisimple. Proposition \ref{prop:red-cent-ss} together with Theorem \ref{thm:rel-JM-param} imply that $(\varphi_\eta,N_\eta)$ (resp.\@ $(\varphi_s,N_s)$) comes from an $L$-parameter $\psi_1$ (resp.\@ $\psi_2$). Write $\mu_i$ for the restriction of $\theta_{\psi_i}$ to the diagonal maximal torus, which we view as a cocharacter via the map $z\mapsto \left(\begin{smallmatrix}z & 0\\ 0 & z^{-1}\end{smallmatrix}\right)$. Fix $w_0$ to be an arithmetic Frobenius lift. By Frobenius semi-simplicity and the fact that $A$ is strictly Henselian, there is, up to conjugacy, a positive integer $m_0$ divisible by $[F^\ast:F]$ such that $\check{\varphi} (w_0^{m_0})$ is contained in the $A$-points of a maximal torus $T$ of $\check{G}_{\ov{\Q}}$. By the relation between $\psi_1$ and $\varphi_\eta$, as well as the relationship between $\psi_2$ and $\varphi_s$, and the argument of \cite[Lemma 3.1]{GRAinv}, we see that up to replacing $m_0$ by a power, we may further assume that $\check{\varphi}_\eta (w_0^{2m_0})=\mu_1 (q^{m_0})$ and  $\check{\varphi}_s(w_0^{2m_0})=\mu_2 (q^{m_0})$. From this first equality it is simple to see that $\mu_1$ factorizes through $T_\eta$, and thus there exists a unique lift $\mu_A$ of $\mu_1$ to $T_A$ where $\mu$ is a cocharacter of $T$. We note as $N_s\ne 0$, that $\mu_2$ is characterized by the property that the image of $\mu_2$ contains $\check{\varphi}_s (w_0^{2m_0})$ and $\mathrm{Ad}(\mu_2 (q^{m_0}))(N_s)=q^{2m_0}N_s$. As $\check{G}_A$ and $\wh{\mf{g}}_A$ are separated over $A$, we have that the image of $\mu$ contains $\check{\varphi} (w_0^{2m_0})$ and $\mathrm{Ad}(\mu(q^{m_0}))(N)=q^{2m_0}N$. Hence, $\mu_s$ satisfies the above characterization of $\mu_2$, so $\mu_s=\mu_2$. 
Let $P(\mu)$ be the parabolic subgroup of $\widehat{G}_{\ov{\Q}}$ associated to $\mu$. Define $\widehat{\mathfrak{g}}_\eta(j)$ (resp.\@ $\widehat{\mathfrak{g}}_s(i)$) using $\mu_\eta$ (resp.\@ $\mu_s$) as in \cite[\S5.7]{CarFinLie}. Then by \cite[Proposition 5.7.3]{CarFinLie} $N_\eta$ (resp.\@ $N_s$) is in the unique open $P(\mu)_\eta$-orbit (resp.\@ $P(\mu)_s$) of $\bigoplus_{i \geq 2} \widehat{\mathfrak{g}}_\eta(i)$ (resp.\@ $\bigoplus_{i \geq 2} \widehat{\mathfrak{g}}_s(i)$). But, by the uniqueness of this open orbit, we then see that $N_\eta$ and $N_s$ are both conjugate to any $\ov{\Q}$-point of the unique open orbit of $P(\mu)$ on $\bigoplus_{i \geq 2} \widehat{\mathfrak{g}}(i)$, from where the conclusion follows. We are then done by Proposition \ref{prop:split-nilp-desc}.
\end{proof}

We next show the pleasant property that $\WDP_G^{K,\mathrm{rc}}$ actually has dense image in $\WDP_G^{K,\sqcup}$.

\begin{lem}\label{lem:stab-dim}
Let $k$ be a field, $X$ an irreducible finite type $k$-scheme equipped with an action of an algebraic $k$-group $H$, and $Y$ an irreducible locally closed subscheme of $X$. Assume that the action morphism $\mu \colon H \times Y \to X$ is dominant. Then there is a dense open subset $U$ of $Y$ such that $\dim Z_H(y)\leqslant \dim(H) +\dim(Y) -\dim(X)$ for all $y \in U$. 
\end{lem}
\begin{proof}
By \cite[Corollary 14.116]{GortzWedhorn} there exists a dense open subset $V$ of $X$ with the property that $\dim \mu^{-1}(y)=\dim H +\dim Y -\dim X$ for all $y \in V$. As $\mu$ is $H$-equivariant when $H$ is made to act on the first component of $H\times Y$, we may assume that $V$ is $H$-stable by replacing $V$ with $HV$. We put $U=V \cap Y$, which is non-empty as $\mu$ is dominant and $V$ is $H$-stable. As $Z_H(y) \times \{ y \}  \subseteq \mu^{-1}(y)$ for $y \in U$, we obtain the claim. 
\end{proof}

\begin{prop}\label{prop:dense-tor-cent} 
The set 
\[
 \{ x \in \WDP_G^{K,\sqcup} : Z_{\wh{G}}(\varphi^\univ,N^\univ)^\circ_x \textrm{ is a torus}\}
\] 
contains an open dense subset of $\WDP_G^{K,\sqcup}$.
\end{prop}
\begin{proof} Observe that this may be checked over $\ov{\Q}$, as the morphism $\Spec(\ov{\Q})\to\Spec(\Q)$ is surjective and universally open (see \stacks{0383}). Thus, from Theorem \ref{thm:WD-const-decomp} it suffices to show that for each $(\gamma,\phi,N)$ in $\WDP_G^K(\ov{\Q})$ corresponding to $(\varphi,N)$, one has that the set of points $x$ in $U(\gamma,\phi,N)$ such that $Z_{\wh{G}}(\varphi^\univ,N^\univ)^\circ_x$ is a torus contains a dense open subset.

Let $H$ be the normalizer of $\phi$ in $(\wh{G} \rtimes \underline{\Gamma_\ast})_{\ol{\bQ}}$. Then $H^\circ=Z_{\wh{G}}(\phi)^\circ$ which is a reductive group by Lemma \ref{lem:fixed-points-reductive} as $Z_{\widehat{G}}(\phi)$ is the same as $\widehat{G}^\Sigma$ where $\Sigma$ is $\phi(I_F/I_K)\subseteq \wh{G}(\ov{\Q})\rtimes (I_F/I_K)$ (which is finite as $I_F/I_K$ is) acting on $\widehat{G}$ by conjugation. Consider the linear algebraic $\ov{\Q}$-group $S'_H(N)$ representing the functor
\begin{equation*}
    \cat{Alg}_{\ol{\bQ}}\to \cat{Grp},\qquad A\mapsto \left\{(h,z)\in H(A)\times A^\times: \Ad(h)(N)=z^2N \right\},
\end{equation*}
which is clearly seen to be a closed subgroup scheme of $((\wh{G}\times \bb{G}_{m}) \rtimes \underline{\Gamma_\ast})_{\ov{\Q}}$ by changing the order of the components. 
Let $S_H(N)$ be the image of $S'_H(N)$ in $(\check{G} \rtimes \underline{\Gamma_\ast})_{\ov{\Q}}$. Let $s_0 u_0$ be the Jordan decomposition of $\ov{\varphi}(w_0)$ in $S_H(N)$. 
Then the image of $u_0$ in $\mathbb{G}_{m,\ol{\bQ}}$ is trivial. Hence $u_0$ is an element of $Z_{\phi,N}^{\circ}$. Replacing $\gamma$ by $u_0^{-1} \gamma$, we may assume that $\varphi$ is Frobenius semi-simple from the beginning.

Let $\psi$ be an element of $\LP_G^K(\ov{\Q})$ such that $\JM(\psi)=(\varphi,N)$ and write $\theta=\theta_\psi$. Let $U_{H}(N)$ be the unipotent radical of $Z_{H}(N)$. Then, as in Proposition \ref{prop:Zudec}, we have $Z_{H}(N)=U_{H}(N)\rtimes Z_{H}(\theta)$. We take a maximal quasi-torus $T$ of $Z_{H}(\theta)$ in the sense of \cite[Definition 8.6]{HaPaCryChe}. Set $s_1$ to be the image of $\left( \theta \left( \left(\begin{smallmatrix} q^{1/2} & 0 \\ 0 & q^{-1/2} \end{smallmatrix}\right)\right),q^{1/2}\right)$ in $\check{G}(\ov{\Q})$. Then $Z_{\phi,N}^{\circ} \gamma s_1^{-1} \subseteq Z_H(N)$. Note that $Z_{\phi,N}^{\circ} \gamma s_1^{-1}$ is a connected component of $Z_H(N)$ since $Z_{\phi,N}^{\circ}$ is the identity component of $Z_H(N)$. So we can write $T \cap Z_{\phi,N}^{\circ} \gamma s_1^{-1}=t_1 T^{\circ}$ for some $t_1 \in T(\ov{\Q})$ by \cite[Theorem 8.10 (d)]{HaPaCryChe}. Then, we have $Z_{\phi,N}^{\circ} \gamma=t_1 Z_{\phi,N}^{\circ} s_1$ as $Z_{\phi,N}^{\circ} \gamma s_1^{-1}$ is a connected component of $Z_H(N)$ containing $t_1$. 

We let $T^{t_1}$ be the closed subgroup scheme of $T$ of elements commuting with $t_1$. 
For $t_0$ in $(T^{t_1})^\circ(\ov{\Q})$, we consider the morphism 
\[
 \Lambda_{t_0} \colon Z_{H}(N)^{\circ} \times (T^{t_1})^\circ \to Z_{H}(N)^{\circ},\qquad (h,t) \mapsto (t_1t_0)^{-1}ht_1t_0 t s_1 h^{-1} s_1^{-1}. 
\]
This induces 
\[
 \Lie (\Lambda_{t_0}) \colon \mathrm{Lie}(Z_{H}(N)^{\circ}) \times \mathrm{Lie}((T^{t_1})^\circ) \to \mathrm{Lie}(Z_{H}(N)^{\circ}),\qquad (x,y) \mapsto \mathrm{ad}((t_1t_0)^{-1})x + y - \mathrm{ad}(s_1)x. 
\] 
This is identified with the direct sum of 
\begin{align*}
  &\Lie (\Lambda_{t_0})_1 \colon \mathrm{Lie}(Z_{H}(\theta)^{\circ}) \times \mathrm{Lie}((T^{t_1})^\circ) \to \mathrm{Lie}(Z_{H}(\theta)^{\circ}),\qquad (x,y) \mapsto \mathrm{ad}((t_1t_0)^{-1})x + y - x,\\
  &\Lie (\Lambda_{t_0})_2 \colon \mathrm{Lie}(U_{H}(N)^{\circ}) \to \mathrm{Lie}(U_{H}(N)^{\circ}),\qquad z \mapsto \mathrm{ad}((t_1t_0)^{-1})z  - \mathrm{ad}(s_1)z. 
\end{align*}
In the proof of \cite[Theorem 8.9 (c)]{HaPaCryChe}, it is shown that the morphism 
\begin{equation*}
    Z_{H}(\theta)^{\circ} \times t_1 (T^{t_1})^\circ \to t_1 Z_{H}(\theta)^{\circ},\qquad (g,t) \mapsto g t g^{-1}
\end{equation*} 
is dominant. Therefore, by Lemma \ref{lem:stab-dim} and the fact that $(T^{t_1})^\circ \subseteq Z_{Z_{H}(\theta)^{\circ}}(t_1t_0)^{\circ}$ for any $t_0 \in (T^{t_1})^\circ$, there is an open dense subset $U_{t_1,1} \subseteq (T^{t_1})^\circ$ such that $Z_{Z_{H}(\theta)^{\circ}}(t_1t_0)^{\circ}=(T^{t_1})^\circ$ for $t_0 \in U_{t_1,1}$. This implies that $\Lie (\Lambda_{t_0})_1$ is surjective for $t_0 \in U_{t_1,1}$. 

The eigenvalues of the diagonalizable $\mathrm{ad}(s_1)$ on $\mathrm{Lie}(U_{H}(N))$ are contained in $\{ q^{i/2} \}_{1 \leq i \leq n_0}$ for some $n_0$ by Proposition \ref{prop:Zudec}. Let $m_1$ be the order of $t_1$ in $\pi_0(T)$. 
Then there is a positive integer $m$ such that the eigenvalues of the diagonalizable  $\mathrm{ad}(t_1^{-1-mm_1})$ on $\mathrm{Lie}(U_{H}(N))$ are disjoint from $\{ q^{i/2} \}_{1 \leq i \leq n_0}$. Since $t_1^{-1-mm_1}$ and $s_1$ are commutative, $\mathrm{ad}(t_1^{-1-mm_1})$ and $\mathrm{ad}(s_1)$ are simultaneously diagonalizable. Hence we have the surjectivity of $\Lie (\Lambda_{t_1^{mm_1}})_2$. 
Since the surjectivity of $\Lie (\Lambda_{t_0})_2$ defines an open subset on $(T^{t_1})^\circ$, which we now know is  non-empty, there is an open dense subset $U_{t_1,2} \subseteq (T^{t_1})^\circ$ such that 
$\Lie (\Lambda_{t_0})_2$ is surjective for $t_0 \in U_{t_1,2}$. 

We put $U_{t_1}=U_{t_1,1} \cap U_{t_1,2}$. Then, 
for $t_0 \in U_{t_1}$, 
the map $\Lie (\Lambda_{t_0})$ is surjective, hence $\Lambda_{t_0}$ is dominant. 
This implies that 
\begin{equation*}
    Z_{H}(N)^{\circ} \times t_1 (T^{t_1})^\circ s_1 \to t_1 Z_{H}(N)^{\circ} s_1,\qquad  (g,t) \mapsto gtg^{-1}
\end{equation*}is dominant. 
Further, for $t_0 \in U_{t_1}$, 
the surjectivity of $\Lie (\Lambda_{t_0})$ 
implies that the kernel of
\begin{equation*}
    \mathrm{Lie}(Z_{H}(N)^{\circ}) \to \mathrm{Lie}(Z_{H}(N)^{\circ}),\qquad  x \mapsto \mathrm{ad}((t_1t_0)^{-1})x  - \mathrm{ad}(s_1)x
\end{equation*} 
is equal to $\mathrm{Lie}((T^{t_1})^\circ)$. This means that for $t_0 \in U_{t_1}$, we have $Z_{Z_{H}(N)} (t_1 t_0 s_1)^{\circ} =(T^{t_1})^\circ$. So we have toral centralizer for all points in the image of the dominant map
\begin{equation*}
    Z_{H}(N)^{\circ} \times t_1 U_{t_1} s_1 \to t_1 Z_{H}(N)^{\circ} s_1,\qquad (g,t) \mapsto gtg^{-1}, 
\end{equation*}
whose target is equal to $Z_{\phi,N}^{\circ} \gamma$, 
and so the conclusion follows from Chevalley's theorem (see \cite[Theorem 10.19]{GortzWedhorn}).
\end{proof}

From this, together with Proposition \ref{prop:sqcup-to-square-bir} and Lemma \ref{lem:stratification-isom}, we deduce that the two maps $\WDP_G^{K,\mathrm{rc}}\to \WDP_G^{K,\sqcup}$ and $\WDP_G^{K,\mathrm{rc}}\to \WDP_G^{K}$ are weakly birational. To connect this discussion to the Jacobson--Morozov map, we now show that $\JM$ is an isomorphism over $\WDP_G^{K,\mathrm{rc}}$.

\begin{prop}\label{prop:JM-isom-over-red-locus} The morphism $\JM\colon \JM^{-1}(\WDP_G^{K,\mathrm{rc}})\to \WDP_G^{K,\mathrm{rc}}$ is an isomorphism.
\end{prop}
\begin{proof} Let $A$ be a $\Q$-algebra. As $\JM$ is $\wh{G}(A)$-equivariant, to show that this map is a bijection on $A$-points it suffices to prove that the map on $A$-points is a bijection upon quotienting both sides by $\wh{G}(A)$, and that for all $\psi$ in $\JM^{-1}(\WDP_G^{K,\mathrm{rc}}(A))$ the equality $Z_{\wh{G}}(\psi)=Z_{\wh{G}}(\varphi,N)$ holds where $(\varphi,N)=\JM(\psi)$. For the bijectivity on quotient sets, it suffices by Theorem \ref{thm:rel-JM-param} to show that every element of $\WDP_G^{K,\mathrm{rc}}(A)$ belongs to $\WDP_G^{K,\sqcup,\ss}(A)$. But, this follows from Proposition \ref{prop:red-cent-constant-N}. Suppose now that $\psi$ is an element of $\JM^{-1}(\WDP_G^{K,\mathrm{rc}}(A))$. To show that $Z_{\wh{G}}(\psi)=Z_{\wh{G}}(\varphi,N)$ it suffices by Proposition \ref{prop:Zphidesc-rel} to show that $U^N(\varphi)$ is trivial. Applying the fiberwise criterion for isomorphism (see \cite[Lemma B.3.1]{ConRgrsch}) to identity section of $U^N(\varphi)$ it suffices to show that $U^N(\varphi)_x$ is trivial for all $x$ in $\Spec(A)$. But, as $U^N(\varphi)_x$ is unipotent it is contained in $Z(\varphi,N)^\circ_x$, and as it is also normal it must be trivial by our assumption that $Z(\varphi,N)^\circ_x$ is reductive.
\end{proof}

We deduce that $\WDP_G^{K,\mathrm{rc}}$ also admits a weakly birational monomorphism to $\LP_G^K$. So, we now come to our main geometric result concerning the Jacobson--Morozov morphism.

\begin{thm}\label{thm:JM-omnibus} The morphism $\JM\colon \LP_G^K\to \WDP_G^{K,\sqcup}$ (resp. $\JM\colon \LP_G^K\to \WDP_G^K$) is birational (resp.\@ weakly birational).
\end{thm}
\begin{proof} The weak birationality of both maps is clear from the above discussion, and therefore it suffices to show that the map $\JM\colon \LP_G^K\to \WDP_G^{K,\sqcup}$ induces a bijection on irreducible components. It clearly suffices to check this after base changing to $\ov{\Q}$. By Theorem \ref{thm:WD-const-decomp} and Theorem \ref{thm:L-const-decomp} the connected components of $\LP_{G,\ov{\Q}}^K$ and $\WDP_{G,\ov{\Q}}^{K,\sqcup}$ are irreducible, so it suffices to show that the map $\JM\colon \pi_0(\LP_{G,\ov{\Q}}^K)\to \pi_0(\WDP_{G,\ov{\Q}}^{K,\sqcup})$ is bijective. 

To do this we first show that the Jacobson--Morozov map induces a bijection $[\mathsf{LP}_G^K(\overline{\mathbb{Q}})]\to [\mathsf{WDP}_G^K(\overline{\mathbb{Q}})]$. By Proposition \ref{prop:dense-tor-cent} and Proposition \ref{prop:red-cent-ss} every equivalence class of the target contains a Frobenius semi-simple element and thus surjectivity follows from Theorem \ref{thm:rel-JM-param}. To show injectivity suppose that $(\gamma_i,\phi_i,\theta_i)$ for $i=1,2$ are elements of $\mathsf{LP}_G^K(\overline{\mathbb{Q}})$ such that $(\gamma_i,\phi_i,N_i)$ are equivalent in $\WDP_G^K(\ov{\Q})$. Without loss of generality, we may assume that $\phi_1=\phi_2=:\phi$ and $N_1=N_2=:N$ and that $\gamma_2=h\gamma_1$ with $h$ in $Z_{\phi,N}(\overline{\mathbb{Q}})$. By  Proposition \ref{prop:gr-prop} there exists $z$ in $Z_{\phi,N}(\ov{\Q})$ such that $z\theta_1 z^{-1}=\theta_2$. Note then that $(\gamma_2,\phi,\theta_2)=z(s\gamma_1,\phi,\theta_1)z^{-1}$ where $s=z^{-1}\gamma_2z\gamma_1^{-1}$. Writing $s=z^{-1}h\gamma_1 z\gamma_1^{-1}$ one sees from the fact that $z^{-1}$ and $h$ both centralize $\phi$ and $\gamma_1$ normalizes $\phi$ that $s$ centralizes $\phi$. On the other hand, one can just as easily check that as $\gamma_1$ centralizes $\theta_1$ and $\gamma_2$ centralizes $\theta_2$ that $s=z^{-1}\gamma_2z\gamma_1^{-1}$ also centralizes $\theta_1$. Therefore as $(\gamma_2,\phi,\theta_2)=z(s\gamma_1,\phi,\theta_1)z^{-1}$ we deduce that $(\gamma_2,\phi,\theta_2)$ and $(\gamma_1,\phi,\theta_1)$ are equivalent in $\mathsf{LP}_G^K(\overline{\mathbb{Q}})$ as desired. 

But, for $(\gamma,\phi,\theta)$ with image $(\gamma',\phi,N)$ under the Jacobson--Morozov map, one has $\pi_0(Z_{\phi,N})=\pi_0(Z_{\theta,N})$ as follows quickly from Proposition \ref{prop:Zphidesc-rel}. These observations together with Corollary \ref{cor:WDP-pi0} and Corollary \ref{cor:LP-pi0} give the desired conclusion.
\end{proof}

Let us finally note that as a possibly useful corollary of the above results, we also obtain the density of Frobenius semi-simple parameters in all three of these moduli spaces.

\begin{cor} The subsets 
\begin{equation*}
    \LP_G^\ss(\ov{\Q})\subseteq \LP_G,\qquad \WDP_G^{\sqcup,\ss}(\ov{\Q}) \subseteq \WDP_G^{\sqcup},\qquad \WDP_G^\ss(\ov{\Q})\subseteq \WDP_G
\end{equation*}
are dense.
\end{cor}

\subsection{Isomorphism over the discrete locus}

In this final section we apply the material to give a geometric analogue of Corollary \ref{cor:bij-et-disc} or, in other words, we show that the Jacobson--Morozov morphism is an isomorphism over the discrete loci in $\LP_G^K$ and $\WDP_G^K$.

We have defined the discrete locus $\WDP_G^{K,\disc}$ in Definition \ref{defn:locus-of-red}, and we now do so for $\LP_G^K$.

\begin{defn}\label{defn:disc-locus-L} Let $\psi^\univ$ be the universal $L$-parameter over $\LP_G^K$. Then, the \emph{discrete locus} in $\LP_G^K$ is the subset
\begin{equation*}
   \LP^{K,\disc}_G\defeq \left\{x\in \LP_G^K: Z_{\widehat{G}}(\psi^\univ)_x/Z_0(\wh{G})_x\to \Spec(k(x))\text{ is finite}\right\}.
\end{equation*}
\end{defn}

The same argument as in the proof of Proposition \ref{prop:red-equi-dim-loc-closed} shows that $\LP_G^{K,\disc}$ is an open subset of $\LP_G^K$ and we endow it with the open subscheme structure. The following relates the discrete loci in $\WDP_G^K$ and $\LP_G^K$, giving a geometrization of Corollary \ref{cor:bij-et-disc}.

\begin{prop} The equality $\JM^{-1}(\WDP_G^{K,\disc})=\LP_G^{K,\disc}$ holds.
\end{prop}
\begin{proof} As these are both open subsets of the finite type affine $\Q$-scheme $\LP_G^K$, it suffices to show that they have the same $\ov{\Q}$-points. In other words, we must show that for an element $\LP_G^K(\ov{\Q})$ one has that $Z_{\wh{G}}(\psi)$ is finite (as a set) if and only if $Z_{\wh{G}}(\JM(\psi))$ is finite. Choosing an embedding $\ov{\Q}\to \bb{C}$ one then quickly deduces this from Proposition \ref{prop:temp-cent-equal} and its proof.
\end{proof}

From this, and Proposition \ref{prop:JM-isom-over-red-locus} we deduce the following.

\begin{thm}\label{thm:JM-isom-disc-locus} The morphism $\JM\colon \LP_G^{K,\disc}\to\WDP_G^{K,\disc}$ is an isomorphism.
\end{thm}


\begin{thebibliography}{DHKM20}
	\providecommand{\url}[1]{\texttt{#1}}
	\providecommand{\urlprefix}{URL }
	\providecommand{\eprint}[2][]{\url{#2}}
	
	\bibitem[BG14]{BuGeconjc}
	K.~Buzzard and T.~Gee, The conjectural connections between automorphic
	representations and {G}alois representations, in Automorphic forms and
	{G}alois representations. {V}ol. 1, vol. 414 of London Math. Soc. Lecture
	Note Ser., pp. 135--187, Cambridge Univ. Press, Cambridge, 2014.
	
	\bibitem[BG19]{BeGeGdef}
	R.~Bellovin and T.~Gee, {$G$}-valued local deformation rings and global lifts,
	Algebra Number Theory 13 (2019), no.~2, 333--378.
	
	\bibitem[BMY23]{Characterization}
	A.~Bertoloni~Meli and A.~Youcis, An approach to the characterization of the
	local {L}anglands correspondence, Represent. Theory 27 (2023), 415--430.
	
	\bibitem[Bor79]{BorelCorvallis}
	A.~Borel, Automorphic {$L$}-functions, in Automorphic forms, representations
	and {$L$}-functions ({P}roc. {S}ympos. {P}ure {M}ath., {O}regon {S}tate
	{U}niv., {C}orvallis, {O}re., 1977), {P}art 2, Proc. Sympos. Pure Math.,
	XXXIII, Amer. Math. Soc., Providence, R.I., 1979 pp. 27--61.
	
	\bibitem[Bou72]{Bourbaki}
	N.~Bourbaki, \'{E}l\'{e}ments de math\'{e}matique. {F}asc. {XXXVII}. {G}roupes
	et alg\`ebres de {L}ie. {C}hapitre {II}: {A}lg\`ebres de {L}ie libres.
	{C}hapitre {III}: {G}roupes de {L}ie, Actualit\'{e}s Scientifiques et
	Industrielles [Current Scientific and Industrial Topics], No. 1349, Hermann,
	Paris, 1972.
	
	\bibitem[Bou75]{BourLie78}
	N.~Bourbaki, \'{E}l\'{e}ments de math\'{e}matique. {F}asc. {XXXVIII}: {G}roupes
	et alg\`ebres de {L}ie. {C}hapitre {VII}: {S}ous-alg\`ebres de {C}artan,
	\'{e}l\'{e}ments r\'{e}guliers. {C}hapitre {VIII}: {A}lg\`ebres de {L}ie
	semi-simples d\'{e}ploy\'{e}es, Hermann, Paris, 1975, actualit\'{e}s Sci.
	Indust., No. 1364.
	
	\bibitem[Bri22]{Brion}
	M.~Brion, Homomorphisms of algebraic groups: representability and rigidity,
	Michigan Math. J. 72 (2022), 51--76.
	
	\bibitem[BV85]{BaVoUnipss}
	D.~Barbasch and D.~A. Vogan, Jr., Unipotent representations of complex
	semisimple groups, Ann. of Math. (2) 121 (1985), no.~1, 41--110.
	
	\bibitem[Car85]{CarFinLie}
	R.~W. Carter, Finite groups of {L}ie type, Pure and Applied Mathematics (New
	York), John Wiley \& Sons, Inc., New York, 1985, conjugacy classes and
	complex characters, A Wiley-Interscience Publication.
	
	\bibitem[CGP15]{CGP}
	B.~Conrad, O.~Gabber and G.~Prasad, Pseudo-reductive groups, vol.~26 of New
	Mathematical Monographs, Cambridge University Press, Cambridge, second edn.,
	2015.
	
	\bibitem[Con14]{ConRgrsch}
	B.~Conrad, Reductive group schemes, in Autour des sch\'{e}mas en groupes.
	{V}ol. {I}, vol. 42/43 of Panor. Synth\`eses, pp. 93--444, Soc. Math. France,
	Paris, 2014.
	
	\bibitem[Dat05]{DatNu}
	J.-F. Dat, {$\nu$}-tempered representations of {$p$}-adic groups. {I}.
	{$\ell$}-adic case, Duke Math. J. 126 (2005), no.~3, 397--469.
	
	\bibitem[Del82]{DeligneHodge}
	P.~Deligne, Hodge Cycles on Abelian Varieties, pp. 9--100, Springer Berlin
	Heidelberg, Berlin, Heidelberg, 1982.
	
	\bibitem[DG70]{DemazureGabriel}
	M.~Demazure and P.~Gabriel, Groupes alg\'{e}briques. {T}ome {I}:
	{G}\'{e}om\'{e}trie alg\'{e}brique, g\'{e}n\'{e}ralit\'{e}s, groupes
	commutatifs, Masson \& Cie, \'{E}diteurs, Paris; North-Holland Publishing
	Co., Amsterdam, 1970, avec un appendice {{\i}t Corps de classes local} par
	Michiel Hazewinkel.
	
	\bibitem[DHKM20]{DHKMModLp}
	J.-F. Dat, D.~Helm, R.~Kurinczuk and G.~Moss, Moduli of Langlands Parameters,
	2020, arXiv:2009.06708.
	
	\bibitem[Elk72]{Elkington}
	G.~B. Elkington, Centralizers of unipotent elements in semisimple algebraic
	groups, J. Algebra 23 (1972), 137--163.
	
	\bibitem[FS21]{FaScGeomLLC}
	L.~Fargues and P.~Scholze, Geometrization of the local {L}anglands
	correspondence, 2021, arXiv:2102.13459.
	
	\bibitem[Gir71]{Giraud}
	J.~Giraud, Cohomologie non ab\'{e}lienne, Die Grundlehren der mathematischen
	Wissenschaften, Band 179, Springer-Verlag, Berlin-New York, 1971.
	
	\bibitem[SGA3-3]{SGA3-3new}
	P.~Gille and P.~Polo (editors), Sch\'{e}mas en groupes ({SGA} 3). {T}ome {III}.
	{S}tructure des sch\'{e}mas en groupes r\'{e}ductifs, vol.~8 of Documents
	Math\'{e}matiques (Paris) [Mathematical Documents (Paris)], Soci\'{e}t\'{e}
	Math\'{e}matique de France, Paris, 2011, s\'{e}minaire de G\'{e}om\'{e}trie
	Alg\'{e}brique du Bois Marie 1962--64. [Algebraic Geometry Seminar of Bois
	Marie 1962--64], A seminar directed by M. Demazure and A. Grothendieck with
	the collaboration of M. Artin, J.-E. Bertin, P. Gabriel, M. Raynaud and J-P.
	Serre, Revised and annotated edition of the 1970 French original.
	
	\bibitem[GP13]{GillePianzola}
	P.~Gille and A.~Pianzola, Torsors, reductive group schemes and extended affine
	{L}ie algebras, Mem. Amer. Math. Soc. 226 (2013), no. 1063, vi+112.
	
	\bibitem[GR10]{GRAinv}
	B.~H. Gross and M.~Reeder, Arithmetic invariants of discrete {L}anglands
	parameters, Duke Math. J. 154 (2010), no.~3, 431--508.
	
	\bibitem[EGA4-4]{EGA4-4}
	A.~Grothendieck, \'{E}l\'{e}ments de g\'{e}om\'{e}trie alg\'{e}brique. {IV}.
	\'{E}tude locale des sch\'{e}mas et des morphismes de sch\'{e}mas {IV}, Inst.
	Hautes \'{E}tudes Sci. Publ. Math.  (1967), no.~32, 361.
	
	\bibitem[Gro68]{GrothendieckBrauerIII}
	A.~Grothendieck, Le groupe de {B}rauer. {III}. {E}xemples et compl\'{e}ments,
	in Dix expos\'{e}s sur la cohomologie des sch\'{e}mas, vol.~3 of Adv. Stud.
	Pure Math., pp. 88--188, North-Holland, Amsterdam, 1968.
	
	\bibitem[GW20]{GortzWedhorn}
	U.~G\"{o}rtz and T.~Wedhorn, Algebraic geometry {I}. {S}chemes---with examples
	and exercises, Springer Studium Mathematik---Master, Springer Spektrum,
	Wiesbaden, [2020] \copyright 2020, second edition [of 2675155].
	
	\bibitem[Hel20]{Helm}
	D.~Helm, Curtis homomorphisms and the integral {B}ernstein center for {${\rm
			GL}_n$}, Algebra Number Theory 14 (2020), no.~10, 2607--2645.
	
	\bibitem[HP18]{HaPaCryChe}
	U.~Hartl and A.~Pal, Crystalline Chebotar\"{e}v density theorems, 2018,
	arXiv:1811.07084.
	
	\bibitem[Ima20]{ImaLLCell}
	N.~Imai, Local {L}anglands correspondences in $\ell$-adic coefficients, 2020,
	arXiv:2003.14154.
	
	\bibitem[Jac79]{Jacobson}
	N.~Jacobson, Lie algebras, Dover Publications, Inc., New York, 1979,
	republication of the 1962 original.
	
	\bibitem[Jan04]{Jantzen}
	J.~C. Jantzen, Nilpotent orbits in representation theory, in Lie theory, vol.
	228 of Progr. Math., pp. 1--211, Birkh\"{a}user Boston, Boston, MA, 2004.
	
	\bibitem[Kal16]{KalLLCnqs}
	T.~Kaletha, The local {L}anglands conjectures for non-quasi-split groups, in
	Families of automorphic forms and the trace formula, Simons Symp., pp.
	217--257, Springer, [Cham], 2016.
	
	\bibitem[Kos59]{KostdsB}
	B.~Kostant, The principal three-dimensional subgroup and the {B}etti numbers of
	a complex simple {L}ie group, Amer. J. Math. 81 (1959), 973--1032.
	
	\bibitem[Kos63]{KostantLieGroupReps}
	B.~Kostant, Lie group representations on polynomial rings, Amer. J. Math. 85
	(1963), 327--404.
	
	\bibitem[Kot84]{KotStfcus}
	R.~E. Kottwitz, Stable trace formula: cuspidal tempered terms, Duke Math. J. 51
	(1984), no.~3, 611--650.
	
	\bibitem[MFK94]{Mumford}
	D.~Mumford, J.~Fogarty and F.~Kirwan, Geometric invariant theory, vol.~34 of
	Ergebnisse der Mathematik und ihrer Grenzgebiete (2) [Results in Mathematics
	and Related Areas (2)], Springer-Verlag, Berlin, third edn., 1994.
	
	\bibitem[Mil17]{MilneGroups}
	J.~S. Milne, Algebraic groups, vol. 170 of Cambridge Studies in Advanced
	Mathematics, Cambridge University Press, Cambridge, 2017, the theory of group
	schemes of finite type over a field.
	
	\bibitem[PY02]{PrasadYu}
	G.~Prasad and J.-K. Yu, On finite group actions on reductive groups and
	buildings, Invent. Math. 147 (2002), no.~3, 545--560.
	
	\bibitem[SGA3-1]{SGA3-1}
	Sch\'{e}mas en groupes. {I}: {P}ropri\'{e}t\'{e}s g\'{e}n\'{e}rales des
	sch\'{e}mas en groupes, S\'{e}minaire de G\'{e}om\'{e}trie Alg\'{e}brique du
	Bois Marie 1962/64 (SGA 3). Dirig\'{e} par M. Demazure et A. Grothendieck.
	Lecture Notes in Mathematics, Vol. 151, Springer-Verlag, Berlin-New York,
	1970.
	
	\bibitem[Sol20]{Solleveld}
	M.~Solleveld, Conjugacy of {L}evi subgroups of reductive groups and a
	generalization to linear algebraic groups, J. Pure Appl. Algebra 224 (2020),
	no.~6, 106254, 16.
	
	\bibitem[Spr69]{SpringerUnipotent}
	T.~A. Springer, The unipotent variety of a semi-simple group, in Algebraic
	{G}eometry ({I}nternat. {C}olloq., {T}ata {I}nst. {F}und. {R}es., {B}ombay,
	1968), pp. 373--391, Oxford Univ. Press, London, 1969.
	
	\bibitem[{Sta}21]{StacksProject}
	T.~{Stacks Project Authors}, \textit{Stacks Project},
	\url{http://stacks.math.columbia.edu}, 2021.
	
	\bibitem[SZ18]{SilbergerZink}
	A.~J. Silberger and E.-W. Zink, Langlands classification for {$L$}-parameters,
	J. Algebra 511 (2018), 299--357.
	
	\bibitem[Tat79]{TatNtb}
	J.~Tate, Number theoretic background, in Automorphic forms, representations and
	{$L$}-functions ({P}roc. {S}ympos. {P}ure {M}ath., {O}regon {S}tate {U}niv.,
	{C}orvallis, {O}re., 1977), {P}art 2, Proc. Sympos. Pure Math., XXXIII, pp.
	3--26, Amer. Math. Soc., Providence, R.I., 1979.
	
	\bibitem[Vog93]{VoganLL}
	D.~A. Vogan, Jr., The local {L}anglands conjecture, in Representation theory of
	groups and algebras, vol. 145 of Contemp. Math., pp. 305--379, Amer. Math.
	Soc., Providence, RI, 1993.
	
	\bibitem[Zhu20]{ZhuCohLp}
	X.~Zhu, Coherent sheaves on the stack of Langlands parameters, 2020,
	arXiv:2008.02998.
	
\end{thebibliography}
\end{document}